\documentclass[11pt,leqno]{amsart}

\usepackage[all]{xy}
\usepackage{amssymb,bbm,xr,a4wide,color,cancel,nicefrac,enumitem}
\usepackage[mathscr]{euscript}
\usepackage{mathrsfs}

\usepackage[colorlinks=true,pagebackref=true]{hyperref} 
\usepackage{cleveref}

\title{Quadratic Riemann-Roch formulas}

\author{Fr\'ed\'eric D\'eglise}
\address{ENS de Lyon, UMPA, UMR 5669, 46 all\'ee d'Italie, 69364 Lyon Cedex 07, France}
\email{frederic.deglise@ens-lyon.fr}
\urladdr{http://perso.ens-lyon.fr/frederic.deglise/}
 
\author{Jean Fasel}
\address{Institut Fourier - UMR 5582, Universit\'e Grenoble-Alpes, CNRS, CS 40700, 38058 Grenoble Cedex 9, France}
\email{Jean.Fasel@univ-grenoble-alpes.fr}
\urladdr{https://www-fourier.univ-grenoble-alpes.fr/~faselj/}

\date{\today}

\newtheorem{thm}{Theorem}[subsection]
\newtheorem{prop}[thm]{Proposition}
\newtheorem{lm}[thm]{Lemma}
\newtheorem{cor}[thm]{Corollary}
\newtheorem{thmi}{Theorem}
\newtheorem{propi}[thmi]{Proposition}
\newtheorem{dfi}[thmi]{Definition}

\theoremstyle{remark}
\newtheorem{rem}[thm]{Remark}

\newtheorem{ex}[thm]{Example}

\theoremstyle{definition}
\newtheorem{df}[thm]{Definition}
\newtheorem{num}[thm]{}

\numberwithin{equation}{thm}

\newtheorem{thm*}{Theorem}

\DeclareMathOperator{\base}{\mathscr Sch_\Sigma}

\DeclareMathOperator{\SH}{SH}

\DeclareMathOperator{\Ab}{\mathscr Ab}
\DeclareMathOperator{\Cat}{\mathscr Cat}

\DeclareMathOperator{\Tors}{Tors} 
\DeclareMathOperator{\VB}{\mathcal V} 

\DeclareMathOperator{\Id}{Id} 

\newcommand{\modd}{\!-\!\operatorname{mod}} 

\renewcommand{\AA}{\mathbb A}
\newcommand{\PP}{\mathbb P}
\newcommand{\Proj}{{\mathbb P}}
\newcommand{\HP}{\mathrm{H}\mathbb P}
\newcommand{\GG}{\mathbb G_m}

\newcommand{\Gr}{\mathrm{Gr}} 
\newcommand{\HGr}{\mathrm{HGr}} 

\newcommand{\OO}{\mathcal O}

\newcommand{\cU}{\mathcal U}
\newcommand{\cV}{\mathcal V}
\newcommand{\cW}{\mathcal W}
\newcommand{\cE}{\mathcal E}

\newcommand{\BGL}{\mathrm{BGL}}
\newcommand{\BSp}{\mathrm{BSp}}

\newcommand{\un}{\mathbbm 1}
\newcommand{\E}{\mathbb E}
\newcommand{\EE}{\mathbb E}
\newcommand{\F}{\mathbb F}
\newcommand{\G}{\mathbb G}

\newcommand{\MGL}{\mathbf{MGL}}
\newcommand{\MSp}{\mathbf{MSp}}
\newcommand{\KSp}{\mathbf{KSp}}
\newcommand{\KO}{\mathbf{KO}}
\newcommand{\KGL}{\mathbf{KGL}}
\newcommand{\GWsp}{\mathbf{GW}}

\newcommand{\ZZ}{\mathbb{Z}}
\newcommand{\ZZtwo}{\mathbb{Z}[\nicefrac 1 2]}
\newcommand{\ZZe}{\mathbb{Z}_\epsilon}
\newcommand{\ZZep}{\mathbb{Z}_{\epsilon+}}
\newcommand{\ZZem}{\mathbb{Z}_{\epsilon-}}
\newcommand{\QQe}{\mathbb{Q}_\epsilon}
\newcommand{\QQ}{\mathbb{Q}}
\newcommand{\RR}{\mathbb{R}}

\newcommand{\NN}{\mathbb{N}}
\newcommand{\Orth}{\mathrm{O}}
\newcommand{\Spin}{\mathrm{Spin}}

\newcommand{\Sp}{\mathrm{Sp}}
\newcommand{\GL}{\mathrm{GL}}
\newcommand{\SL}{\mathrm{SL}}

\newcommand{\cO}{\mathcal O}
\newcommand{\cL}{\mathcal L}

\newcommand{\MForm}{\mathscr F} 

\newcommand{\FTL}{\mathcal{FTL}}

\renewcommand*{\H}{\operatorname{H}}        
\newcommand*{\HMW}{\mathbf{H}_\mathrm{MW}}
\newcommand*{\HM}{\mathbf{H}_\mathrm{M}}

\DeclareMathOperator{\Hom}{Hom}
\DeclareMathOperator{\colim}{colim}

\DeclareMathOperator{\Spec}{Spec}
\DeclareMathOperator{\rk}{rk}

\DeclareMathOperator{\GW}{GW}
\DeclareMathOperator{\W}{W}

\DeclareMathOperator{\PicSp}{Pic^{Sp}}
\newcommand{\CHt}{\widetilde{\operatorname{CH}}}
\DeclareMathOperator{\CH}{CH}

\DeclareMathOperator{\Th}{Th} 
\renewcommand{\th}{\operatorname{th}} 

\newcommand{\tw}[1]{\langle #1 \rangle} 
\newcommand{\thom}{\mathfrak{t}}
\DeclareMathOperator{\td}{td}
\DeclareMathOperator{\pur}{\mathfrak p} 

\DeclareMathOperator{\bo}{bo} 

\newcommand{\tdeg}{\tilde{\operatorname{deg}}} 

\DeclareMathOperator{\uK}{\underline K} 
\DeclareMathOperator{\uKSp}{\underline{KSp}} 
\DeclareMathOperator{\uZ}{\underline \ZZ} 

\DeclareMathOperator{\cupp} {\scriptstyle\cup\textstyle} 

\newcommand{\ux}{\underline x}
\newcommand{\uy}{\underline y}

\begin{document}


\maketitle

\setcounter{tocdepth}{2}
\tableofcontents

\section*{Introduction}

\subsubsection*{Orientation theory and Riemann-Roch formulas after Grothendieck}

Starting from an astonishing six pages paper by Quillen \cite{QuiCob}, the intrication between \emph{formal group laws}
 and homotopy theory has emerged and flourished for more than fifty years.
 The picture of chromatic homotopy theory has arisen from Ravenel's conjectures and yielded deep insight in stable homotopy groups of spheres.
 In particular, the stable homotopy category can not only be divided into the rational and $p$-local parts,
 but each $p$-local components can be subsequently divided into $v_n$-local parts reflecting
 Lazard's classification of $p$-local formal group laws by height (see \cite{BBchrom} for a historical account).
 This colorful filtration is at the heart of many developments in stable homotopy theory since then.
 The main players here are the spectrum of complex cobordism
 and its subsequent cohort of complex oriented ring spectra.

Motivic homotopy theory is driven by its analogy with its topological realization.
 In particular, complex orientation has a well-known motivic analogue, the theory of $\GL$-orientation.
 In the latter, the theory of formal group laws (\emph{aka} FGL) has been successfully developed, and algebraic cobordism $\MGL$
 is known to possess the universal formal group law by works of Morel and
 Levine.\footnote{Over fields of characteristic $0$: see \cite{LMcobord}
 and \cite{LevMGL}.}
 However, from the chromatic point of view, $\GL$-oriented theories missed an important
 part of motivic stable homotopy: the quadratic part (or more accurately the symmetric bilinear part). This part is visible
 in the computation of the $0$-th stable homotopy groups of spheres, whose sections over a field $k$
 are given by the Grothendieck-Witt group of $k$. From the $\GL$-oriented point of view,
 only the rank of the corresponding class of quadratic forms is visible.
 Another indication is that rational motivic ring spectra
 are not $\GL$-oriented in general --- whereas rational (topological) ring spectra are complex
 oriented. However, they are both oriented with respect to the algebraic group
 $\SL$ and $\Sp$, in the sense of Panin and Walter.

\bigskip

Whereas $\SL$-orientations are very close to $\GL$-orientations from the point of view of Thom classes,
 it is only with \emph{$\Sp$-orientations} that one can develop a satisfactory theory of characteristic
 (\emph{Borel}) classes, analogous to the $\GL$-oriented case (Chern classes).
 The price to pay is to restrict one's attention to symplectic vector bundles
 (\emph{i.e.} vector bundles equipped with a symplectic form, or equivalently an $\Sp$-torsor).
Once this restriction is accepted, one can, following Panin and Walter, transport
 all concepts of classical orientations theory:
 starting from Thom classes, one derives Euler classes, proves the \emph{symplectic projective bundle} theorem,
 gets \emph{higher Borel classes} following the method of Grothendieck, and derives all their properties
 thanks to a \emph{symplectic splitting principle}.\footnote{For all this,
 the reader may consult the original paper of Panin and Walter \cite{PW_MSP},
 and for more precise references, the current text: Thom classes: \Cref{df:G-orientation} and \Cref{thm:PW_Sp-oriented},
 Euler classes: \Cref{num:Sp-Euler}, symplectic projective bundle theorem \Cref{num:Sp-proj-bdl},
 Borel classes: \Cref{num:Borel_Sp}, symplectic splitting: \Cref{num:sp-splitting-principle}.}

In \cite{DF3, DFJK}, we started the study of the analog of FGL for $\Sp$-orientations,
 the theory of \emph{Formal Ternary Laws}, \emph{aka} FTL. Recall that the FGL associated
 to an oriented ring spectrum is the power series given by expressing the first Chern class of a tensor product of two line
 bundles in terms of the respective first Chern classes of the line bundles.
 In the symplectic case, the analog of line bundles are rank $2$ symplectic bundles.
 However, the tensor product of two symplectic bundles is not symplectic but orthogonal.
 According to an original idea of Walter, one has to consider the tensor product of three copies
 of the tautological symplectic bundles over the threefold product $\BSp^{\times 3}$ of the classifying space $\BSp$, which gives a rank $8$ vector bundle having
four non trivial $4$ Borel classes. In particular, FTL are given by power series in $3$ variables,
 with $4$ components --- an algebraic structure that we called a \emph{$(4,3)$-series} in \cite{CDFH}.

This gives the theory of formal ternary laws a more intricate aspect compared to its parent,
 the formal group laws.
 Nonetheless, we were able to develop substantially this theory in the aforementioned papers.
 We described a good analog of the additive (resp. multiplicative) FGL in the symplectic case,
 and proved that it arises from the Chow-Witt groups or more generally the Milnor-Witt motivic cohomology
 (resp. the higher Grothendieck-Witt groups).\footnote{See Examples \ref{ex:FTP_low_complexity}
 and \ref{ex:representableFTL} for a reminder.}
 As for formal groups, the existence of a universal FTL and its universal ring of coefficients,
 called the \emph{Walter ring}, can be easily established.
 We have set up computational tools to understand the latter in \cite{DFJK}, and derived
 an explicit description of certain sub-quotients (see \emph{loc. cit.}, A.3 and A.4).
 
\subsubsection*{Generalized orientations and GRR formulas}

In this paper, we are interested by another classical aspect of orientation theory,
 the Grothendieck-Riemann-Roch formulas. Fundamentally, these formulas describe
 the effect of a change of orientation on a fixed ring spectrum. There exists Todd
 classes which precisely describe how to go from the characteristic classes of one orientation
 to the other. Moreover, this Todd class can be computed explicitly via the isomorphism
 between the two FGLs associated with each orientations.
 In the present article,
 we extend these results to generalized orientations in the sense of Panin and Walter,
 and specialize these theoretical results to the case of symplectic orientations
 and their associated FTLs. 

To explain the context of our generalized Riemann-Roch formulas,
 let us first recall that the idea behind Panin and Walter's orientation theory
 is that certain cohomology theories possess Thom classes only for vector bundles
 with additional data, such as a symplectic form or more generally
 $G_*$-torsors for a graded algebraic sub-group $G_*$ of the
 general linear group $\GL_*$ (see \Cref{df:G-orientation} for a reminder).
 To formulate our Grothendieck-Riemann-Roch formulas, we need to
 extend these Thom classes to \emph{virtual} $G$-torsors (see \Cref{prop:virt_thom_G-orient}).
 This leads us to introduce a new, more general, definition of orientations.
 With that in mind, we introduce the notion of stable structures on vector bundles over a scheme $X$.
 Recall that Deligne has defined the category $\uK(X)$ of virtual vector bundles over $X$,
 which is a \emph{Picard groupoid} whose isomorphism classes are given by the 
 K-theory group $\mathrm{K}_0(X)$.
In this context, a stable structure on vector bundles is a collection of maps (or a so-called pseudo-functor)
$$
\sigma:\uK^\sigma(X) \rightarrow \uK(X)
$$
for various schemes $X$ with suitable properties: compatibility with pullbacks
 and with the monoidal structure (see \Cref{df:stable_structure} for more precision).
 One example is given by the functor $\sigma_G:\uK^G(X) \rightarrow \uK(X)$ from virtual $G$-torsors
 to virtual vector bundles that ``forgets'' the $G$-structure.
 Let us roughly state our main definition:
\begin{dfi}(See \Cref{df:sigma-orientation})
A $\sigma$-orientation on the ring spectrum $\E$ is 
 the data of Thom classes $\th(\tilde v) \in \E^{2r,r}(\Th(v))$ 
 associated to any object $\tilde v \in \uK^\sigma(X)$, where $v=\sigma(\tilde v)$
 is a virtual vector bundle of rank $r$ and (motivic) Thom space $\Th(v)$.
 These Thom classes are required to be compatible with isomorphisms, pullbacks
 and the monoidal structure on $\uK^\sigma(X)$.
\end{dfi}
As in the classical case, the properties listed in the preceding definition imply
 that $\thom(\tilde v)$ is a basis of the $\E^{**}(X)$-module
 $\E^{**}(\Th(v))$ (as expected). The reader will be relieved 
 to know that $\sigma_G$-orientations are in bijection with $G$-orientations
 in the sense of Panin and Walter (\Cref{prop:stable_G-orientations}).
 One easily deduces from the above definition
 a new family of characteristic (or fundamental) classes.
 
\begin{propi}(See \Cref{def:orientedfundamentalclass})\label{propi:fdl_classes}
Let $\E$ be a $\sigma$-oriented ring spectrum, and let $f:X \rightarrow S$ be a smoothable lci morphism. A \emph{$\sigma$-orientation} of $f$ is a class $\tilde \tau_f \in \uK^\sigma(X)$ such
 that $\sigma(\tilde \tau_f)$ is isomorphic to the image $\tau_f$ of the cotangent complex
 $\mathcal L_f$ of $f$ in $\uK(X)$. One associates to the pair $(f,\tilde \tau_f)$ a canonical fundamental class
 with coefficients  in $\E$:
$$
\tilde \eta_f^\E \in \E_{2d,d}(X/S)
$$
where $\E_{n,i}(X/S)=\Hom_{\SH(S)}(\un_S(i)[n],f^!\un_S)$
 is the bivariant theory associated with $\E$ (see \Cref{sec:bivariant} for more information).
\end{propi}
These $\sigma$-oriented fundamental classes satisfy the usual formalism of their parent
 the motivic fundamental classes: compatibility with composition,
 transverse pullback formula, excess of intersection (see \cite{DJK}).
 They also fulfill the formalism of Fulton and MacPherson \cite{FMP}.
 Note in particular that they induce various Gysin morphisms.
 For example, if $f$ is in addition proper, one gets a pushforward in cohomology:
$$
f^\E_!:\E^{n,i}(X) \rightarrow \E^{n+2d,i+d}(S), a \mapsto f_!(a.\tilde \eta_f^\E)
$$
where one uses the action of cohomology and the axiomatic pushforward on bivariant theory
 (see also e.g. \cite[Def. 3.3.2]{Deg16}).
Additionally, they satisfy our sought-for Grothendieck-Riemann-Roch formula,
 in the style of \cite[1.4]{FMP}.
\begin{thmi}(\Cref{prop:computeToddclass} and \Cref{thm:GRR})\label{thmi:GRR}
Let $\psi:\E \rightarrow \F$ be a morphism of ring spectra.
 Assume further that $\E$ and $\F$ admit $\sigma$-orientations. Then there exists a unique \emph{Todd class}
\[
\td_\psi\colon \mathrm{K}_0^{\sigma}(X) \rightarrow \F^{00}(X)^\times
\]
which is a morphism of abelian groups natural in $X$ such that
 the following relation holds
$
\thom(\tilde v,\F)=\td_\psi(\tilde v)\cdot\psi_*(\thom(\tilde v,\E)).
$

Moreover, for any $\sigma$-oriented morphism $(f:X \rightarrow S,\tilde \tau_f)$,
 one gets the following Grothendieck-Riemann-Roch formula:
$$
\psi_*(\tilde\eta_f^\E)=\td_\psi(\tilde \tau_f)\cdot\tilde\eta_f^\F.
$$
\end{thmi}
As underlined by Fulton and MacPherson, one deduces from the above formula
 the usual formulations of Grothendieck-Riemann-Roch formulas,
 involving Gysin morphisms. Namely, when $f$ is in addition proper, the formula reads
$$
\psi(f^\E_!(a))=f^\F_!(\td_\psi(\tilde \tau_f).\psi(a)).
$$
One also deduces analogous formulas in homology
 and, without the properness assumptions, formulas in cohomology
 with compact support or bivariant theories. The reader
 can find details in the statement of \Cref{thm:GRR}.
 The problem with this abstract formula is to be able to compute
 the relevant Todd class. This is one of the motivations to restrict our attention
 to the symplectic case.
 
\subsubsection*{Symplectic orientations and formal ternary laws}

Symplectic orientations are very close to $\GL$-orientations.
 Firstly, there is an appropriate variant of the projective bundle theorem,
 for the so-called symplectic projective spaces $\HP^n_S$
 (built out of the classifying spaces $\BSp_{2n}$, see \Cref{num:sp-grassmanian}).
 This allowed Panin and Walter to characterize $\Sp$-orientations of a ring spectrum $\E$
 as the datum of a class $b \in \tilde \E^{4,2}(\HP^\infty_S)$ in reduced cohomology
 which restricts to the tautological class on $\HP^1_S\simeq \un(2)[4]$
 (see \Cref{thm:PW_Sp-oriented} for a reminder).

Our first task is to relate the theory of Thom classes
 associated to symplectic orientations with the theory of virtual Thom classes
 mentioned previously.
 For this, we show that these Thom classes not only extend to 
 virtual classes of symplectic bundles, but also factors modulo the so-called metabolic relation.
 More precisely, we introduce a stable structure
 $f_\Sp:\uKSp(X) \rightarrow \uK(X)$ whose source is the associated homotopy groupoid of the $\infty$-category $\mathrm{KSp}(X)$ and we show that $\Sp$-orientations
 on a ring spectrum $\E$ are in bijection with $f_\Sp$-orientations in the
 sense explained above (see \Cref{prop:sp-or&KSp-or}).

Our second task is to improve the properties of symplectic Thom classes
 by establishing the following formula:
\begin{thmi}(see \Cref{thm:epsilon} and its corollary)\label{thmi:epsilon}
Let $\E$ be an $\Sp$-oriented ring spectrum with Thom classes $\thom$.
 Then for any symplectic vector bundle $(\cU,\psi)$ over a scheme $X$,
 any unit $\lambda \in \cO_X(X)^\times$ with associated class
 $\langle \lambda \rangle \in [\un_X,\Sigma^\infty \GG]_{\SH(X)}$,
 the following relations hold:
\begin{align*}
\thom(\cU,\lambda.\psi)&=\langle u \rangle.\thom(\cU,\psi), \\
b_i(\cU,\lambda.\psi)&=\langle \lambda^i \rangle.b_i(\cU,\psi), i \geq 0.
\end{align*}
\end{thmi}
The proof of this formula is surprisingly involved. It was known before only
 for $\SL$-oriented theories (see \cite[Prop. 2.2.8]{DF3}).
 It is nevertheless central for the theory of FTL,
 as it implies the so-called \emph{$\epsilon$-linearity axiom} of the FTL
 associated to a symplectically oriented spectrum $\E$
 (see \Cref{df:FTL} and \Cref{num:SP-orient->FTL}).
 Such FTL are called \emph{representable}.

Once all these preliminaries are settled,
 the link between FTL and $\Sp$-orientations takes the same
 form as for $\GL$-orientations: given an $\Sp$-oriented
 ring spectrum with associated FTL $F_t$, choosing another
 FTL amounts to choose a strict automorphism of $F_t$
 (see \Cref{prop:sp-orient&FTL}).
 This result contains the key to compute Todd classes
 as we will see later in this introduction.
 One also deduces the following
 result, that confirms the analogy between FTL and FGL.
\begin{thmi}(see \Cref{thm:log_FTL}) \label{thmi:log_FTL}
Any representable rational formal ternary law $F_t$ admits a logarithm,
 \emph{i.e.} a strict isomorphism $\log_{F_t}$ with the additive FTL.
 Moreover, $\log_{F_t}$ is unique among such strict isomorphisms.
\end{thmi}
We expect that the representability assumption is unnecessary.
 In fact, we believe that much more properties of FTL could be developed
 in analogy with FGL. This includes a general chromatic study
 of motivic homotopy theory.
 On the other hand, the properties of the Walter ring classifying FTL
 seems richer than its analog the Lazard ring, due to
 the presence of torsion phenomena.
 
\subsubsection*{The quadratic Riemann-Roch formula}

Our main application is a quadratic enrichment
 of the usual Grothendieck-Riemman-Roch formula.
 In the former, the Borel character defined in \cite{DF3} plays the role
 of the Chern character.

Recall that, for a field of characteristic different from $2$,\footnote{this
 assumption can probably be removed using \cite{Calmes24}}
 and any smooth $k$-scheme, the (total) Borel character is an isomorphism 
 of ring spectra over $k$
$$
\bo_t:\KO_\QQ \rightarrow \bigoplus_{i \in \ZZ} \big(\HMW\QQ(4i)[8i] \oplus \HM\QQ(4i+2)[8i+4]\big)
$$
where $\KO_\QQ$, $\HMW\QQ$, $\HM\QQ$ are respectively the ring spectra representing
 higher-Grothendieck-Witt groups (aka Hermitian $K$-theory), Milnor-Witt motivic cohomology and motivic cohomology over $k$,
 all with rational coefficients. 
 Plainly, the \emph{general quadratic Grothendieck-Riemann-Roch theorem(s)}
 are the formulas that one obtains by applying \Cref{thmi:GRR}
 to this morphism of ring spectra, each equipped with their canonical
 symplectic orientation.
 The interesting point is that this particular case is amenable to calculations.
 Our first result is a computation of the associated Todd class for a smooth $k$-scheme $X$
$$
\tilde \td:=\td_{\bo_t}:\KSp_0(X) \rightarrow  \CH{}^{4*}(X)\oplus \CHt{}^{4*+2}(X)
$$
using the symplectic splitting principle (see Propositions \ref{prop:Toddforgetful},
 \ref{prop:Toddquadratic}).

\bigskip

We specialize the general quadratic GRR-formulas attached to $\bo_t$
 to the \emph{Hirzebruch-Riemann-Roch case}.
 To describe it, we need a little more notation.
 Let $X$ be a smooth proper $k$-scheme of even dimension $d=2n$,
 which is \emph{symplectically oriented} by a class $\tilde \tau_X \in \KSp_0(X)$,
 as in \Cref{propi:fdl_classes}. Note that this induces
 an isomorphism $f:\omega_X \simeq \OO_X$ where $\omega_X=\det(\Omega_X)$ is the canonical
 bundle of $X/k$.

We first note that, given the $(8,4)$-periodicity of the ring spectra $\GW$,
 the Borel character induces in particular isomorphisms
\begin{align*}
\bo_t^{0,0}:\KO_0(X)_\QQ &\rightarrow \CHt{}^{4*}(X)_\QQ \oplus \CH^{4*+2}(X)_\QQ \\
\bo_t^{4,2}:\KSp_0(X)_\QQ &\rightarrow \CHt{}^{4*+2}(X)_\QQ \oplus \CH^{4*}(X)_\QQ
\end{align*}
where the first $\KO_0$ and $\KSp_0$ are respectively the orthogonal
 and symplectic K-theory. One will retain that $\bo_t$ is defined on stable metabolic
 classes of symmetric and symplectic vector bundles.

We next consider $(V,\varphi)$ either a symmetric bundle if $d=0\pmod 4$ or a symplectic bundle if $d=2\pmod 4$.
The orientation $f$ of $X/k$ allows us to define a \emph{quadratic Euler characteristic}
 (\Cref{df:quad-Euler-char}) as the element in $\GW(k)$:
$$
\tilde \chi(X,f;V,\varphi)=
[\H^n(X,V),\varphi_n]+(-1)^{\lfloor d/2 \rfloor}\sum_{i=0}^{n-1}(-1)^{i}\mathrm{dim}_k\H^i(X,V)\cdot h
$$
where $h$ is the class of the hyperbolic form, and we have considered the composite isomorphism:
$$
\varphi_n:\H^n(X,V) \xrightarrow \sim \H^n(X,(V)^\vee\otimes\omega_{X/k})
 \xrightarrow \sim \H^{n}(X,V)^*
$$
where the first map is induced by $f$ and the second one by Serre duality.
\begin{thmi}[quadratic HRR formula](\Cref{thm:HRR}) \label{thmi:HRR}
Consider the above notation. Then the following formula holds in $\GW(k)$:
$$
\tilde \chi(X,f;V,\varphi)=\tdeg_{\tilde \tau_f}\big(\tilde \td(\tilde \tau_X)\cdot \bo_t(V,\varphi)\big)
$$
where $\tdeg_{\tilde \tau_f}:\CHt{}^d(X) \stackrel{\tilde \tau_X}\simeq \CHt_0(X,\omega_X) \xrightarrow{p_*} \GW(k)$
 is the quadratic degree map associated with the orientation $f$ of $X/k$.
\end{thmi}
We compute explicitly the above formula in the case of K3-surfaces (see \Cref{ex:HRR-K3}).

\subsubsection{Plan of the article}
This work can be roughly divided into two parts.
 In the first one, comprising Sections 1 and 2,
 we establish a framework of generalized orientations
 (or Thom classes attached to vector bundles with additional structures)
 that allows one to define fundamental classes, Todd classes and prove several
 Riemann-Roch formulas.
 In the second part, comprising Sections 3 and 4,
 we specialize the above discussion to the case of symplectic orientations,
 whose specificity is to be tied with the algebraic notion of formal ternary laws.
 The reader which is only interested in the symplectic case, can
 start reading sections 3 and 4 independently.

In more details, Section 1 is mainly concerned with a theory
 that extends Panin-Walter's orientation theory. The main objective is to define 
 a good notion of \emph{virtual Thom classes},
 which are associated with ``virtual vector bundles with additional structures'',
 which are rigorously axiomatized in \Cref{df:stable_structure}.
 We also remind  in Section 1.1 some material about bivariant theories in $\AA^1$-homotopy.
 Section 2.1 is then devoted to the application of the general theory of virtual Thom classes
 to define fundamental classes in the sense of Fulton and MacPherson,
 associated with the bivariant theories discussed in Section 1.1.
Section 2.2 finally presents the various Riemann-Roch formula that one
 can obtain out of an arbitrary morphism of ring spectra which are suitably
 \emph{oriented}.

Section 3 is more involved.
 Let us first describe the sections that can be read independently of sections 1 and 2.
 We first recall in 3.1 the algebraic theory of formal ternary laws,
 and in 3.2 Panin-Walter's notion of symplectic orientations on a ring spectrum.
 Section 3.4 is devoted to prove \Cref{thm:epsilon} (stated above as \Cref{thmi:epsilon}),
 which is crucial to prove that each $\Sp$-oriented ring spectrum
 has a canonically associated FTL. This association is studied in detail in 3.5.
 We derive some interesting facts about FTL such as \Cref{thm:log_FTL}
 (stated as \Cref{thmi:log_FTL} above).
 Section 3 is connected with the two previous one by 3.3,
 where we show that Panin-Walter's Thom classes can be extended as Thom classes
 of virtual metabolic symplectic bundles. The final section 3.6 explains
 how to compute Todd classes in the symplectic case
 (by perfect analogy with the $\GL$-orientation case).

The final section is devoted to computations, determining respectively the Todd classes associated with Adams operations (1.1)
 and with the Borel character (1.2). We conclude by a proof of the
 quadratic HRR-formula stated as \Cref{thmi:HRR} above.

\section*{Notations}

We will fix throughout the paper a base scheme $\Sigma$
 (eg: $\Spec(\ZZ)$, $\Spec(\ZZtwo)$, $\Spec(k)$ for a field $k$).
 All schemes are assumed to be quasi-compact and quasi-separated (qcqs) $\Sigma$-schemes;  we denote by $\base$ the category whose objects are such schemes. 

Given a scheme $X$, we let $\VB(X)$ be the category of vector bundles over $X$,
 that we view as Zariski (or Nisnevich) locally trivial $\GL_*$-torsors.

We let $\ZZe=\GW(\ZZ)$ be the Grothendieck-Witt ring of $\ZZ$:
 in other words, $\ZZe=\ZZ[\epsilon]/(\epsilon^2-1)$,
 $\epsilon$ corresponding to the opposite of the non-degenerate symmetric bilinear form $(x,y) \mapsto -xy$
 (see eg. \cite[Rem. 2.1.7]{Deg18}).

We use the terminology ``ring spectrum'' over a scheme $S$ for a commutative monoid object of the (homotopy category) $\SH(S)$
 - we will not need $E_\infty$-structures. Without further precision, a (ring) spectrum is defined
 over the fixed base scheme $\Sigma$.
 Given a scheme $S$, we let $[-,-]_S$ denote the bifunctor of morphisms in $\SH(S)$,
 and usually do not mention the indexing $S$ when it is clear from the context.

Given a vector bundle $V$ over a scheme $X$, we let $\Th(V)=\Sigma^\infty (V/V-X)$ be the Thom space
 associated with $V$ in the stable homotopy category $\SH(X)$.
 Recall that this construction extends canonically to virtual vector bundles $v$ over $X$,
 giving an object $\Th(v)$ in $\SH(X)$, such that $\Th(v+w)= \Th(v) \otimes \Th(w)$
 (see \cite[4.1]{Riou}).
 The rank of a virtual bundle over a scheme $X$ is considered as a $\ZZ$-valued Zariski locally
 constant function on $X$; we denote by $\uZ(X)$ the abelian groups of such functions.
 Similarly, given  $r \in \uZ(X)$, we let $\tw r$ be the constant virtual bundle
 with rank $r(x)$ on the connected component corresponding to a generic point $x$ of $X$.

For an lci morphism $f:X \rightarrow S$, we will denote by $\tau_f$ the virtual vector bundle
 on $X$ associated with its cotangent complex $\mathcal L_f$.


\subsubsection*{Acknowledgments}

The authors thank Adeel Khan and Fangzhou Jin for first exchanges
 that initially shaped this work.
They are also grateful towards Adrien Dubouloz and Michel Brion for several insightful conversations,
 and especially around the proof of Theorem \ref{thm:epsilon}. The second author also wishes to thank Marc Levine for useful discussions about K3-surfaces.
 The authors were supported by the ANG project HQDIAG, grant number ANR-21-CE40-0015.

\section{Orientation theory and fundamental classes}

\subsection{Recollections on bivariant theories}\label{sec:bivariant}

\begin{num}
In this section, we recall the bivariant formalism that automatically comes out of
the six functors formalism for ring spectra, as explained in \cite{Deg16}.\footnote{Recall
that bivariant theories were introduced by Fulton and MacPherson in \cite{FMP}. We refer the 
interested reader to the introduction of \cite{Deg16} for more historical remarks.}

Let $\EE$ be a spectrum (resp. ring spectrum) over $\Sigma$. For any ($\Sigma$-)scheme $S$
with structural map $p:S \rightarrow \Sigma$, we put $\EE_S=p^*\EE$  (resp. equipped with its ring structure)
--- this produces an \emph{absolute (resp. ring) spectrum} over $\base$
in the sense of \cite[Def. 1.1.1]{Deg16} (for $\mathscr T=\SH$).

Recall from \cite{LZ} and \cite[Th. 2.34]{Khan} that the existence of the exceptional functors
$(f_!,f^!)$ on the motivic category $\SH$ has been extended to all finite type morphisms,
instead of separated morphisms of finite type as in \cite{CD3}.
Therefore, one can slightly extend the definitions of \cite[Sec. 1.2]{Deg16}.\footnote{Unlike in \cite[Def. 1.2.2]{Deg16},
we will not call this the \emph{Borel-Moore homology} of $X/S$ with coefficients in $\E$.
This terminology appears a bit unfortunate as $f^!(\E_S)$ is not a dualizing object of $\SH(X)$ in general.}
\end{num}
\begin{df}\label{df:bivariant}
Let $\E$ be a spectrum over $\Sigma$.
We define the \emph{bivariant theory} with coefficients in $\E$ as follows.
Let $f:X \rightarrow S$ be a morphism of finite type and  $(n,m)$ be a couple of integers,
or more generally of Zariski locally constant functions $X \rightarrow \ZZ$.
The bivariant group of $X/S$ in degree $(n,m)$ and coefficients in $\E$ is:
$$
\E_{n,m}(X/S)=[\un_X(m)[n],f^!\E_S].
$$
The natural cohomology associated with this bivariant theory is for any scheme $X$:
\begin{equation}\label{eq:biv->coh}
\E^{n,m}(X)=\E_{-n,-m}(X/X)=[\un_X,\E_X(m)[n]].
\end{equation}
\end{df}

\begin{rem} \emph{Twisting conventions}. There are three possible, and useful, choices of indexes for a bivariant (and therefore also for cohomology/homology theories)
associated with (ring) spectra. Apart from the above choice, one can also introduce:
\begin{enumerate}
\item (See \cite[\textsection 1]{DJK}) For any integer $n \in \ZZ$ and any virtual vector bundle $v$ over $X$,
$$
\E_{n}(X/S,v)=[\Th(v)[n],f^!\E_S].
$$
\item For an integer $n \in \ZZ$ and a graded line bundle $(\cL,r)$ over $X$,
$$
\E_{n}(X/S,\cL,r)=[\Th(\cL)(r-1)[n+2r-2],f^!\E_S].
$$
\end{enumerate}
The indexing (1) is the most general one, and covers all the other cases according to
the following rules:
\begin{align*}
\E_{n,m}(X/S)&=\E_{n-2m}(X/S,\tw m), \\
\E_{n}(X/S,\cL,r)&=\E_n(X/S,[\cL] \oplus \tw{r-1}),
\end{align*}
where $[\cL]$ (resp. $\tw n$) denotes the virtual vector bundles associated with $\cL$
(resp. $\AA^n$).
When dealing with $\SL$-oriented theories,
convention (2) is the natural one (see \cite[\textsection 7]{DFJK}).
When dealing with ($\mathrm{GL}$ or $\Sp$)-oriented ring spectra, convention \eqref{eq:biv->coh} is the most relevant,
because of Thom isomorphisms. 

When following the indexing convention (1), note that the associated cohomology theory is of the form
\[
\E^{n}(X,v)=\E_{-n}(X/X,-v)\simeq [\un_X,\E_X\otimes \Th(v)[n]].
\] 
\end{rem}

Thanks to the motivic six functors formalism, bivariant theories enjoy a rich structure
that we will briefly survey for completeness. We start by the functorial properties.\footnote{We state
 these properties taking care about possible twists even though we will be mainly interested in trivial twists
 when dealing with $G$-oriented spectra.}
\begin{prop}\label{prop:biv_functorial}
Consider the assumptions of the previous definition.
\begin{enumerate}
\item \emph{Base change map}. For any cartesian square
$$
\xymatrix@=10pt{
Y\ar[r]^-{f'}\ar[d]\ar@{}|\Delta[rd] & X\ar[d] \\
T\ar_-f[r] & S,
}
$$
any virtual vector bundle $v$ over $X$ and any integer $n\in \ZZ$
there exists a morphism $\Delta^*:\E_{n}(X/S,v) \rightarrow \E_{n}(Y/T,(f')^*v)$, functorial
with respect to horizontal composition of cartesian squares. We sometimes write abusively: $f^*=\Delta^*$.
\item \emph{Proper covariance}. For any proper morphism $p:Y \rightarrow X$ of $S$-schemes, any virtual vector bundle $v$ on $X$ and any integer $n\in\ZZ$
there exists a morphism $p_*:\E_{n}(Y/S,p^*v) \rightarrow \E_{n}(X/S,v)$,
corresponding to a covariant functor on the category of $S$-schemes (with proper morphisms as morphisms).
\item \emph{\'Etale contravariance}. For any étale morphism $f:Y \rightarrow X$ of $S$-schemes, any virtual vector bundle $v$ on $X$ and any integer $n\in\ZZ$
there exists a morphism $f^*:\E_{n}(X/S,v) \rightarrow \E_{n}(Y/S,f^*v)$,
corresponding to a contravariant functor on the category of $S$-schemes and étale morphisms.
\end{enumerate}
These structural maps satisfy the following properties:
\begin{enumerate}[resume]
\item \emph{Homotopy invariance}. For any vector bundle $p:E \rightarrow S$, any $S$-scheme $X$ and any virtual vector bundle $v$ over $X$
the base change map $p^*:\E_{n}(X/S,v) \rightarrow \E_{n}(X \times_S E/E,(p')^*v)$ is an isomorphism.
\item \emph{\'Etale invariance}. For any étale morphism $f:T \rightarrow S$,
there exists an isomorphism $f^\sharp:\E_{n}(X/S,v) \rightarrow \E_{n}(X/T,v)$ which is natural
with respect to base change maps, proper covariance and étale contravariance.
\item \emph{Base change formulas}. Base change maps are compatible with proper covariance and étale covariance.
Moreover, for any cartesian square of $S$-schemes
$$
\xymatrix@=10pt{
W\ar^g[r]\ar_q[d] & Y\ar^p[d] \\
V\ar_f[r] & X
}
$$
such that $p$ is proper and $f$ is étale, one has: $f^*p_*=q_*g^*$.
\item \emph{Localization}. For any closed immersion $i:Z \rightarrow X$ of $S$-schemes, 
with complementary open immersion $j:U \rightarrow X$, and any virtual vector bundle $v$ on $X$ there exists a long exact sequence of the form:
$$
\E_{n}(Z/S,i^*v) \xrightarrow{i_*}
\E_{n}(X/S,v) \xrightarrow{j^*}
\E_{n}(U/S,j^*v) \xrightarrow{\partial}
\E_{n-1}(Z/S,i^*v)
$$
which is natural with respect to base change, proper covariance and  \'etale contravariance.
\item \emph{Nisnevich descent}. For any Nisnevich distinguished square\footnote{i.e. the square is cartesian, $p$ \'etale, $j$ open immersion,
and $p$ induces an isomorphism $(V-W)_{red} \rightarrow (X-U)_{red}$} of $S$-schemes 
$$
\xymatrix@=10pt{
W\ar^-l[r]\ar_-q[d] & V\ar^-p[d] \\
U\ar_-j[r] & X
}
$$
and any virtual vector bundle $v$ on $X$ there exists a long exact sequence:
$$
\E_{n}(X/S,v)
\xrightarrow{j^*+p^*} \E_{n}(U/S,j^*v) \oplus \E_{n}(V/S,p^*v)
\xrightarrow{q^*-l^*} \E_{n}(W/S,q^*j^*v)
\rightarrow \E_{n-1}(X/S,v),
$$
functorial with respect to base change maps, proper covariance and étale contravariance.
\item \emph{cdh-descent}. For any cdh-distinguished square\footnote{i.e. the square is cartesian, $p$ proper, $i$ closed immersion,
and $p$ is an isomorphism over $(X-Z)$} of $S$-schemes 
$$
\xymatrix@=10pt{
T\ar_-q[d]\ar^-k[r] & Y\ar^-p[d] \\
Z\ar_-i[r] & X
}
$$
and any virtual vector bundle $v$ on $X$ there exists a long exact sequence
\[
\E_{n}(T/S,q^*i^*v)
\xrightarrow{q_*+k_*} \E_{n}(Z/S,i^*v) \oplus \E_{n}(Y/S,p^*v)
\xrightarrow{i_*-p_*} \E_{n}(X/S,v)
\rightarrow \E_{n-1}(T/S,q^*i^*v)
\]
which is again functorial with respect to base change maps, proper covariance and étale contravariance.
\end{enumerate}
\end{prop}

These properties follow from the Grothendieck six functors formalism,
as stated in \cite[Introduction]{CD3}. This is a good exercise for the interested reader
(see \cite{Deg16} after Prop. 1.2.4 for hints).

\begin{rem}\label{rem:biv->cohsup}
As can be guessed from the localization exact sequence, the bivariant theory can be thought of as an extension of
cohomology with support. Indeed, when considering a closed immersion $i:Z \rightarrow X$ and a virtual vector bundle $v$ on $X$, if we view $Z$ as
an $X$-scheme (of finite type!), then one has:
$$
\E^{n}_Z(X,v)=\E_{-n}(Z/X,-i^*v),
$$
where one can define the left hand-side, cohomology of $X$ with support in $Z$, as:
$$
\E^{n}_Z(X,v)=[\Sigma^\infty X/(X-Z),\E_X\otimes \Th(v)[n]].
$$
So bivariant theory of closed immersions is really cohomology with support (see also \cite[Rem. 1.2.5]{Deg16}).
This can shed light on the bivariant formalism.
\end{rem}

\begin{num}\textit{Theories with proper support}.\label{num:biv_support}
It is possible to introduce
two variants of the bivariant theory for a (ring) spectrum $\E$ over $\base$,
a morphism $f:X \rightarrow S$ of finite type, integers $(n,m) \in \ZZ^2$
and a virtual vector bundle $v$ over $X$
(\cite[Def. 2.2.1]{DJK}):
\begin{itemize}
\item Bivariant theory with proper support:
\begin{align*}
\E_{n,m}^c(X/S)&=[\un_S(m)[n],f_!(f^!\E_S)] \\
\E_{n}^c(X/S,v)&=[\un_S[n],f_!(f^!\E_S \otimes \Th(-v))]
\end{align*}
\item Cohomology with proper support:
\begin{align*}
\E^{n,m}_c(X/S)&=[\un_S,f_!(\E_X)(m)[n]] \\
\E_{n}^c(X/S,v)&=[\un_S,f_!(\E_X \otimes \Th(v))[n]].
\end{align*}
\end{itemize}
Using again the six functors formalism, one deduces that they satisfy properties similar
to that of (plain) bivariant theory:
bivariant theory (resp. cohomology) with proper support is covariant (resp. contravariant) 
in $X$ with respect to any morphism (resp. proper morphisms), contravariant (resp. covariant)
in $X$ with respect to \'etale morphisms. The base change map (point (1)) exists for any cartesian square
in cohomology with proper support,
but only when $f$ is smooth for bivariant homology with proper support.

The properties (4) to (9) hold true for these two theories, with the following changes:
\begin{itemize}
\item for (5), one needs to assume that $f$ is \'etale and finite;
\item for (7) and cohomology with proper support, one needs to consider $j_*$ and $i^*$
instead of $i_*$ and $j^*$;
\item for (9) and cohomology with proper support,
the exact sequence  must be formulated in terms of pullbacks.
\end{itemize}
\end{num}

We next explain product structures on (plain) bivariant theories coming from ring spectra. We state the definition and properties for bivariant theories under the convention \eqref{eq:biv->coh}. The reader may consult \cite[\S 2.2.7 (4)]{DJK} for the most general indexing convention.

\begin{prop}\label{prop:biv_product}
Let $\E$ be a ring spectrum over $\Sigma$ and consider the assumptions of the previous definition.

For any composable morphisms of finite type $Y \xrightarrow g X \xrightarrow f S$,
there exists a product of the form
$$
\E_{n,m}(Y/X)\otimes \E_{s,t}(X/S) \rightarrow \E_{n+s,m+t}(Y/S), (y,x) \mapsto y.x
$$
which satisfies the following properties:
\begin{enumerate}
\item \emph{Associativity}.-- Consider four morphisms of finite type $Z/Y/X/S$,
and a triple $(z,y,x) \in  \E_{**}(Z/Y) \times \E_{**}(Y/X) \times \E_{**}(X/S)$, we have:
$$(z.y).x=z.(y.x).$$
\item \emph{Compatibility with pullbacks}.-- Consider morphisms of finite type $Y/X/S$ and any morphism $f:S' \rightarrow S$.
We let $g:X' \rightarrow X$ be the pullback of $f$ along $X/S$.
Then for any pair $(y,x) \in \E_{**}(Y/X) \times \E_{**}(X/S)$, the following formula holds:
$$
f^*(y.x)=g^*(y).f^*(x).
$$
\item \emph{(First) projection formula}.-- Consider morphisms $Z \xrightarrow f Y \rightarrow X \rightarrow S$ of finite type
such that $f$ is proper. Then for any pair $(z,y) \in  \E_{**}(Z/Y) \times \E_{**}(Y/X)$, one has:
$$
f_*(z.y)=f_*(z).y.
$$
\item \emph{(Second) projection formula}.-- Given a cartesian square in the category
of $S$-schemes of finite type:
$$
\xymatrix@=10pt{
Y'\ar^g[r]\ar[d]\ar@{}|\Delta[rd] & Y\ar[d] \\
X'\ar_f[r] & X,
}
$$
such that $f$ is proper, for any pair $(y,x') \in \E_{**}(Y/X) \times \E_{**}(X'/S)$, one has: 
$$
g_*(\Delta^*(y).x')=y.f_*(x').
$$
\end{enumerate}
\end{prop}
For the construction of the product, we start with maps
$$
y:\un_Y(m)[n] \rightarrow g^!(\E_X), x:\un_X(t)[s] \rightarrow f^!(\E_S)
$$
and then define $y.x$ as the following composite:
\begin{align*}
\un_Y(m+t)[n+s] &\xrightarrow{y(t)[s]} g^!(\E_X)(t)[s] \simeq g^!(\E_X \otimes \un_X(t)[s]) \\
& \xrightarrow{g^!(Id \otimes x)} g^!\big(\E_X \otimes f^!(\E_S)\big)=g^!\big(f^*(\E_S) \otimes f^!(\E_S)\big) \\
& \xrightarrow{Ex^{!*}_\otimes} g^!\big(f^!(\E_S \otimes \E_S)\big) \simeq (fg)^!(\E_S \otimes \E_S) \\
& \xrightarrow{(fg)^!(\mu)} (fg)^!(\E_S).
\end{align*}
Here, the isomorphism on the first line holds as Tate twists are invertible.
Second, the map labeled $Ex_\otimes^{!*}$ stand for the map obtained by adjunction
from canonical map:
$$
f_!\big(f^*(\E_S) \otimes f^!(\E_S)\big) \xrightarrow \sim \E_S \otimes f!f^!(\E_S) \xrightarrow{ad(f_!,f^!)} \E_S \otimes \E_S
$$
where the first map is the projection formula isomorphism and the second one is the unit map of the indicated adjunction.
Finally, $\mu$ is the product map of the ring spectrum $\E$.

The stated properties now follows from the six functors formalism,
again left as an exercise to the reader (hints can be found in the proof of \cite[Prop. 1.2.10]{Deg12}).

\begin{rem}
\begin{enumerate}[leftmargin=*]
\item The above product corresponds to the usual cup-product on cohomology, via formula \eqref{eq:biv->coh}.
More generally, it also induces the cup-product on cohomology with support (recall \Cref{rem:biv->cohsup}).
Given a scheme $X$ and closed subschemes $Z$ and $T$ of $X$, one can indeed recover the usual cup product as follows:
\begin{align*}
\E^{n,m}_T(X) \otimes \E^{s,t}_Z(X)=\E_{-n,-m}(T/X) & \otimes \E_{-s,-t}(Z/X)
\rightarrow \E_{-n,-m}(T/X) \otimes \E_{-s,-t}(Z \times_X T/T) \\
& \rightarrow \E_{-n-s,-m-t}(Z \times_X T/X)=\E^{n+s,m+t}_{Z \times_X T}(X) \\
(\tau,\zeta) &\mapsto \tau.\Delta^*(\zeta)=:\tau \cupp \zeta
\end{align*}
where $\Delta$ stands for the obvious cartesian square.

It is the opportunity to recall that cup-products in $\AA^1$-homotopy are not merely graded commutative, because of the double indexing.
Rather, using the above notation, one gets:
$$
\tau \cupp \zeta=(-1)^{(n-m)(s-t)}\epsilon^{mt}.\zeta \cupp \tau.
$$
This follows from the fact the permutation on the sphere $S^1$ is $(-1)$
and that on the sphere $\GG=\un(1)[1]$ is $\epsilon$ (see e.g. \cite[Lem. 6.1.1]{MorTri}).
\item Note also that given $X/S$,
one gets as a particular case both a left and right action of cohomology on bivariant theory:
\begin{align*}
\E^{s,t}(X) \otimes \E_{n,m}(X/S) & \rightarrow \E_{n-s,m-t}(X/S), \\
\E_{n,m}(X/S) \otimes \E^{s,t}(S) & \rightarrow \E_{n-s,m-t}(X/S).
\end{align*}
\end{enumerate}
\end{rem}

\begin{num}\textit{Fundamental classes}.\label{num:fdl_class}
Let $f:X \rightarrow S$ be a smoothable lci morphism, with virtual tangent bundle $\tau_f$.
According to \cite[Th. 3.3.2, Def. 4.1.3]{DJK}, one associates to $f$ its \emph{fundamental
class}:
$$
\eta_f^\E \in \E_{0,0}(X/S,\tau_f).
$$
By definition, this is the map: $\Th(\tau_f) \rightarrow f^!(\E_S)$ in $\SH(X)$.
Working with $\E$-modules, one obtains by adjunction an $\E$-linear morphism:
\begin{equation}\label{eq:fdl_biv}
\eta_f^\E:\E_X \otimes \Th(\tau_f) \rightarrow f^!(\E_S).
\end{equation}
Note also that this map is simply obtained by evaluating the purity
natural transformation $\pur_f$ of \cite[4.3.1]{DJK} at $\E$.
In particular, if $f$ is smooth, $\eta_f^\E$ is an isomorphism (\cite[Rem. 4.3.8]{DJK}).
Recall more generally the following definition (\cite[Def. 4.3.7]{DJK}):
\end{num}
\begin{df}\label{df:f-pure}
Let $f:X \rightarrow S$ be a smoothable lci map,
 and $\E$ be a ring spectrum over $X$.
 One says that $\E$ is \emph{$f$-pure} if the natural transformation $\eta_f^\E$ defined above is an isomorphism.
\end{df}

\begin{rem}
A spectrum $\E$ over $\ZZ$ is said to be \emph{absolutely pure} (\cite[4.3.11]{DJK}, \cite[1.3.2]{Deg12})
 if for any smoothable morphism $f:X \rightarrow S$ between regular schemes, $\E_S$ is $f$-pure.
 It is known that the K-theory spectrum, the rational motivic Eilenberg MacLane spectrum (\cite[14.4.1]{CD3}),
 the rational sphere spectrum (\cite{DFJK}), and the rational cobordism spectrum are absolute pure.
 It is conjectured that the sphere spectrum, and therefore, the algebraic cobordism spectrum, is absolutely pure.
\end{rem}

\subsection{Orientations and associated virtual Thom classes}

%

In this section, we will make use of the following terminology.

\begin{df}\label{df:graded_lin_alg_group}
Given a positive integer $d$, a \emph{linear algebraic group of degree $d$} is a $\ZZ$-graded linear algebraic group $G_*$
with a homogeneous embeding $G_* \rightarrow \GL_*$ of degree $d$.
To simplify notation, we simply write $G$ for $G_*$.
\end{df}

\begin{num}\label{num:lin_gps&torsors}
Let us recall the theory of orientations in motivic homotopy theory after Panin and Walter,
following the point of view of \cite[\textsection 2]{DF3}.

Let $G$ be a linear algebraic group of degree $d$.\footnote{The most relevant cases for us are $G_*=\GL_*, \SL_*, \Sp_*$,
corresponding to $d=1, 1, 2$. Recall from \emph{loc. cit.} that one can also consider
the case of $\SL_*^c$. Besides, the orthogonal and spinor cases $ \Orth_*$ and $\Spin_*$ are also of interest.}
We will consider the category $\Tors_G(X)$ of Nisnevich $G$-torsors over a
scheme $X$.\footnote{Recall
this category is a groupoid: every morphism must be an isomorphism.}
In the classical cases, this can be viewed as the category of vector bundles over $X$
with some additional data. In any case, there is a canonical functor
$$
o:\Tors_G(X) \rightarrow \Tors_{\GL}(X)=\VB(X).
$$
The rank $r$ of a $G$-torsor $\cV$ is the rank of the associated vector bundle,
seen as a Zariski locally constant function $r:X \rightarrow \NN$.
We assume there exists a $G$-torsor $\un^G_X$ (called the \emph{unit} $G$-torsor)
such that $o(\un^G_X)=\AA^d_X$.

For convenience, let us recall the definition of $G$-orientation (under the form we use, this definition is due to Panin and Walter,
see \cite[Def. 2.1.3]{DF3})
\end{num}
\begin{df}\label{df:G-orientation}
Let $\E$ be a ring spectrum over $\Sigma$.
A $G$-orientation $\thom$ (or absolute orientation) of $\E$
is the data for all schemes $X$ and all $G$-torsors $\cV$ of rank $r$ of an element
$\thom(\cV) \in \E^{2r,r}(\Th(o(\cV)))$ such that the following properties hold:
\begin{enumerate}
\item \textit{Isomorphisms compatibility}. $\phi^*\thom(\cW)=\thom(\cV)$ for an isomorphism
$\phi:\cV \xrightarrow \sim \cW$.
\item \textit{Pullbacks compatibility}. $f^*\thom(\cV)=\thom(f^{-1}\cV)$ for a morphism of schemes $f:Y \rightarrow X$.
\item \textit{Products compatibility}. $\thom(\cV \oplus \cW)=\thom(\cV) \cdot \thom(\cW)$.
\item \textit{Normalization}. $\thom(\un_X^G)=1_X^\E$ via the identification
$\E^{2d,d}(\Th(\AA^d_X)) \simeq \E^{0,0}(X)$, where $d$ is the degree of $G$.
\end{enumerate}
One also says that $(\E,\thom)$ is a $G$-oriented ring spectrum.
\end{df}

\begin{ex} Recall the following classical examples:
\begin{enumerate}
\item $\GL$-oriented: $\mathbf{H}_{\mathrm{M}}R$ ($R$-linear motivic cohomology), $\KGL$, $\MGL$, any ring spectrum associated
with a mixed Weil theory (\cite{CD2}).
\item $\mathrm{SL}^c$-oriented: $\mathbf{H}_{\mathrm{MW}}R$ ($R$-linear MW-motivic cohomology), 
the ring spectrum $\mathbf{H}M$ associated to any MW-module $M$ over a perfect field (\cite{Feld1, Feld2}),
more generaly any ring spectrum which belong to the perverse ($\delta$-)homotopy heart of $\SH(S)$ (\cite{BD1}).
These examples include the Chow-Witt groups. Other examples include $\KO$, the ring spectrum representing higher Grothendieck-Witt groups, and
$\W=\KO[\eta^{-1}]$ (higher Balmer-Witt groups).
\end{enumerate}
Recall moreover:
\begin{itemize}
\item  If $\E$ is $G$-oriented over $S$, any $\E$-algebra in $\SH(S)$ inherits
an induced $G$-orientation. 
\item We have the following string of implications: $\GL$-oriented $\Rightarrow$ $\SL^c$-oriented $\Rightarrow$ $\SL$-oriented $\Rightarrow$ $\Sp$-oriented.
\end{itemize}
\end{ex}

\begin{num}
Recall that a (commutative) Picard groupoid is a symmetric monoidal category $\mathscr C$
which is a groupoid and such that all objects are invertible for the tensor product.
We will generically denote by $+$ the tensor product of all the groupoids considered
in this paper.

Quillen's Q-construction is a basic tool to construct Picard groupoids (see \cite{Del_det}).
Given a scheme $X$, we let $\uK(X)$ be the Picard groupoid of \emph{virtual vector bundles} over $X$:
this is the groupoid associated with the category $Q\VB(X)$.
Similarly, in the notation of the preceding definition, we let $\uK^G(X)$ be
the groupoid of \emph{virtual $G$-torsors} associated with the category $Q\Tors_G(X)$,
where the exact structure on $\Tors^G(X)$ is induced by the exact structure on $\VB(X)$
(\emph{i.e.} from exactness of sequences of $\cO_X$-modules).
We let $0^G_X$ be the $0$-object of $\uK^G(X)$.
In particular, we get a canonical functor 
\begin{equation}\label{eq:forget_virtual}
o_X:\uK^G(X) \rightarrow \uK(X)
\end{equation}
A morphism of schemes $f:Y \rightarrow X$ induces a pullback functor $f^*:\uK^G(X) \rightarrow \uK^G(Y)$
compatible with the functors $o_X$ and $o_Y$ (see also \Cref{num:fibred_cat&K} for more discussion).
The rank $\rk(v)$ of a virtual $G$-torsor $v$ is the rank of the associated virtual vector bundle.

Using homotopy invariance, one can easily extend the existence of Thom classes
from the previous definition to virtual objects.
\end{num}
\begin{prop}\label{prop:virt_thom_G-orient}\label{prop:thom_virtual}
Let $(\E,\thom)$ be a $G$-oriented ring spectrum over $\Sigma$.
Then for any virtual $G$-torsor $\tilde v$ of rank $r$,
 there exists a unique class $\thom(\tilde v) \in \E^{2r,r}(\Th(o(\tilde v)))$ such that:
\begin{enumerate}
\item For a $G$-torsor $\cV$, $\thom([\cV])=\thom(\cV)$.
\item The Thom class of virtual $G$-torsors satisfies the same properties as for $G$-torsors:
$$
\phi^*\thom(\tilde w)=\thom(\tilde v), f^*\thom(\tilde v)=\thom(f^*\tilde v), \thom(\tilde v+\tilde w)=\thom(\tilde v).\thom(\tilde w), \thom(0^G_X)=1_X^\EE,
$$
where $\phi:\tilde v \rightarrow \tilde w$ is an isomorphism in the first relation.
\end{enumerate}
\end{prop}
\begin{proof}
The proof works as in \cite[proof of 4.1.1]{Riou}. Given an exact sequence of $G$-torsors:
$$
0 \rightarrow \cV' \rightarrow \cV \rightarrow \cV'' \rightarrow 0,
$$
we have to show that one has $\thom(\cV)=\thom(\cV').\thom(\cV'')$ through
the canonical identification $\Th(o(\cV))=\Th(o(\cV')) \otimes \Th(o(\cV''))$ produced in \emph{loc. cit.}
As the $\infty$-category $\SH$ satisfies Nisnevich descent, is it sufficient to check
this over local henselian schemes. In other words, one can assume the $G$-torsors
in the above sequence are trivial, which implies that the sequence is split,
the splitting being $G$-equivariant.
In this latter case, the desired relation follows from property (3) of \Cref{df:G-orientation}.
\end{proof}

\begin{num}\textit{Cohomology of Thom spaces}.
Let $\E$ be a spectrum over $\Sigma$.
One usually extends the definition of the cohomology associated with $\E$ 
(\Cref{df:bivariant}) to Thom spaces. Given a virtual vector bundle $v$ over a scheme $X$,
and $(n,m) \in \uZ(X)^2$, one puts:
$$
\E^{n,m}(\Th(v)):=[\Th(v),\E_X(m)[n]].
$$
More generally, for any virtual vector bundle $w$ on $X$ and any $n\in\uZ(X)$ we set
\[
\E^n(\Th(v),w):=[\Th(v),\E_X\otimes\Th(w)[n]]\simeq [\Th(v-w),\E_X[n]],
\]
so that $\E^{n,m}(\Th(v))=\E^{n-2m}(\Th(v),\langle m\rangle)=\E^{n-2m}(\Th(v-\langle m\rangle)$. Therefore the locally constant integer $m$ is somewhat redundant, but we keep it for the sake of clarity.

The above abelian group is functorial in $v$ with respect to isomorphisms of virtual vector bundles,
thus producing a functor
$$
\uZ(X)^2 \times \uK(X) \rightarrow \Ab, (n,m,v) \mapsto \E^{n,m}(\Th(v)).
$$
If $\E$ has a ring structure, one obtains a product:
$$
\E^{n,m}(\Th(v)) \otimes_\ZZ \E^{s,t}(\Th(w)) \rightarrow \E^{n+s,m+t}(\Th(v+w)).
$$
In other words, $\E^{**}(\Th(*))$ is a $\uZ(X)^2 \times \uK(X)$-graded ring.
Note in particular that, fixing $v$, $\E^{**}(\Th(v))$ is an $\E^{**}(X)$-graded module.
\end{num}

As in classical cases, one deduces from Thom classes the so-called \emph{Thom isomorphisms}.
\begin{prop}
Let $(\E,\thom)$ be a $G$-oriented ring spectrum.
Then for any virtual $G$-torsor $\tilde v$ of rank $n$ over a scheme $X$, the map
$$
\E^{**}(X) \rightarrow \E^{*+2n,*+n}(\Th(\tilde v)), \lambda \mapsto \lambda.\thom(\tilde v)
$$
is an isomorphism. In other words, $\E^{**}(\Th(\tilde v))$ is a free rank one $\E^{**}(X)$-module
with basis given by the singleton $\big(\thom(\tilde v)\big)$.
\end{prop}
\begin{proof}
By the compatibility with product of Thom classes, one reduces to the case where
$\tilde v$ is associated to the class of a $G$-torsor $\cV$. The question is Nisnevich local
in $X$, so we can assume that $\cV$ is a trivial $G$-torsor. By compatibility with isomorphisms
and products, one can assume that $\cV=\un_X^G$, in which case the result follows
from the normalization property. 
\end{proof}

\begin{num}\label{num:E-modules}
Let $(\E,\thom)$ be a $G$-oriented ring spectrum over $\Sigma$.
For a scheme $X$, we can consider the category $\E_X\modd$ of $\E_X$-modules
in the monoidal category $\SH(X)$.\footnote{Beware that this category is only additive, which is sufficient for our needs. See \ref{} to get a theory with more structure.}
Given a virtual $G$-torsor $v$ of rank $r$,
the Thom class is by definition a map (in the homotopy category)
$$
\thom(\tilde v):\Th(o(\tilde v)) \rightarrow \E_X(r)[2r].
$$
By adjunction, one gets a morphism in $\E_X\modd$:
$$
\gamma_{\thom(\tilde v)}:\E_X \otimes \Th(o(\tilde v)) \rightarrow \E_X(r)[2r]
$$
which after forgetting the module structure can be written as the composite
\begin{equation}\label{eq:thomclass}
\E_X \otimes \Th(\tilde v) \xrightarrow{\Id_{\E_X} \otimes \thom(\tilde v)} \E_X \otimes \E_X(r)[2r]
\xrightarrow{\mu_X} \E_X(r)[2r]
\end{equation}
where $\mu_X$ is the multiplication map of the ring spectrum $\E_X$. In other words,
$\gamma_{\thom(\tilde v)}$ is the multiplication by the class $\thom(\tilde v)$ seen internally.
The next result is merely a reformulation of the preceding proposition.
\end{num}
\begin{cor}\label{cor:stable-thom}
Consider the above notation. Then the map $\gamma_{\thom(v)}$ is an isomorphism
of $\E_X$-modules.
\end{cor}

\begin{num}\label{num:virtual_Thom_iso}
One deduces from the previous corollary Thom isomorphisms for the various theories associated
 with a ring spectrum (\Cref{df:bivariant}, \Cref{num:biv_support}). As an example,
 for a finite type $S$-scheme $X$, locally constant integers $(n,m) \in \uZ(X)^2$ and a virtual $G$-torsor $\tilde v$ of rank $r$,
 one gets Thom isomorphisms:
\begin{equation}\label{eq:virtual_Thom_iso}
\begin{split}
\gamma_{\thom(\tilde v)}&:\E_{n,m}(X/S,\tilde v) \simeq \E_{n+2r,m+r}(X/S) \\
\gamma_{\thom(\tilde v)}&:\E^{n,m}(X,\tilde v) \simeq \E^{n-2r,m-r}(X)
\end{split}
\end{equation}
The second case is obviously a particular case of the first one!
 Both isomorphisms are induced by multiplication with another form of the Thom
 class: $\thom(v) \in \E^{2r,r}(X,-v)$, which corresponds to the Thom class of \Cref{prop:thom_virtual}
 via the isomorphism: $\E^{2r,r}(Th(v)) \simeq \E^{2r,r}(X,-v)$
 (which comes from definitions and the invertibility of the Thom space $\Th(v)$).
\end{num}

\subsection{Euler classes}

As usual, one can associate to a Thom class an Euler class (e.g. \cite[Def. 2.1.7]{DF3}). We give here a convenient definition for our purposes.

\begin{df}\label{df:Euler_class}
Consider a $G$-oriented ring spectrum $(E,\thom)$ as above.

Given a vector bundle $V$ over a scheme $X$,
 a stable $G$-orientation of $V$ is a pair $(\cV,\omega)$, where $\cV \in \uK^G(X)$
 and $\omega\colon V\simeq o(\cV)$ in $\uK(X)$.
Given an orientation $(\cV,\omega)$ on $V$,
 one defines the ($G$-oriented) \emph{Euler class} of $(V,\cV,\omega)$ as
$$
e(V,\cV,\omega)=p^*(\th(\cV)) \in  \E^{2r,r}(X),
$$ 
where $r=\mathrm{rank}(V)$ and $p$ is the composite
\[
p:X \xrightarrow{s_0} V \rightarrow \Th(V)\xrightarrow{\omega}\Th(o(\cV))
\]
of the zero section and the quotient map. We say that the stable orientations $(\cV,\omega)$ and $(\cW,\psi)$ are isomorphic if there exists an isomorphism $f:\cV\to \cW$ in $\uK^G(X)$ making the diagram
\[
\xymatrix{V\ar[r]^-{\omega}\ar@{=}[d] & o(\cV)\ar[d]^-{o(f)} \\
V\ar[r]_-{\psi} & o(\cW)}
\]
commutative.
\end{df}

\begin{rem}\label{rem:invariantGisom}
It is easy to check that the $G$-oriented Euler class of $(V,\cV,\omega)$ depends only on the isomorphism class
 of the $G$-orientation $(\cV,\omega)$. 
\end{rem}

\begin{rem}
When $E$ is $\GL$-oriented, the Euler class of a vector bundle $V$ is the top Chern class of $V$
(see eg. \cite[Rem. 2.4.5]{Deg12}). In Chow-Witt groups (of MW-motivic cohomology),
the Euler class of an $\SL^c$-torsor was first considered by Barge and Morel, and the second author
(see \cite{Fasel08a}).
\end{rem}

\begin{rem}\label{ex:virtual_Euler_classes}
Note that the definition of Euler classes (\Cref{df:Euler_class})
cannot be extended to virtual $G$-torsors $\tilde v$ over a scheme $X$,
as there is no map $X \rightarrow \Th(\tilde v)$ in general.
\end{rem}

\subsection{Fibred categories and virtual Thom classes}

\begin{num}\label{num:fibred_cat&K}
There is a more concise way of encoding
a $G$-orientation on a ring spectrum, using the language of \emph{fibred categories}.
We refer the reader to \cite{Vistoli} (see also \cite{Gray}).
 By what is classicaly called the Grothendieck construction,
 a fibred category $\phi:\mathscr S \rightarrow \mathscr B$ (\cite[Def. 3.5]{Vistoli})
 can equivalently be described by a (contravariant) \emph{pseudo-functor} (\cite[Prop. 3.13]{Vistoli}):
$$
\psi:\mathscr B^{\mathrm{op}} \rightarrow \Cat
$$
where $\Cat$ is the category of (small) categories, such that for any object $S$ of $\mathscr B$,
$\psi(S)$ is the \emph{fibre category} of $\phi$ over $S$: \emph{i.e.} the full subcategory of $\mathscr S$ whose objects
$X$ satisfy $\phi(X)=S$. Such a pseudo-functor is uniquely associated to the choice of a \emph{cleavage},
which amounts for any morphism $f:T \rightarrow S$ of a pullback functor
$f^*:\psi(S) \rightarrow \psi(T)$. Recall that the word pseudo-functor refers to the fact that one
only gets isomorphism $(\Id_S)^* \simeq \Id_{\psi(S)}$ and $g^*f^* \simeq (gf)^*$, satisfying
natural compatibility conditions and uniquely determined by the cleavage (see \cite[Def. 3.10]{Vistoli}).

All our fibred categories come with an exclicit choice of cleavage, so that we will consider them
as pseudo-functors. Beware however that a morphism of fibred categories (\cite[Def. 3.6]{Vistoli})
only corresponds to what is called a pseudonatural transformation (see \cite[end of 1.5 and 1.6(c)]{Gray}).
Considering a sub-category $\Cat_0$ of $\Cat$,
a fibred category over $\mathscr B$ in $\Cat_0$ (or simply \emph{$\mathscr B$-fibred $\Cat_0$-category})
will be a fibred category whose pseudo-functor has
the form $\psi_0:\mathscr B \rightarrow \Cat_0$.
In our main example, $\Cat_0$ will the be category of Picard groupoids,
but many other cases can appear (additive, triangulated monoidal).

Consider a graded linear algebraic group $G$ as in \Cref{num:lin_gps&torsors}.
Then $X \mapsto \Tors_G(X)$ is a fibred additive category over $\base$,
whose pullback functor is given by the pullback of $G$-torsors.
One deduces (from the functoriality of the associated Picard groupoid)
that $X \mapsto \uK^G(X)$ is a fibred Picard groupoid over $\base$.
Moreover, the functor \eqref{eq:forget_virtual} induces a morphism
of $\base$-fibred Picard groupoids:
$$
o:\uK^G \rightarrow \uK.
$$

The following definition gives an axiomatization of this situation
(it will be useful in the symplectic and orthogonal cases,
see \Cref{sec:sp_orientaiton}).
\end{num}
\begin{df}\label{df:stable_structure}
Let $d>0$ be an integer.
A \emph{stable structure} of degree $d$ on vector bundles (over $\Sigma$) will be:
\begin{itemize}
\item A pseudo-functor $\uK^\sigma:\base \rightarrow \mathscr P$ 
from  $\base$ to the category of Picard groupoids.
\item A natural transformation of pseudo-functors:
$o:\uK^\sigma \rightarrow \uK$.
\item An integer $d>0$ and a cartesian section\footnote{\emph{i.e.} a collection of objects $\un^\sigma_X
\in \uK^\sigma(X)$ indexed by schemes $X$ in $\base$ and stable under pullbacks} $\un^\sigma$ of $\uK^\sigma$
such that $o(\un_X^\sigma)=\tw d$.
\end{itemize}
As a companion, we slightly extend the first part of \Cref{df:Euler_class}. 
 Let $v$ be a virtual vector bundle $v \in \uK(X)$ (resp. $V$ be a vector bundle over $X$).
 A \emph{$\sigma$-orientation} of $v$ (resp. $V$) is a pair $(\tilde v,\omega)$, where $\tilde v \in \uK^\sigma(X)$
 and $\omega$ is an isomorphism $o(\tilde v) \simeq v$ (resp. $o(\tilde v) \simeq \tw V$).
\end{df}

\begin{ex}\label{ex:stable_G-structure}
Let $G$ be a linear algebraic group of degree $d$.
The fibered Picard groupoid $\uK^G$ and the functor $o$ defines a stable structure $\sigma_G$
on vector bundles, as explained in \Cref{num:lin_gps&torsors},
which we call the \emph{$G$-stable structure}.
The main cases for us are $G=\SL$ ($d=1$) and $G=\Sp$ ($d=2$).
Note that the tautological case $G=\GL$ ($\sigma=\Id$, $d=1$) is also interesting
from the point of view of orientations.
\end{ex}

\begin{num}
The stable homotopy category $\SH$ is a monoidal triangulated fibred category over $\base$.
We will consider $\SH^\times$ the underlying Picard groupoid of $\otimes$-invertible objects:
for a scheme $S$, $\SH^\times(S)$ is the monoidal category of spectra which are $\otimes$-invertible,
and whose morphisms are isomorphism is $\SH(S)$.

Given a ring spectrum $\E$ (over $\Sigma$), the fibred structure of $\SH$ induces
a fibred structure on $X \mapsto \E_X\modd$, which is therefore an additive $\base$-fibred category $\E\modd$.
As the ring structure on $\E$ is only considered in the homotopy category,
it is not possible to put a monoidal structure on the latter fibred category. However,
there is a morphism $L_{L_\E}:\SH \rightarrow \E\modd$ of additive fibred categories, induced by the
free $\E$-module functor. Given a scheme $S$, we consider the category $\E_X\modd^\times$
whose objects are the $\otimes$-invertible spectra $\F$, and morphismes from $\F$ to $\F'$ are given 
by the sub-abelian group of
$$
\Hom_{\SH(X)}(\F,\E_X \otimes \F') \simeq \Hom_{\E_X\modd}(\E_X \otimes \F,\E_X \otimes \F')
$$
consisting of isomorphisms of $\E_X$-modules.
There is a natural tensor product $\otimes_\E$ on this category, given on objects by the tensor structure
on $\SH(X)$ and on morphisms by the cup-product:
\begin{align*}
&u:\F \rightarrow \E_X \otimes \G, v:\F' \rightarrow \E_X \otimes \G', \\
& u \otimes_\E v:\F \otimes \F' \xrightarrow{u \otimes v} \E_X \otimes \E_X \otimes \G \otimes \G'
\xrightarrow{\mu_X^\E} \E_X \otimes \G \otimes \G'.
\end{align*}
This defines a $\base$-fibred Picard groupoid $\E\modd^\times$, and the free $\E$-module functor
induces a morphism of fibred Picard groupoids: $\SH^\times \xrightarrow{L_\E} \E\modd^\times$.
\end{num}
\begin{df}\label{df:sigma-orientation}
Consider a stable structure $\sigma$ on vector bundles
and an absolute ring spectrum $\E$.

A \emph{$\sigma$-orientation} on the ring spectrum $\E$ will be the data of a natural isomorphism
$\gamma$ 
pictured by the double arrow in the following diagram 
whose objects are pseudo-functors,
and simple arrows are natural transformations of pseudo-functors of fibred Picard groupoids:
$$
\xymatrix@R=6pt@C=30pt{
\uK^\sigma\ar^\sigma[r]\ar_{\rk \circ \sigma}[ddd]\ar@{=>}_{\gamma}[rdd]
& \uK\ar^-{\Th}[r]
& \SH^\times\ar^-{L_\E}[r]
&  \E\modd^\times \\
&&& \\
&&& \\
\uZ\ar_{tw^\E_{\PP^1}}[rrruuu] &&&
}
$$
where $\rk:\uK \rightarrow \uZ$ is the pseudo-functor induced by the rank of a virtual bundle,
and $tw^\E_{\PP^1}$ is the pseudo-functor which to an element $n \in \uZ(X)$, associates
the object
$$
tw_{\PP^1}(n)=\bigoplus_{x \in X^{(0)}} \E(n(x))[2n(x)].
$$
One requires further the following normalization condition:
For any scheme $X$, there exists an isomorphism
$\gamma_{(\un^\sigma_X)} \simeq \Id_{\E(d)[2d]}$ compatible with pullbacks.

In this situation, we call $\gamma$ the \emph{$\sigma$-stable Thom isomorphism} on $\E$.
\end{df}

\begin{num}\label{num:Thom_iso_explicit}
Explicitly, a $\sigma$-orientation is the data for any scheme $X$,
and for every $\tilde v \in \uK^\sigma(X)$, of a \emph{Thom isomorphism} in the
category of $\E$-modules over $X$:
\begin{equation}\label{eq:sigma_Thom_iso_map}
\gamma_{\tilde v}:\E_X \otimes \Th(v) \xrightarrow{\sim} \E_X(r)[2r],
\end{equation}
where $v=o(\tilde v)$ is the associated virtual bundle, and $r$ is the rank of $v$.
In addition, this Thom isomorphism is functorial in $\tilde v$ with respect to isomorphisms,
compatible with base change, and compatible with the addition of $v$.

The isomorphism $\gamma_{\tilde v}$ defines a cohomology class:
\begin{equation}\label{eq:sigma-Thom_class}
\th(\tilde v) \in \E^{2r,r}(\Th(v)) \simeq \Hom_{\E\modd_X}(\E_X \otimes \Th(v),\E_X(r)[2r])
\end{equation}
which will be called the \emph{Thom class} of $\tilde v$ with coefficients in $\E$.
 One readily deduces that multiplication by this Thom class induces isomorphisms:
$$
\E^{*,*}(X) \xrightarrow{\sim} \E^{*+2r,*+r}(\Th(v))
$$
and similarly for the other theories associated with $E$ (as in \Cref{num:virtual_Thom_iso}).
\end{num}

The following result is merely a reformulation of \Cref{cor:stable-thom},
showing that \Cref{df:sigma-orientation} is in fact a generalization of 
\Cref{df:G-orientation}.
\begin{prop}\label{prop:stable_G-orientations}
Let $\E$ be a ring spectrum and $G$ a linear algebraic group of degree $d$ (\Cref{df:graded_lin_alg_group}).
Then the map $\thom \mapsto \gamma_\thom$ of \Cref{cor:stable-thom} gives a bijection between the following structures:
\begin{enumerate}
\item the $G$-orientations $\thom$ on $\E$;
\item the $\sigma_G$-stable Thom isomorphisms $\gamma$ on $\E$ (\Cref{ex:stable_G-structure}).
\end{enumerate}
\end{prop}
\begin{proof}
Indeed, the fact that $\gamma_\thom$ is a natural transformation of pseudo-functors comes from the construction
and the compatibility with pullbacks and products in \Cref{df:G-orientation}. 
\end{proof}

\begin{rem}
Using hermitian K-theory, we will give a finer statement in the symplectic case in \Cref{sec:virt_symplect}.
\end{rem}


\begin{num}\textit{$\sigma$-oriented Euler classes}.\label{num:sigma-Thom&Euler}
As in \Cref{num:virtual_Thom_iso}, one gets Thom isomorphisms \eqref{eq:virtual_Thom_iso}
 associated with an element $\tilde v \in \uK^\sigma(X)$ in the various theories
 (cohomology with/without proper support, bivariant with/without proper support)
 with coefficients in $\E$.

One can therefore extend \Cref{ex:virtual_Euler_classes}.
 Let $V/X$ be a vector bundle of rank $r$,
  with a given $\sigma$-orientation $(\tilde v,\omega)$, where $\tilde v \in \uK^{\sigma}(X)$ and $\omega:V\simeq o(\tilde v)$
  (\Cref{df:sigma-orientation}). Then one defines the associated Euler class as follows:
$$
e(V,\tilde v,\omega,\E)=\gamma_{\thom(\tilde v)}\big(e(V,\E)\big) \in \E^{2r,r}(X).
$$
Note that when $\sigma=\sigma_G$ as above, this definition obviously extend
 \Cref{df:Euler_class}.
 Moreover, as in \Cref{rem:invariantGisom},
 one deduces that once $\omega$ is chosen this $\sigma$-oriented Euler class depends only on the isomorphism
 class of $\tilde v$ in $\uK^\sigma(X)$. In other words,
 one can consider that $\tilde v$ belongs to the abelian group
 $K_0^\sigma(X)=\pi_0(\uK^\sigma(X))$.
\end{num}

%

\section{Fundamental classes and Grothendieck-Riemann-Roch formulas}

\subsection{Oriented fundamental classes}

\begin{df}
Let $\sigma$ be a stable structure on vector bundles (\Cref{df:stable_structure}).

Let $f:X \rightarrow S$ be an lci morphism with virtual tangent bundle $\tau_f$,
 seen as an object of the category $\uK(X)$.
 A \emph{$\sigma$-orientation} on $f$ is an element $\tilde \tau_f \in \uK^\sigma(X)$ and an isomorphism
 $o_\sigma(\tilde \tau_f) \simeq \tau_f$ --- in other words, a $\sigma$-orientation of its virtual
 tangent bundle (\Cref{df:sigma-orientation}).
 When $\sigma=\sigma_G$ is the stable structure associated with a linear algebraic group $G$ of degree $d$
 (\Cref{ex:stable_G-structure}), we simply say $G$-orientation.
\end{df}

%

\begin{ex}\label{ex:orientations}
\begin{enumerate}
\item Consider a smooth projective K3-surface over a field. 
 By definition, its canonical sheaf is trivial. In particular, it admits a symplectic form
 $\psi:\Lambda^2 T_X \rightarrow \cO_X$.
 Examples include for instance smooth quartics in $\PP^3_k$, or a generic intersection of a quadric and a cubic in $\PP^4_k$.
 More generally, any fibre product of $\mathrm{K}3$-surfaces admits a symplectic orientation.
 Thus, there are varieties over $k$ with a symplectic orientation for any even dimension. More complicated examples in any even dimension can be found in \cite{Beauville83}.
\item An $\SL^c$-orientation of a smooth morphism $f$ corresponds is a choice of a square root
 of the determinant $\omega_f$ of the cotangent sheaf $\Omega_f$: $\phi:\omega_f \simeq \mathcal L^{\otimes 2}$.
This is sometimes simply called an orientation of $f$ (or of $X/S$).

When $f:C \rightarrow \Spec(k)$ is a projective smooth curve over a field,
such an ($\SL^c$-)orientation is more classically called a $\Theta$-characteristic of $C$.
\item A symplectic orientation $\tilde \tau_f$ of $f$ uniquely determines
 an ($\SL^c$-)orientation of $f$.
\end{enumerate}
\end{ex}

We can now introduce the oriented fundamental class in the most general case of $\sigma$-structures.
\begin{df}\label{def:orientedfundamentalclass}
Let $\sigma$ be a stable structure on vector bundles, as in the previous definition.

Let $\E$ be a $\sigma$-oriented ring spectrum,
 and $f:X \rightarrow S$ be a $\sigma$-orientable smoothable lci morphism,
 with virtual dimension $d$, virtual tangent bundle $\tau_f$ and $\sigma$-orientation
 $\varphi:o_\sigma(\tilde \tau_f)\to \tau_f$.

One defines the \emph{$\sigma$-oriented fundamental class} of $f$ with coefficients in $\E$ as
 the class 
$$\tilde \eta_f^\E \in \E_{2d,d}(X/S)$$
obtained as the image of the
 fundamental class $\eta_f^\E$ (\Cref{num:fdl_class}) under the inverse of
 the Thom isomorphism $\gamma_{\tilde \tau_f}:\E_{2d,d}(X/S) \rightarrow \E_{0,0}(X/S,\tau_f)$
 associated with $\tilde \tau_f$ -- see \Cref{num:sigma-Thom&Euler}.
\end{df}

\begin{rem}
Beware that the class $\tilde \eta_f^\E$ depends on the choice of
 both the $\sigma$-orientation on $\E$ and on $f$: in our notation,
 it is intended that $\E$ and $f$ are $\sigma$-oriented objects.
 Similar conventions are taken below.
\end{rem}

\begin{num}\label{num:sigma_fundamental_classes_ppty}
The preceding fundamental classes satisfy good properties: 
 when restricted to the class $\mathcal C$ consisting of $\sigma$-oriented smoothable lci morphisms,
 they form a system of fundamental classes in the sense of \cite[Def. 2.3.4]{DJK},
 where the \emph{twist} associated to $f \in \mathcal C$ is the element
 $e_f=\tw{r_f}$ of $K_0(X)$, the trivial virtual vector bundle whose rank is the of rank of $f$.

In particular, they are normalized, compatible with composition (associativity formula),
 and satisfy the transversal base change (see \emph{loc. cit.} for details). This follows
 from the analogue properties of the classes $(\eta_f)$, and the compatibility of Thom isomorphism
 with pullbacks, and addition (\Cref{num:Thom_iso_explicit}).
 One also has an excess intersection formula.
\end{num}
\begin{prop}\label{prop:sigma_excess_intersection}
Consider a $\sigma$-oriented ring spectrum, and a cartesian square of schemes
$$
\xymatrix@=10pt{
Y\ar^g[r]\ar_v[d]\ar@{}|\Delta[rd] & T\ar^u[d] \\
X\ar_f[r] & S
}
$$
such that $f$ and $g$ are smoothable lci and $\sigma$-oriented, $n=\dim(f)$.
 Let $\xi$ be the excess intersection bundle associated with
 $\delta$.\footnote{To define $\xi$, one consider a factorisation $X \rightarrow P \rightarrow S$
 into a regular closed immersion and a smooth morphism, and put, let $Y \rightarrow Q \rightarrow T$
 be its pullback along $T/S$, and put: $\xi=v^{-1}N_XP/N_YQ$. See e.g. \cite[3.3.3]{DJK}.}
 We put $\tilde \xi=v^{-1}\tilde \tau_f-\tilde \tau_g$, and see it as a $\sigma$-orientation of $\xi$.

Then the following \emph{excess intersection formula} holds in $\E^{2n,n}(Y/T)$:
$$
\Delta^*(\tilde \eta^\E_f)=e(\tilde \xi,\E).\tilde \eta^\E_g.
$$
\end{prop}

\begin{num}\textit{Oriented Gysin maps}.\label{num:general_Gysin}
 As explained in \cite[\textsection 2.4]{DJK} (following \cite{FMP}),
 one deduces from the previous fundamental classes the existence of Gysin maps in the theories with coefficients
 in a $\sigma$-oriented ring spectrum.
 Explicitly, let us consider a $\sigma$-oriented lci smoothable morphism $f:Y \rightarrow X$ of relative dimension $d$.
 Then one gets:
\begin{align*}
f_!^\E&:\E^{n,m}(Y) \rightarrow \E^{n-2d,m-d}(X), &f \text{ proper}, \\
f^!_\E&:\E_{n,m}(X/S) \rightarrow \E_{n+2d,m+d}(Y/S), &f \text{ $S$-morphism}, \\
f_!^\E&:\E^{n,m}_c(Y/S) \rightarrow \E^{n-2d,m-d}_c(X/S), &f \text{ $S$-morphism}, \\
f^!_\E&:\E_{n,m}^c(X/S) \rightarrow \E_{n+2d,m+d}^c(Y/S), &f \text{ proper $S$-morphism}.
\end{align*}
These maps are essentially induced by multiplication with the fundamental class $\tilde \eta_f^\E$:
 see \cite[3.3.2]{Deg16}, \cite[2.4.1]{DJK}.
 As an example, for $y \in \E^{n,m}(Y)$, one has:
$$
f_!(y)=f_*(y.\tilde\eta_f^\E)
$$
where the product is taken in the underlying bivariant theory,
 and lands in $\E_{2d-n,d-m}(Y/X)$, while
 $f_*:\E_{2d-n,d-m}(Y/X) \rightarrow \E_{2d-n,d-m}(Y/Y)=\E^{n-2d,m-d}(Y/Y)$
 is the pushforward in bivariant theory. See also \cite[2.4.1]{DJK} in the second case.
\end{num}

\begin{rem}
From the properties of $\sigma$-oriented fundamental classes (\Cref{num:sigma_fundamental_classes_ppty}), \Cref{prop:sigma_excess_intersection},
 one deduces the ``usual'' formulas for these Gysin maps: compatibility with composition,
 projection formula in the transversal case, excess intersection formula.
\end{rem}

\begin{num}\textit{Internal Gysin maps}. It is possible to give a more categorical
 formulation of Gysin maps, in the spirit of Grothendieck's trace map in coherent duality
 (\cite[III, Th. 10.5]{Hart}) or trace map in the \'etale formalism (\cite[XVIII, Th. 2.9]{SGA4}).

Through the isomorphism $\E_{2d,d}(X/S) \simeq \Hom_{\E\modd_X}(\E_X(d)[2d],f^!\E_S)$,
 the $\sigma$-oriented fundamental class corresponds to the following composite map,
 called the \emph{cotrace map}:
\begin{equation}\label{eq:cotrace}
\mathrm{cotr}_f^\E:\E_X(d)[2d] \xrightarrow{\gamma_{\tilde v}^{-1}} \E_X\otimes \Th(o(\tilde v))\xrightarrow{\varphi^*}\E_X \otimes \Th(v)
 \xrightarrow{\eta_f^\E} f^!(\E_S).
\end{equation}
The first morphism is the (inverse of the) Thom isomorphism \eqref{eq:sigma_Thom_iso_map}, the second one the stable orientation of $v$
 and the third one is the fundamental class \eqref{eq:fdl_biv}.
 By adjunction, one deduces the \emph{trace map}:
\begin{equation}\label{eq:trace}
\mathrm{tr}_f^\E:f_!(\E_X)(d)[2d] \rightarrow \E_S.
\end{equation}
Both maps can be considered in the category of $\E$-modules or
 (after forgetting the $\E$-module structure) in the stable homotopy category.
 They directly induce the Gysin morphisms defined above
 (see also \cite[\textsection 4.3]{DJK}).
\end{num}

%
%

The following result is a tautology, given that the first map
 in \eqref{eq:cotrace} is always an isomorphism.
\begin{prop}
Let $\E$ be  a $\sigma$-oriented ring spectrum
 and $f$ a smoothable lci $\sigma$-oriented morphism, of relative dimension $d$.
 Assume that $\E$ is $f$-pure (\Cref{df:f-pure}).

Then, the cotrace map \eqref{eq:cotrace} associated with $f$ and $\E$ is an isomorphism.
In particular, it induces by functoriality \emph{duality isomorphisms}:
\begin{align*}
\E^{n,m}(X) &\xrightarrow{\ \sim\ } \E_{2d-n,d-m}(X/S) \\
\E^{n,m}_c(X/S) &\xrightarrow{\ \sim\ } \E^c_{2d-n,d-m}(X/S).
\end{align*}
In fact, both morphisms are given by multiplication with the $\sigma$-oriented fundamental class $\tilde \eta_f^\E$.
\end{prop}

\begin{rem}
In fact, one can extend the first duality isomorphism.
 Given any $X$-scheme $Y$ of finite type, mutliplication by the class $\tilde \eta_f^\E$ induces an isomorphism:
$$
\E^{n,m}(Y/X) \xrightarrow{\ \sim\ } \E_{2d-n,d-m}(Y/S).
$$
This statement is actually a generalization of the axiom considered in Bloch-Ogus twisted duality theory
 (\cite{BlochOgus}).
\end{rem}

\subsection{Todd classes and GRR formulas}

\begin{num}\label{num:todd}
As in the preceding section, we consider a stable structure $\sigma$
 of rank $r$ on vector bundles
 (\Cref{df:stable_structure}).
 We also consider $\sigma$-oriented ring spectra $(\E,\thom(-,\E))$ and $(\F,\thom(-,\F))$,
 and a morphism of ring spectra $\psi:\E \rightarrow \F$.
 We will generically denote by $\psi_*$ the map of cohomology theories or bivariant theories,
 with or without support, induced by $\psi$.
 Given a scheme $X$, we let $\mathrm{K}_0^\sigma(X)$ be the abelian group of isomorphisms classes of objects
 of $\uK^\sigma(X)$.

From the formalism developed previously,
 the following fundamental proposition is now obvious.
\end{num}
\begin{prop}\label{prop:computeToddclass}
Consider the above assumption.

Then for any scheme $X$, there exists a unique morphim of abelian groups:
$$
\td_\psi:\mathrm{K}_0^\sigma(X) \rightarrow \F^{00}(X)^\times
$$
contravariantly functorial in $X$ and such that for any element $\tilde v \in \mathrm{K}^\sigma_0(X)$
 of rank $n$,
 the following relation holds in $\F^{2n,n}(\Th(o(\tilde v)))$
\begin{equation}\label{eq:Todd&Thom}
\thom(\tilde v,\F)=\td_\psi(\tilde v)\cdot\psi_*(\thom(\tilde v,\E))
\end{equation}
using the product of bivariant theories (\Cref{prop:biv_product}).

Moreover, for any $\sigma$-oriented vector bundle $V/X$,
 with orientation $(\tilde v,\omega)$ where $\tilde v \in \uK^\sigma(X)$ is of rank $n$ and $\omega:o(\tilde v)\simeq V$,
 one has the following relation between Euler classes in $\F^{2n,n}(X)$:
\begin{equation}\label{eq:Todd&Euler}
e(V,\tilde v,\omega,\F)=\td_\psi(\tilde v)\cdot\psi_*(e(V,\tilde v,\omega,\E)).
\end{equation}
\end{prop}
\begin{proof}
Consider an element $\tilde v \in \uK^\sigma(X)$, and put $v=o(\tilde v)$.
 According to \Cref{num:Thom_iso_explicit}, the Thom class $\th(v)$
 forms a basis of the bigraded $\F^{**}(X)$-module $\F^{**}(\Th(v))$.
 Therefore, the existence and uniqueness of $\td_\psi(\tilde v)$ is obvious,
 except that one has to prove that the latter class depends only on the isomorphism
 class of $\tilde v$.
 Assume there exists an isomorphism $\tilde \mu:\tilde v \rightarrow \tilde v'$ in $\uK^\sigma(X)$.
One deduces an isomorphism $\mu:v=o(\tilde v) \rightarrow o(\tilde v')=v'$ of virtual vector
 bundles. By compatibility of Thom classes with isomorphisms,
 one deduces the following relations, where $\epsilon=\E,\F$:
$$
\mu_*(\thom^\epsilon(\tilde v))=\thom^\epsilon(\tilde v').
$$
Thus the statement follows from relation \eqref{eq:Todd&Thom} and the uniqueness
 of $\td_\varphi(-)$.
 Finally, relation \eqref{eq:Todd&Euler} follows from the definition given in \Cref{num:sigma-Thom&Euler}
 and relation \eqref{eq:Todd&Thom}. 
\end{proof}

\begin{rem}
The main problem of the theory of Todd classes given by the above proposition
 is to find methods to obtain them explicitly. In the $\GL$ (resp. $\Sp$) oriented case,
 we have a theory of characteristic classes and a splitting principle
 that will allow us to express Todd classes as formal power series and to have
 an effective tool to compute them.
\end{rem}

As an immediate corollary, one obtains the generalized
 Grothendieck-Riemann-Roch formula \emph{à la} Fulton-MacPherson:
\begin{thm}[Grothendieck-Riemann-Roch formula]\label{thm:GRR}
Consider the assumptions of the previous proposition.
 Let $f:Y \rightarrow X$ be a smoothable lci morphism of dimension $d$,
 which is $\sigma$-oriented with $\sigma$-orientation $\tilde \tau_f$. Then the following GRR formula holds in $\F^{2d,d}(Y/X)$:
$$
\psi_*(\tilde\eta_f^\E)=\td_\psi(\tilde \tau_f)\cdot\tilde\eta_f^\F.
$$
Consequently, the following diagrams commute
\[
\xymatrix@=16pt@C=44pt{
\E^{**}(Y)\ar^{f_!^\E}[r]\ar_{\td_\psi(\tilde \tau_f).\psi_*}[d]\ar@{}|{(1)}[rd] & \E^{**}(X)\ar^{\psi_*}[d] 
 & \E_{**}(X/S)\ar^{f^!_\E}[r]\ar_{\td_\psi(\tilde \tau_f).\psi_*}[d]\ar@{}|{(2)}[rd] & \E_{**}(Y/S)\ar^{\psi_*}[d] \\
\F^{**}(Y)\ar_{f_!^\F}[r] & \F^{**}(X)
 & \F_{**}(X/S)\ar_{f^!_\F}[r] & \F_{**}(Y/S) \\
\E^{**}_c(Y/S)\ar^{f_!^\E}[r]\ar_{\td_\psi(\tilde \tau_f).\psi_*}[d]\ar@{}|{(3)}[rd] & \E^{**}_c(X/S)\ar^{\psi_*}[d]
 & \E_{**}^c(X/S)\ar^{f^!_\E}[r]\ar_{\td_\psi(\tilde \tau_f).\psi_*}[d]\ar@{}|{(4)}[rd] & \E_{**}^c(Y/S)\ar^{\psi_*}[d] \\
\F^{**}_c(Y/S)\ar_{f_!^\F}[r] & \F^{**}_c(X/S)
 & \F_{**}^c(X/S)\ar_{f^!_\F}[r] & \F_{**}^c(Y/S) 
}
\]
where the $!$-maps are the Gysin morphisms as defined in \Cref{num:general_Gysin},
that exist under the following assumptions:
\begin{enumerate}
\item[(a)] $f$ is proper in cases (1) and (4);
\item[(b)] $f$ is a morphism of finite type $S$-schemes in cases (2), (3) and (4).
\end{enumerate}
\end{thm}

\begin{proof}
Recall first from \eqref{eq:cotrace} that the $\sigma$-oriented fundamental class is obtained from the cotrace map
\[
\tilde\eta_f^\E\colon\E_X(d)[2d] \xrightarrow{(\gamma_{\tilde \tau_f}^{\E})^{-1}} \E_X\otimes \Th(o(\tilde \tau_f))\xrightarrow{\varphi^*}\E_X \otimes \Th(\tau_f)
\xrightarrow{\eta_f^\E} f^!(\E_S).
\]
Applying $\psi_*$ and using the fact that it respects the \emph{non oriented} fundamental classes of \Cref{num:fdl_class}, we obtain 
\[
\psi_*(\tilde\eta_f^\E)\colon\F_X(d)[2d] \xrightarrow{\psi_*(\gamma_{\tilde \tau_f}^{\E})^{-1}} \F_X\otimes \Th(o(\tau_f))\xrightarrow{\varphi^*}\F_X \otimes \Th(\tau_f)
\xrightarrow{\eta_f^\F} f^!(\F_S).
\]
From \eqref{eq:Todd&Thom} and \eqref{eq:thomclass} we deduce that $\psi_*(\gamma_{\tilde \tau_f}^{\E})^{-1}=(\gamma_{\tilde \tau_f}^{\F})^{-1}\cdot \td_{\psi}(\tilde \tau_f)$, yielding the equality $\psi_*(\tilde\eta_f^\E)=\td_\psi(\tilde \tau_f)\cdot\tilde\eta_f^\F$. The displayed diagrams then commute in view of definition \Cref{num:general_Gysin}.
\end{proof}


\section{Symplectic orientations and formal ternary laws}

\subsection{Recollections on FTL}

\begin{num}\label{num:multi-valued} \textit{Multi-valued series}.--
Formal ternary laws were introduced by Walter in an unpublished work,
 and then formally developped in \cite{DF3}. More recently, we further expanded the theory
 in \cite{CDFH}, introducing in particular a slightly more accurate framework that we now briefly recall,
 referring the reader to \cite[\textsection 1]{CDFH} for more details. 

 Let $R$ be a (commutative) ring. Given a pair of integers $(n,d) \in \ZZ^2$, $\ux=(x_1,...,x_d)$ a $d$-tuple of integers,
 an $n$-valued $d$-dimensional series, $(n,d)$-series for short,
 with coefficients in $R$ will be an element 
$$
F_t(\ux)=1+F_1(\ux)t+\hdots+F_n(\ux)t^n
$$
in $R[[\ux]][t]$, polynomial in $t$ of degree $n$.
 These $(n,d)$-series form an $R$-module $\MForm_{n,d}(R)$.
 We will use a \emph{$\ZZ$-filtration} on $\MForm_{n,d}(R)$ by assigning to the variables $x_i$ degree $+1$
 and to the variable $t$ degree $-1$. Therefore, an $(n,d)$-series has valuation greater or equal to $(-n)$
 and can have infinite degree.

One can define a generalized composition on these $(*,*)$-series
 (see \cite[Def. 2.1.13]{CDFH} for more details).
 Let $G_t(\uy)$ be an $(m,r)$-series. We say that $G_t(\uy)$ is \emph{composable}
 if $G_t(\underline 0)=1$. We then introduce formal variables $(G^{[1]},\hdots,G^{[m]})$,
 called the \emph{roots} of $G_t(\uy)$, satisfying the formal identity
\begin{equation}\label{eq:roots}
G_t(\uy)=\prod_{l=1}^m (1+G^{[l]}t).
\end{equation}
In other words, the coefficients of $G_t(\uy)$ are given by the elementary symmetric functions
 in the variables $G^{[l]}$.

Then we define the \emph{substitution of $G_t(\uy)$ in $F_t(\ux)$ at the variable $x_i$}
 as the $(nm,d+r-1)$-series in the variables $(x_1,\hdots,\cancel{x_i},\hdots,x_d,y_1,\hdots,y_r)$
 by the following equality:
$$
F_t\big(x_1,...,x_{i-1},G_t(\uy),x_{i+1},...,x_d\big)
 :=\prod_{l=1}^m F_t\big(x_1,...,x_{i-1},G^{[l]},x_{i+1},...,x_d\big).
$$
More precisely, the expression on the right is obviously symmetric in the variables $G^{[l]}$.
 Thus it can be expressed as a polynomial in the elementary symmetric polynomials $e_s$ in the
 $G^{[l]}$, and we can therefore substitute to the $e_s$ the actual coefficients of the
 polynomial $G_t(\uy)$, following the rule \eqref{eq:roots}.

The following algebraic notion is an analog of formal group laws,
 especially relevant in motivic homotopy theory (see also \cite[Def. 3.1.5]{CDFH}).
\end{num}
\begin{df}\label{df:FTL}
Let $R$ be a $\ZZe$-algebra.
A \emph{formal ternary law}, FTL for short,
 with coefficients in $R$ is a $(4,3)$-series with coefficients in $R$
$$
F_t(\ux)=1+F_1(x,y,z)t+F_2(x,y,z)t^2+F_3(x,y,z)t^3+F_4(x,y,z)t^4
$$
satisfying the following properties:
\begin{enumerate}
\item \textit{Neutral element}. The $(4,1)$-series $F_t(x,0,0)$ is split
with roots $x$ and $-\epsilon x$ each with multiplicity $2$,
 \emph{i.e.} $F_t(x,0,0)=(1+xt)^2(1-\epsilon xt)^2$.
\item \textit{Semi-neutral element}. The following relation holds:
$$
F_4(x,x,0)=0.
$$
\item \textit{Symmetry}. The element $F_t(x,y,z)$ of $R[[x,y,z]][t]$ is fixed
 under the action of the group $\mathfrak S(x,y,z)$ permuting the formal variables.
\item \textit{Associativity}. Given formal variables $(x,y,z,u,v)$,
 one has the following equality of $(16,5)$-series:
$$
F_t\big(F_t(x,y,z),u,v\big)=F_t\big(x,F_t(y,z,u),v\big)
$$
where we have used the above substitution operation as $F_t(x,y,z)$ is composable by property (1).
\item \textit{$\epsilon$-Linearity}. One has the following relation in $\MForm_{4,3}(R)$:
$$
F_t(-\epsilon x,y,z)=F_{-\epsilon t}(x,y,z).
$$
\end{enumerate}
We frequently display an FTL by its coefficients:
\begin{equation}\label{eq:generitc_FTL}
F_t(x,y,z)=1+\sum_{i,j,k\geq 0, 1 \leq l \leq 4} a_{ijk}^lx^iy^jz^kt^l.
\end{equation}
Then, according to the above paragraph, the \emph{degree} of $F_t(\ux)$ is the integer
$$
d=\max \{i+j+k-l \mid a_{ijk}^l \neq 0\}.
$$
\end{df}

\begin{rem} Most of the above axioms are easy to translate in terms of coefficients.
 Using the notation \eqref{eq:generitc_FTL}, one obtains the following translations:
\begin{enumerate}
\item \textit{Neutral element}:
 $a_{i00}^l=\begin{cases}
1 & i=l=4, \\
2(1-\epsilon) & i=l=1,3, \\
2(1-2\epsilon) & i=l=2, \\
0 & \text{otherwise.}
\end{cases}$
\item \textit{Semi-neutral element}: For all $n \geq 0$, $\sum_{i+j=n} a^4_{ij0}=0$.
\item \textit{Symmetry}: for all $l$, $a_{ijk}^l$ is invariant under permutations of the triple $(i,j,k)$.
\item[(5)] \textit{$\epsilon$-Linearity}. The relation $(1+\epsilon)a_{ijk}^l=0$ holds
 whenever $l$ and one of the $i,j,k$ do not have the same parity.
\end{enumerate}
The associativity axiom is computationaly very involved and we refer the reader to \cite[Appendix]{CDFH}
 for a computer-based calculation.
\end{rem}

\begin{ex}\label{ex:FTP_low_complexity}
 The following FTL are examples with bounded degree. As such, they are respective analogues
 of the additive and multiplicative formal group laws (FGL for short).
\begin{enumerate}
\item We have an FTL of \emph{degree $0$} with coefficients in $\ZZe$ whose non-zero coefficients are:
$$
\begin{array}{llll}
a^1_{100}=2(1-\epsilon) & & & \\
a^2_{200}=2(1-2\epsilon) & a^2_{110}=2(1-\epsilon) & & \\
a^3_{300}=2(1-\epsilon) & a^3_{210}=-2(1-\epsilon) & a^3_{111}=8(2-3\epsilon) & \\
a^4_{400}=1 & a^4_{310}=-2(1-\epsilon) & a^4_{220}=2(1-2\epsilon) & a^4_{211}=2(1-\epsilon).
\end{array}
$$
After inverting $2$, this is one of the two universal FTL of degree $0$, and even $\leq 0$: see \cite[Th. 3.1.12]{CDFH}.
 The other FTL of degree $0$ is given by the same coefficients except that $a^3_{111}=8(3-2\epsilon)$.
 Note that the universality statement tells us that these coefficients will appear in any FTL.
\item We have an FTL of \emph{degree $2$}, having parameters $\tau$ and $\gamma$,
 \emph{i.e.} with coefficients in the polynomial ring
 $\ZZe^{mul}:=\ZZe[\tau,\gamma^{\pm 1}]/\langle \tau^2-2(1-\epsilon)\gamma, (1+\epsilon)\tau\rangle$,
 whose non-zero coefficients are, in addition to the ones appearing in the preceding law:
$$
\begin{array}{llll}
a^1_{110}=\tau\gamma^{-1} & a^1_{111}=\gamma^{-1} & & \\
a^2_{210}=2\tau\gamma^{-1} & a^2_{111}=-3\tau\gamma^{-1} & a^2_{220}=\gamma^{-1} & \\
a^3_{310}=\tau\gamma^{-1} & a^2_{220}=-2\tau\gamma^{-1} & a^3_{211}=3\tau\gamma^{-1} & a^3_{220}=\gamma^{-1} \\
a^4_{311}=-\tau\gamma^{-1} & a^4_{221}=2\tau\gamma^{-1} & a^4_{222}=\gamma^{-1}. &
\end{array}
$$
\end{enumerate}
See \cite[Theorem 6.6]{FH21} for more details. Note that these two FTL have no non-trivial coefficients in degrees less than $0$.
\end{ex}

By analogy with the case of formal group laws as underlined above,
 we will adopt the following definition.
\begin{df}\label{df:addmulFTL}
The FTL with coefficients in $\ZZe$ (resp. $\ZZe^{mul}$) of point (1) of the above example
 will be called the \emph{additive} (resp. \emph{multiplicative}) FTL
 and denoted by $F_t^{add}$ (resp. $F_t^{mul}$).

We will use the same terminology for any FTL obtained by extension of scalars
 from the two previous ones.
\end{df}

\begin{num}
As for FGL, FTL can be organized in a category.
 First note that if $F_t(x,y,z)$ is an FTL with coefficients $a_{ijk}^l$ in a ring $R$,
 and $\varphi:R \rightarrow R'$ is a morphism of rings,
 one obviously obtains a new FTL with coefficients in $R'$ by applying $\varphi$ to all
 coefficients of $F_t(x,y,z)$. We will denote this new FTL by
$$
\varphi_*F_t(x,y,z):=1+\sum_{i,j,k\geq 0, 1 \leq l \leq 4} \varphi(a_{ijk}^l)x^iy^jz^kt^l.
$$
\end{num}
\begin{df}\label{df:morph_FTL}
Let $F_t(x,y,z)$ and $G_t(x,y,z)$ be FTL with coefficients respectively
 in $\ZZe$-algebras $R$ and $R'$. A morphism from $F_t(x,y,z)$ to $G_t(x,y,z)$ is pair $(\varphi,\Theta)$ where $\varphi:R \rightarrow R'$
 is a morphism of $\ZZe$-algebras,
 $\Theta \in R[[x]]$ is a composable power series such that
 the following equality of $(4,3)$-series holds:
$$
\Theta_t\big(\varphi_*F_t(x,y,z)\big)=G_t\big(\Theta_t(x),\Theta_t(y),\Theta_t(y)\big),
$$
where $\Theta_t(x)=1+\Theta(x).t$ as a $(1,1)$-series.
 Note that the two terms of this equality are well defined,
 using the substitution operation of \Cref{num:multi-valued},
 as both $\Theta_t(x)$ and $F_t(x,y,z)$ are composable. Explicitly, the above equality reads as
 \[
 \prod_{i=1}^4(1+\Theta(F^{[i]})t)=G_t\big(\Theta_t(x),\Theta_t(y),\Theta_t(y)\big).
 \]

The composition of these morphisms are defined by composition of rings for $\varphi$,
 and by composition of power series for $\Theta(x)$.
 Therefore an isomorphism of FTL is a pair $(\varphi,\Theta)$ such that $\varphi$ is an isomorphism
 of $\ZZe$-algebras, and $\Theta$ is an invertible power series.
 Such an isomorphism will be called \emph{strict} if $\theta(x)=x+\sum_{i\geq 2}a_ix^i$.

The corresponding category will be denoted by $\FTL$.
\end{df}

\begin{rem}
The category $\FTL$ is a $\ZZe$-linear category, cofibred over the category of $\ZZe$-algebras.
In particular, any morphism $(\varphi,\Theta)$ of FTL can be factorised as
 $(\varphi,\Theta)=(\Id,\Theta)\circ (\varphi,\Id)$.
We can then restrict to the case where FTL have the same coefficient ring.
\end{rem}

\begin{ex}
The \emph{$\epsilon$-linearity axiom} for an FTL $F_t(x,y,z)$ with coefficients in $R$ 
 can be translated by saying that the series $\Theta(x)=-\epsilon.x$, seen as a power series
 with coefficients in $R$,
 induces an automorphism $(\Id_R,\Theta(x))$ of $F_t(x,y,z)$.
\end{ex}

\begin{num}
Let us briefly recall the discussion of \cite[3.1.1]{CDFH}.
 We define the plus and minus parts of $\ZZe$ as the quotient
 rings $\ZZep=\ZZe/(1+\epsilon)$, $\ZZem=\ZZe/(1-\epsilon)$.
 There is an injective map $\ZZe \rightarrow \ZZep \times \ZZem$
 which is an isomorphism after inverting $2$.
 The same definition applies to an arbitrary $\ZZe$-algebra $R$,
 and $R \rightarrow R^+ \times R^-$ is injective if $R/\ZZe$ is flat
 and bijective if $2 \in R^\times$.

Consequently, given any FTL $(R,\F_t(x,y,z))$,
 one defines its plus and minus parts by base change 
 respectively to $R^+/R$ and $R^-/R$.
 Equivalently, the FTLs $\big(R^+,F_t^+(x,y,z)\big)$ and $\big(R^-,F_t^-(x,y,z)\big)$
 are respectively obtained by specializing the symbol $\epsilon$ to $(-1)$ and $(+1)$.
\end{num}

\begin{ex}
The minus part of the additive formal group law can be written explicitly:
$$
F_t^{add-}(x,y,z)=1-2x^2.t^2-8xyz.t^3+(x^4-2x^2y^2).t^4.
$$
\end{ex}

\subsection{$\Sp$-oriented theories}\label{sec:sp_orientaiton}

\begin{num}\label{num:sp-grassmanian}
In this section, we recall the links between $\Sp$-oriented ring spectra and symplectic cobordism, following Panin and Walter.

We first start with the construction of symplectic cobordism.
 Recall that Voevodsky introduced in \cite{VoeMilnor} the algebraic cobordism spectrum $\MGL_S$
 over a base scheme $S$, as the algebraic analogue of complex cobordism.
 As in the classical algebraic topological case,
 it is defined as the colimits of the Thom spaces of the universal rank $n$ vector bundle
 $\gamma_n$ on the grassmanian scheme $\Gr_n$ of $n$-vector subspaces of the infinite dimension affine
 space:
$$
\MGL_S=\colim_{n>0} \Omega^n \Th(\gamma_n).
$$
This is a (commutative) ring spectrum, in the strict sense
 (see \cite{PPR} for the construction in terms of symmetric spectra,
 and \cite[\textsection 16]{BH} for that in terms of monoidal $\infty$-categories).
 Recall lastly that thanks to a theorem of Morel and Voevodsky,
 $\Gr_n$ is weakly equivalent in the $\AA^1$-homotopy category to the classifying space $\BGL_n$
 of the general linear group $\GL_n$.

Panin and Walter introduced in \cite{PWcobord} the analogous construction
 by replacing the family $\GL_n$ by that of the symplectic group $\Sp_{2n}$,
 following again classical algebraic topology. The main difference is that
 one naturally produces a $(\PP^1)^{\wedge 2}$-spectrum.

First, the classifying space $\BSp_{2n}$ also admits a geometrical model,
 the hyperbolic (or ``quaternionic'') Grassmannian.
 To fix notations, we consider the standard symplectic form $h$ on $\AA^2_S$,
 and denote by $\omega_2$ the corresponding symplectic bundle.
 We also set $\omega_{2n}=\omega_{2n-2}\perp \omega_2$ for any $n\geq 2$.
 For any integer $0\leq r\leq n$, we can consider the \emph{hyperbolic Grassmannian} $\HGr_S(2r,\omega_{2n})$ defined to be the open subscheme of the Grassmannian $\Gr_S(2r,2n)$ of $2r$-subvector bundles of $\AA^{2n}$
 on which the restriction of $\omega_{2n}$ is non degenerate.
 It satisfies the following compatibilities:
\begin{enumerate}
\item We have a canonical embedding $\HGr_S(2r,\omega_{2n})\to \HGr_S(2r,\omega_{2n+2})$
 which classifies flags of vector bundles of the form $U_{2r}\subset p^*\omega_{2n}\subset  p^*\omega_{2n+2}$.
\item We have a canonical embedding $\HGr_S(2r,\omega_{2n})\to \HGr_S(2r+2,\omega_{2n+2})$ which classifies flags of vector bundles
 $U_{2r}\perp p^*\omega_2\subset p^*\omega_{2n}\perp p^*\omega_{2}$.
\end{enumerate}

Then we get an $\AA^1$-weak equivalence of pointed spaces:
$$
\BSp_{2r}=\colim_{n>0} \HGr_S(2r,\omega_{2n}).
$$
We also consider the \emph{hyperbolic projective space} $\HP^n_S=\HGr_S(2,\omega_{2n+2})$,
 of dimension $2n$ over $S$, and $\HP^\infty_S=\colim_{n>0} \HGr_S(2,\omega_{2n+2})$,
 so that $\HP^\infty_S=\BSp_2$. Note that $\HP^1_S \simeq (\PP^1)^{\wedge 2}$, the two-fold smash-power
 of the Tate sphere.

We then define $\MSp_{2n}=\Th(\cU_{2n})$, as a space over $S$,
 $\cU_{2r}$ being the tautological symplectic bundle on $\BSp_{2n}$.
 The second of the above properties allows to define $(\PP^1)^{\wedge 2}$-suspension maps as follows.
 The embedding $f:\HGr_S(2n,\omega_{2n})\to \HGr_S(2n+2,\omega_{2n+2})$ yields a bundle of rank $2n+2$ on the first factor,
 namely $f^*\cU_{2n+2}$ which splits as $f^*\cU_{2n+2}\simeq \cU_{2n}\oplus \omega_2$.
 We obtain a morphism of Thom spaces
\[
\Th(\cU_{2n})\wedge th(\omega_2)\simeq \Th(\cU_{2n}\oplus \omega_2)\simeq \Th(f^*\cU_{2n+2})\to \Th(\cU_{2n+2})
\]
which is compatible with the inclusions of item 1. This yields the desired suspension maps:
\[
\sigma_{2n}:(\PP^1)^{\wedge 2} \wedge \MSp_{2n} \simeq \Th(\omega_2)\wedge \MSp_{2n}\to \MSp_{2n+2},
\]
and therefore a $(\PP^1)^{\wedge 2}$-spectrum, which in $\infty$-categorical terms can be defined as:
\begin{equation}\label{eq:MSp_as_colim}
\MSp_S=\colim_{r>0} \Omega^{2n} \MSp_{2n}.
\end{equation}
To get the ring structure, one considers the rank $2(n+m)$ symplectic bundle
 $\cU_{2n} \boxplus \cU_{2m}$ over $\BSp_{2n} \times_S \BSp_{2m}$,
 obtained by exterior sum. It is classified by a map
 $$f:\BSp_{2n} \times_S \BSp_{2m} \rightarrow \BSp_{2(n+m)}$$ such that
 $f^{-1}(\cU_{2(n+m)})=\cU_{2n} \boxplus \cU_{2m}$. Taking Thom spaces, one deduces:
$$
\Th(\cU_{2n}) \otimes \Th(\cU_{2m})
 \simeq \Th(\cU_{2n} \boxplus \cU_{2m}) \xrightarrow{f'_*} \Th(\cU_{2(n+m)})
$$
as required. This gives a (commutative) ring structure on the spectrum $\MSp_S$,
 which can be lifted to a strict ring structure at the level of symmetric $(\PP^1)^{\wedge 2}$-spectra
 according to \cite{PWcobord}.\footnote{For a more conceptual construction
 of the corresponding $E_\infty$-ring structure, we refer the reader to \cite[\textsection 16]{BH}.
 However, in this paper, we will only need the (commutative) monoid structure in $\SH(S)$
 on $\MSp_S$.}
\end{num}
\begin{df}
The ring spectrum $\MSp_S$ is called the symplectic cobordism spectrum over $S$.
\end{df}

\begin{num}\label{num:symplectic_Thom_Grass_Euler}
 \textit{Thom classes for symplectic cobordism}.
As in topology, the previous ring spectrum is built to possess ``symplectic Thom classes''.
 First note that the identity of $\MSp_{2r}$ corresponds by construction of $\MSp_S$
 to a map:
\[
\tau_r:\Sigma^{\infty}\MSp_{2r}\to \MSp(2r)[4r],
\]
or in other words, an element $\tau_r \in \tilde \MSp^{4r,2r}(\MSp_{2r})$
 in the reduced symplectic cobordism.
 Let $X$ be a smooth $S$-scheme which is affine.
 Then a symplectic bundle $\cE$ of rank $2r$ is classified by a map:
\[
f:X\to \BSp_{2r}
\] 
such that $\cE=f^{-1}(\cU_{2r})$. Taking Thom spaces and composing with $\tau_r$, we get:
\[
Th(\cE)=Th(f^{-1}\cU_{2r}) \xrightarrow{f_*} Th(\cU_{2r})=\MSp_{2r} \xrightarrow{\tau_r} \MSp(2r)[4r]
\]
This is the Thom class $\thom^\MSp(\cE) \in \tilde \MSp^{4r,2r}(\Th(\cE))$
 associated with $\cE$
 with coefficients in $\MSp$. Equivalently, $\thom(\cE)=f^*(\tau_r)$.

Using Jouanolou's trick, this construction can be extended to any smooth $S$-scheme.
 One can check the basic properties of Thom classes, as stated in \cite[Def. 2.1.3, $G=\Sp_*$]{DF3}:
 compatibility with isomorphisms and pullbacks, products, normalisation.
 In other words, $\MSp_S$ is symplectically oriented in the sense of \emph{loc. cit.}
 (see \cite[\textsection 10, Th. 1.1]{PWcobord} for the proof, analogous to the
 topological case).
Let us introduce some terminology for the statement of the next result.
 Let $\E$ be a ring spectrum over $S$.
 When $\E$ admits Thom classes for symplectic bundles satisfying the properties
 mentioned above, we will say that $\E$ admits symplectic Thom classes --- this corresponds
 to being $\Sp$-oriented in the sense of \cite[Def. 2.1.3]{DF3}.
 Recall that one can define the \emph{Euler class} $e(\cE)$ as the image of the Thom class $\thom(\cE)$
 under the canonical map:
$$
\E^{2r,r}(\Th(\cE)) \rightarrow \E^{2r,r}(\cE) \xrightarrow{\sim} \E^{2r,r}(X).
$$
Moreover, when $\E$ admits a Thom structure, one deduces an isomorphism:
\begin{equation}\label{coh_HPinfty}
\E^{n,m}(\HP^\infty_S) \simeq \colim_{n>0} \E^{n,m}(\HP^n_S)
\end{equation}
\emph{i.e.} the $\colim^{(1)}$ vanishes here.
 This follows from the symplectic projective bundle theorem for $\E$
 (see \cite[Th. 2.2.3]{DF3})

An $\MSp$-algebra structure on $\E$ (over $S$) will be a morphism $\varphi:\MSp_S\rightarrow \E$
 of ring spectra.
 We denote by $\Sigma_T^2 1_\E$ the image of the unit $1_\E \in \E^{00}(S)$ of the ring spectrum $\E$
 under the $(\PP^1)^{\wedge 2}$-suspension isomorphism $\E^{00}(S) \simeq \tilde \E^{4,2}(\HP^1_S)$.
\end{num}
\begin{thm}[Panin-Walter]\label{thm:PW_Sp-oriented}
Let $\E$ be a ring spectrum over $S$. Then the following data are in bijective
 correspondance:
\begin{enumerate}
\item the cohomology classes $b \in \tilde \E^{4,2}(\HP^\infty_S)$ such
 that $b|_{\HP^1_S}=\Sigma_T^2 1_\E$.
\item The $\Sp$-orientation on $\E$ over $S$ in the sense of \Cref{df:G-orientation}.
\item The $\MSp$-algebra structure on $\E$ over $S$: $\varphi:\MSp_S \rightarrow \E$
\end{enumerate}
Moreover, the bijections are constructed as follows:
\begin{itemize}
\item (3) $\rightarrow$ (2): for any symplectic bundle $\cU$ over a smooth $S$-scheme $X$,
 one defines $\thom^\E(\cU)=\varphi_*(\thom^\MSp(\cU))$.
\item (2) $\rightarrow$ (1): one defines: $b_n=\varphi_*(e(\cV_{2})) \in \E^{4,2}(\HP^n_S)$,
 where $\cV_{2}$ is the tautological rank $2$ symplectic bundle over $\HP^n_S$.
 Using formula \eqref{coh_HPinfty}, this defines a class $b \in \E^{4,2}(\HP^\infty_S)$.
\end{itemize}
\end{thm}
For the proof of that theorem, we refer the reader to \cite[Th. 1.1]{PWcobord}.
The condition (1) being much more cost effective,
 and analogous to that of $\GL$-oriented ring spectra,
 we chose to use it as an equivalent replacement of \Cref{df:G-orientation} in the case $G=\Sp$
 (or also Definition 2.3.1 of \cite{DF3}).
\begin{df}\label{df:Sp-orientation}
A symplectic orientation of a ring spectrum $\E$ will be a class $b$
 as in point (1) of the previous proposition.
 We also say that $(\E,b)$ is a symplectically oriented ring spectrum over $S$.

A morphism of symplectically oriented ring spectra
 $\varphi:(\E,b^\E) \rightarrow (\F,b^\F)$ is a morphism of ring spectra
 such that $\varphi_*(b^\E)=b^\F$.
\end{df}
Note that by the above theorem, given such a morphism $\varphi$,
 one immediately deduces for an $\Sp$-bundle $\cU$ over a smooth $S$-scheme $X$,
 the relation
\begin{equation}\label{eq:Sp-morphisms&Thom_classes}
\varphi_*(\thom^\E(\cU))=\thom^\F(\cU).
\end{equation}

\begin{rem} The symplectic orientation theory is completely analogous to
 the classical orientation theory, both in $\AA^1$-homotopy (see \cite[Section 2.1]{Deg12}), 
 and in classical topology (see \cite[Definition 4.1.1]{Ra03}):
\begin{enumerate}
\item Note as a byproduct of the above theorem that $\MSp_S$ equipped with its canonical
 orientation $b^\MSp$ (coming from the above construction of Thom classes),
 is the universal symplectically oriented ring spectrum over $S$.
 That is, it is the initial object among symplectically oriented ring spectra over $S$.
\item From the previous theorem, a symplectically oriented ring spectrum is just
 an $\MSp$-algebra, and a morphism of such is simply a morphism of $\MSp$-algebras.
\item Let $f:T \rightarrow S$ be a morphism of schemes.
 Given an $\Sp$-oriented ring spectrum $(\E,b)$ over $S$,
 we immediately deduce an $\Sp$-orientation $f^*(b)$ on $f^*(\E)$ over $T$.
 This will be very useful to reduce arguments to the base scheme.
\end{enumerate}
\end{rem}

As a corollary of the 

\begin{num}\label{num:Sp-Euler}
One can interpret Euler classes in analogy with the $\GL$-oriented case.
 Given any scheme $S$, according to \cite[Prop. 1.16, p. 130]{MV},
 there exists a canonical map:
$$
\mathrm{H}^1(S,\Sp_2)  \rightarrow [S_+,\BSp_2^S]=[S_+,\HP^\infty_S].\footnote{According to \cite[4.1.2]{AHW},
 this map is an isomorphism whenever $S$
 is affine and ind-smooth over a Dedeking ring but we will not use that fact.}
$$
 The left hand-side can be interpreted
 as the set of isomorphisms classes of rank $2$ symplectic bundle $\PicSp(S)$.
 Note that this map is compatible with pullbacks in $S$. Besides, given a smooth $S$-scheme $X$,
 one similarly gets a canonical map:
$$
\mathrm{H}^1(X,\Sp_2) \rightarrow [X_+,\BSp_2^S]=[X_+,\HP^\infty_S].
$$
One should be aware of an essential difference with the $\GL$-case here:
 there is no group structure on $\PicSp(S)$. First, the tensor product of two symplectic bundles
 is an orthogonal bundle. Secondly, even if one considers triple tensor products,
 the rank becomes $8$. The correct algebraic structure will be explored
 through the notion of formal ternary laws.

However, given an $\Sp$-oriented spectrum $(\E,b)$, one gets a streamlined interpretation of the associated
 Euler classes in terms of $b$.
 First one can view $b$ as a map $b:\Sigma^\infty \HP^\infty_S \rightarrow \E(2)[4]$.
 Then given any smooth $S$-scheme $X$, one gets a natural map: 
\begin{align*}
\PicSp(X)=\mathrm{H}^1(X,\Sp_2) &\rightarrow [X_+,\HP^\infty_S] \xrightarrow{\Sigma^\infty}
 [\Sigma^\infty X_+,\Sigma^\infty \HP^\infty_S]^{st} \\
& \xrightarrow{b_*} [\Sigma^\infty X_+,\E(2)[4]]^{st}=\E^{4,2}(X).
\end{align*}
One easily checks that this is the \emph{Euler class map}, which to the class of a rank $2$
 symplectic bundle $\cU$ associates $e(\cU)$ whose construction was recalled in
 \ref{num:symplectic_Thom_Grass_Euler}. In this case, this coincides with the first
 Borel class $b_1(\cU)$ which will be recalled in the next paragraph.
\end{num}

\begin{rem}
In fact, one easily checks that the data of an $\Sp$-orientation
 on $\E$ over $S$ is equivalent to the data for any smooth $S$-scheme $X$ of
$$
b_1:\PicSp(X) \rightarrow \E^{4,2}(X)
$$
such that:
\begin{enumerate}
\item $b_1$ is natural with respect to the pullback on both functors
\item for the tautological symplectic bundle $\cV$ on $\HP^1_S$,
 one has $b_1(\cV)=(1,0)$ via the identification $\E^{4,2}(\HP^1_S)=\E^{0,0}(S) \oplus \E^{4,2}(S)$.
\end{enumerate}
\end{rem}

\begin{num}\label{num:Sp-proj-bdl}
As already said, a symplectically oriented ring spectrum $(\E,b)$ automatically satisfies
 the symplectic projective bundle theorem, \cite[Th. 2.2.3]{DF3}.
 In particular, putting $\E^{**}=\E^{**}(S)$ as a bigraded ring,
 one obtains an isomorphism of $\E^{**}$-algebra
$$
\E^{**} [[t]] \rightarrow \E^{**}(\HP^\infty_S), t \rightarrow b
$$
which is bihomogeneous, of bidegree $(0,0)$,
 when one decides that $t$ on the left hand-side has bidegree $(4,2)$.
 As a corollary of the previous theorem one therefore deduces:
\end{num}
\begin{cor}\label{cor:chg_Sp-orient}
Let $(\E,b)$ be a symplectically oriented ring spectrum over $S$.
 Then the symplectic orientations $b'$ of $\E$ are in one-to-one correspondence
 with the local parameters of bidegree $(4,2)$ of the power series ring $\E^{**}[[t]]$,
 \emph{i.e.} with the power series of the form
$$
\phi(t)=t+a_2t^2+a_3t^3+\hdots
$$
where $a_n \in \E^{4-4n,2-2n}$.  \\
 The bijection is given by $\phi \mapsto \phi(b)$ in
 $\E^{**}(\HP^\infty_S) \simeq \E^{**}[[b]]$.
\end{cor}
Note moreover that when one has a class $b'$ as above associated
 with a power series $\phi$, then for any rank $2$ symplectic
 bundle $\cU$ over a smooth $S$-scheme $X$, one gets:
\begin{equation}\label{eq:firstBorel&chg_orient}
b'_1(\cU)=\phi(b_1(\cU)) \in \E^{4,2}(X),
\end{equation}
where the right hand-side makes sense as $b_1(\cU)$ is nilpotent
 (see \cite{DF3}).

\subsection{Virtual symplectic Thom classes}\label{sec:virt_symplect}

According to \Cref{prop:stable_G-orientations},
 symplectically oriented ring spectra admits virtual symplectic Thom classes,
 associated with stable symplectic bundles. However, hermitian K-theory of symplectic bundles
 is finer as it involves the equivalence relation modulo metabolic spaces (see e.g. \cite{Kneb}).
 We will show in this subsection that symplectic Thom classes are compatible with
 the metabolic equivalence relation.

\begin{num}
We first give a quick recollection on the hermitian $Q$-construction and the forgetful functor, following \cite{Schlichting09}.

We start with a triple $(\mathcal{E},*,\eta)$, where $\mathcal{E}$ is an exact category, $(-)^*$ is a duality (i.e. an exact functor $\mathcal E^{\mathrm{op}}\to \mathcal{E}$) and a canonical isomorphism $\eta:1\to (-)^{**}$ satisfying an extra condition \cite[Definition 2.1]{Schlichting09}. Recall that a symmetric form is a pair $(X,\varphi)$, where $X$ is an object of $\mathcal{E}$ and $\varphi:X\to X^*$ is a morphism which is symmetric in the sense that the diagram
\[
\xymatrix{X\ar[r]^-{\varphi}\ar[d]_-{\eta} & X^*\ar@{=}[d] \\
X^{**}\ar[r]_-{\varphi*} & X^*}
\]
is commutative. A symmetric form is called \emph{symmetric space} if $\varphi$ is an isomorphism. A morphism of symmetric forms $f:(X,\varphi)\to (Y,\psi)$ is a morphism $f:X\to Y$ such that the diagram
\[
\xymatrix{X\ar[r]^-f\ar[d]_-{\varphi} & Y\ar[d]^-\psi \\
X^* & Y^*\ar[l]^/-2pt/{f^*}}
\]
is commutative. Let now $(X,\varphi)$ be a symmetric space (i.e. $\varphi$ is an isomorphism). Let $i:L\to X$ be an admissible monomorphism. The orthogonal of $L$, denoted by $L^{\perp}$ is the kernel of $X\xrightarrow{i^*\varphi} L^*$, which is also admissible. We say that $L$ is totally isotropic if the composite $i^*\varphi i$ is trivial and the associated morphism $L\to L^{\perp}$ is admissible. In the special case where $L\to L^{\perp}$ is an isomorphism, we say that $L$ is a Lagrangian.

We now form the $Q$-construction in this context, following \cite[\S 4.1]{Schlichting09}. We let $Q^h\mathcal E$ be the category whose objects are symmetric spaces $(X,\varphi)$ and whose morphisms are equivalences classes of diagrams of the form
\[
\xymatrix{X & U\ar[l]_-p\ar[r]^-i & Y}
\] 
where $p$ is an admissible epimorphism, $i$ is an admissible monomorphism, and the diagram
\[
\xymatrix{U\ar[r]^-i\ar[d]_-p & Y\ar[d]^-{i^*\psi} \\
X\ar[r]_-{p^*\varphi} & U^*}
\]
is bicartesian. Two such diagrams $(U,i,p)$ and $(U',i',p')$ are equivalent if there exists an isomorphim $g:U\to U'$ such that $i'g=i$ and $p'g=p$. The composition is obtained as in Quillen's original $Q$-construction, and there is thus an obvious functor
\[
Q^h(\mathcal E,*,\eta)\to Q\mathcal E
\]
which induces a morphism on the nerves of the relevant categories. To simplify the notation, we will now write $\mathcal E$ for $(\mathcal E,*,\eta)$.
 Following Schlichting, we denote by $\mathrm{GW}(\mathcal E)$ the homotopy fiber (with respect to a zero object in $\mathcal E$) of the
 induced map of (Kan) simplicial sets:
\[
N(Q^h\mathcal E) \to N(Q\mathcal E).
\]
As a special case of an exact category with duality, we may consider the category $H\mathcal E=\mathcal E\times \mathcal E^{\mathrm{op}}$. There is an obvious duality, defined on objects as $(A,B)^{\sharp}=(B,A)$ and natural isomorphism $1\to {}^{\sharp\sharp}$ the identity. If $(\mathcal E,*,\eta)$ is an exact category with duality, we define a functor 
\[
F:(\mathcal E,*,\eta)\to H\mathcal E
\] 
on objects by $A\mapsto (A,A^*)$ and on morphisms by $f\mapsto (f,f^*)$. There is a natural isomorphism
\[
F(A^*)=(A^*,A^{**})\to F(A)^{\sharp}=(A^*,A)
\]
given by the pair $(\mathrm{Id}_A,\eta_A)$. Consequently, we obtain a commutative diagram
\[
\xymatrix{Q^h\mathcal E\ar[r]\ar[d] &  Q\mathcal E\ar[d]\\
Q^hH\mathcal E\ar[r] &  QH\mathcal E
}
\] 
Now, there is an equivalence of categories $Q^hH\mathcal E\to Q\mathcal E$ given on objects by $(A,B)\mapsto A$ and on morphisms by $(f,g)\mapsto f$ (see \cite[Remark 4.3]{Schlichting09}).
 It follows that the homotopy fiber of 
\[
N(Q^hH\mathcal E) \to N(QH\mathcal E) \simeq N(Q\mathcal E) \times N(Q\mathcal E)
\]
is $\Omega N(Q\mathcal E):=\mathrm{K}(\mathcal E)$ (see \cite[Proposition 4.7]{Schlichting09}).
 Consequently, we obtain the \emph{forgetful map}
\[
F:\mathrm{GW}(\mathcal E)\to \mathrm{K}(\mathcal E).
\]
Our main example here is the category $\mathcal E=\mathcal{V}(X)$ of vector bundles over a scheme $X$, with usual duality $V^*=\mathrm{Hom}(V,\OO_X)$ and canonical isomorphism $\eta=-\varpi$, where $\varpi:V\to V^{**}$ is the usual identification (given by the evaluation). In the usual notation, this gives the \emph{forgetful map}
\[
F:\mathrm{KSp}(X)\to  \mathrm{K}(X).
\]
We can view $F$ as a morphism of $\infty$-categories, and consider the induced map on the associated homotopy groupoids:
$$
f_\Sp:\uKSp(X) \rightarrow \uK(X).
$$
The pseudo-functoriality of the above construction actually gives
 a stable structure of degree $2$ in the sense of \Cref{df:stable_structure}.
\end{num}
\begin{df}\label{df:fine-sp-structure}
We will call the stable structure $f_\Sp:\uKSp \rightarrow \uK$ over $\base$ defined above
 the \emph{fine symplectic stable structure}.
\end{df}

\begin{num}
Note that, by construction, one gets a commutative diagram:
$$
\xymatrix@R=-2pt@C=40pt{
\uK^{\Sp}(X)\ar^{\sigma_\Sp}[rd]\ar_{\kappa_X}[dd] & \\
& \uK(X) \\
\uKSp(X)\ar_{f_\Sp}[ru]
}
$$
where $\sigma_\Sp$ is stable structure associated with
 the symplectic group in \Cref{ex:stable_G-structure}.
 Moreover, $\kappa_X$ is compatible with pullbacks in $X$, or in other words,
 corresponds to a morphism of $\base$-fibred categories in Picard groupoid
 (see \Cref{num:fibred_cat&K}).
 The next lemma says that Zariski locally, both stable structures coincide.
\end{num}
\begin{lm}
If $X$ an affine scheme, then the functor $\kappa_X$ is an equivalence of categories.
\end{lm}
\begin{proof}
Indeed, by construction, the map induced by $\pi$ on the respective sets of isomorphism classes coincides
 with the quotient map
$$
\mathrm K^{\Sp}_0(X) \rightarrow \mathrm K^{\Sp}_0(X)/\mathcal I
$$
where $\mathcal I$ is the ideal generated by metabolic symplectic spaces.
 Then the result follows as in \cite[\textsection 4, Prop. 1]{Kneb}.
\end{proof}

We have seen in \Cref{prop:stable_G-orientations} that,
 on a given ring spectrum $\E$ over $S$,
 the data of a  symplectic orientations on $\E$ over $S$-schemes
 is equivalent to that of a $\sigma_\Sp$-orientations on $\E$.
 Thanks to the preceding result, we can make a finer statement.
\begin{prop}\label{prop:sp-or&KSp-or}
Let $\E$ be a ring spectrum over $S$.
 Then there are bijective correspondences between:
\begin{enumerate}
\item The symplectic orientations on $\E$ over $S$ in the sense of \Cref{df:Sp-orientation}.
\item The $\sigma_\Sp$-orientations on $\E$ over $S$ in the sense of \Cref{df:sigma-orientation}.
\item The $f_\Sp$-orientations on $\E$ over $S$ in the sense of \Cref{df:sigma-orientation}.
\end{enumerate}
Moreover, the map between (3) and (2) is simply obtained by composing with
 the pseudo-natural transformation $\kappa:\uK^\Sp \rightarrow \uKSp$.
\end{prop}
\begin{proof}
In view of \Cref{prop:stable_G-orientations}, we need only proving the equivalence
 between datas (2) and (3). Recall that, according to Jouanolou's trick,
 given any scheme $X$, there exists a (Zariski-local) $\AA^1$-torsor
 $p:X' \rightarrow X$
 such that $X'$ is affine.
 Then the result follows from the preceding lemma applied to $X'$,
 and the fact the pullback map: $p^*:\E\modd_X \rightarrow \E\modd_{X'}$
 is an equivalence of categories.
\end{proof}

\begin{num}\textit{Virtual symplectic Thom classes}.
Following \Cref{num:Thom_iso_explicit},
 we now make explicit the data of the $f_\Sp$-orientation
 associated with an $\Sp$-oriented ring spectrum $(\E,\thom)$ over $S$.

Given a stable symplectic bundle $\omega \in \uKSp(X)$,
 of rank $r:=\rk(f_\Sp(\omega))$, and with Thom space $\Th(\omega):=\Th(f_\Sp(\omega))$ in $\SH(X)$,
 one gets a \emph{virtual Thom class}:
$$
\thom(\omega) \in \E^{2r,r}(\Th(\omega))
$$
such that the natural map:
$$
\E^{**}(X) \rightarrow \E^{*+2r,*+r}(\Th(\omega)), \lambda \mapsto \lambda.\thom(\omega)
$$
is an isomorphism of bigraded $\E^{**}(X)$-modules.

Moreover, this class is compatible with isomorphisms in $\omega$, pullbacks in $X$
 in an obvious sense. Let us spell out explicitly the compatibility with 
 sums of two stable symplectic bundles $\omega, \omega'$ of respective ranks $r$ and $r'$.
 One gets a natural product map:
$$
\mu:\E^{2r,r}(\Th(\omega)) \otimes_\ZZ \E^{2r',r'}(\Th(\omega'))
  \rightarrow \E^{2r+2r',r+r'}(\Th(\omega+\omega')).
$$
Then the following relation holds: $\mu(\thom(\omega) \otimes \thom(\omega'))=\thom(\omega+\omega')$.
 We can equivalently write this as:
$$
\thom(\omega).\thom(\omega')=\thom(\omega+\omega').
$$
One deduces that $\mu$ induces an isomorphism of bigraded $\E^{**}(X)$-modules:
$$
\E^{**}(\Th(\omega)) \otimes_{\E^{**}(X)} \E^{*}(\Th(\omega'))
 \xrightarrow \sim \E^{**}(\Th(\omega+\omega')).
$$
Finally the virtual Thom class extends the Thom class associated with
 a symplectic vector bundle. Given a symplectic bundle $(V,\psi)$,
 one associates a virtual symplectic bundle $[V,\psi] \in \uKSp(X)$
 such that $f_\Sp([V,\psi])=[V]$ the stable vector bundles assocaited with $V$.
 One further gets an identification $\Th([V,\psi])=\Th([V])=\Th(V)$,
 by construction of the Thom space of the virtual bundle $[V]$.
 Given this identification, one has in the group $\E^{2r,r}(\Th(V))$ with $r=\rk(V)$:
$$
\thom([V,\psi])=\thom(V,\psi).
$$
\end{num}

\subsection{Higher Borel classes and $\GW$-linearity}

\begin{num}\textit{Borel classes}. \label{num:Borel_Sp}
Recall that to a symplectic orientation $b$ of $\E$ is associated by Panin and Walter
 a theory of Borel classes $b_n(\cU) \in \E^{4n,2n}(X)$, $\cU/X$ a symplectic vector bundle
 over a smooth $S$-scheme $X$. Besides, the Euler class introduced in Paragraph
 \ref{num:symplectic_Thom_Grass_Euler}
 satisfies the relation: $e(\cU)=b_r(\cU)$ when $\cU$ has rank $2r$.
 We refer the reader to \cite[Section 2.2]{DF3}.
 Following the classical scheme, one defines the total Borel class associated with
 $\cU$ as the following polynomial expression in $t$:
\begin{equation}\label{eq:total_borel}
b_t(\cU)=\sum_{n=0}^r b_n(\cU).t^r.
\end{equation}
Note moreover that, using relation \eqref{eq:Sp-morphisms&Thom_classes} and the construction
 of Borel classes, one deduces further that for a morphism $\varphi$ as in \Cref{df:Sp-orientation},
 we have the following relation in $\F^{4n,2n}(X)$:
\begin{equation}\label{eq:Sp-morphisms&Borel_classes}
\varphi_*(b_n^\E(\cU))=b_n^\F(\cU).
\end{equation}
\end{num}

\begin{num}\label{num:sp-splitting-principle}
\textit{The symplectic splitting principle}. As for Chern classes,
 one has a splitting principle for Borel classes.
 Indeed, given any symplectic bundle $\cV$ over a scheme $X$, there exists an affine morphism $p:X' \rightarrow X$
 inducing a monomorphism $p^*:\E^{**}(X) \rightarrow \E^{**}(X')$ with the property that $p^{-1}(\cV)$ splits
 as a direct sum of rank $2$ symplectic bundles: $p^{-1}(\cV)=\mathcal X_1 \perp \hdots \perp \mathcal X_n$.

One defines the {\it Borel roots} of $\cV$ as the Borel classes $\xi_i=b_1(\mathcal X_i)$ so that
 by the Whitney sum formula \cite[Theorem 10.5]{PWQuaternionic} the Borel classes of $\cV$ are the elementary symmetric polynomials
 in the variables $\xi_i$.

As in the classical case, one can compute universal formulas involving Borel classes of $\cV$ by introducing Borel roots
 $\xi_i$, which reduces to rank $2$ symplectic bundles, compute as if $\cV$ was completely split, and
 then express the resulting formula in terms of
 the elementary symmetric polynomials in the $\xi_i$. 
\end{num}

\begin{num}
We now give a complement to the theory of Borel classes which was known previously only under the stronger assumption of the existence of an $\SL$-orientation
(see \cite[2.2.8]{DF3}). To start with, let us recall the following definition.

\begin{df}
Let $\lambda\in \OO(S)^\times$. We denote by $\langle \lambda\rangle\in \mathrm{End}_{\mathcal{SH}(S)}(\mathbf{1})$ the endomorphism induced by the pointed morphism $\lambda:\PP^1_{S}\to \PP^1_{S}$ given by $[x:y]\mapsto [\lambda x:y]$.
\end{df}

As a result, $\lambda$ acts on the set $[X,Y]_{\mathcal{SH}(S)}$ for any objects $X,Y$ via 
\[
f\mapsto f\wedge \langle \lambda\rangle:X\wedge 1\to Y\wedge 1.
\]
By abuse of notation, we still denote by $\langle\lambda\rangle$ the endomorphism $\mathrm{Id}_X\wedge \langle\lambda\rangle$. The aim of this section is to prove the following theorem.

\begin{thm}\label{thm:epsilon}
Let $(\E,b)$ be an $\Sp$-oriented ring spectrum, and thom classes $\thom$.
 Then for any unit $\lambda \in \mathcal O_S(X)^\times$, 
 and any $\Sp$-vector bundle $(V,\psi)$ over $X$, the following relation holds:
\[
\thom(V,\lambda\psi)=\langle \lambda \rangle \thom(V,\psi).
\]
\end{thm}

As a corollary, we obtain the following result whose proof is the same as \cite[2.2.8]{DF3}.

\begin{cor}\label{cor:epsilonlinearity}
Let $(\E,b)$ be an $\Sp$-oriented ring spectrum and let $\lambda \in \mathcal O_X(X)^\times$ be a unit. Then
\[
b_i(V,\lambda\psi)=\langle \lambda^i\rangle b_i(V,\psi)
\]
for any $\Sp$-vector bundle $(V,\psi)$ over $X$ and any $i\in\NN$.
\end{cor}

The proof of the theorem will occupy the remaining of this section.
 We first reduces to the case $\E=\MSp_S$ (using \Cref{thm:PW_Sp-oriented})
 and to the case where of the tautological symplectic bundle $(U,\varphi)$ on $\mathrm{BSp}_2$
 seen over the base scheme $S$. Given $\lambda \in \cO_S(S)^\times$, the relation to prove is now:
\[
\mathrm{th}(U,\lambda\varphi)=\langle \lambda\rangle \mathrm{th}(U,\varphi).
\]
The proof of the theorem will require some elaborate computations with the hyperbolic Grassmannians. Recall that $\HP^n_S=\HGr_S(2,\omega_{2n+2})$ and that there are embeddings
\[
i_n:\HP^n_S\to \HP^{n+1}_S
\]
for any $n\in\NN$. In particular, $\HP^0_S=S$ and we will consider $\HP^{n}_S$ as pointed by the image of $\HP^0_S=S$ under the above map. In the sequel, we will suppress the subscript $S$ from the notation to slightly lighten the expressions. 

There is an interesting cell structure on $\HP^n$ for any $n\in\NN$, as explained by Panin-Walter in \cite{PWQuaternionic}. There is a decomposition
\[
\HP^n=X_0\sqcup X_2\sqcup \ldots\sqcup X_{2n}
\]
where each $X_i$ is of codimension $2i$, smooth and quasi-affine. The closure $\overline X_{2i}$ of $X_{2i}$ is of the form $\overline X_{2i}=X_{2i}\sqcup X_{2i+2}\sqcup\ldots\sqcup X_{2n}$ and there is a projection
\[
q_i:\overline X_{2i}\to \HP^{n-i}
\]
identifying $\overline X_{2i}$ with the total space of the universal bundle $\mathcal U_2$ over $\HP^{n-i}$. The composite of the zero section $\HP^{n-i}\to \overline X_{2i}$ with $\overline X_{2i}\to \HP^{n}$ is the composite of embeddings $i_{n-1}\circ\ldots\circ i_{n-i}$ described above. In more detail, write $e_1,\ldots,e_{2n+2}$ for the usual basis of $\OO_S^{2n+2}$. Observe that $E_i=\mathrm{Vect}(e_{2j-1},j=1,\ldots,i)$ is a totally isotropic subspace, in the sense that $E_i\subset E_i^\perp$, and that we have a filtration of $\OO_S^{2n+2}$ of the form
\[
0:=E_0\subset E_1\subset E_{n+1}=E_{n+1}^\perp\subset \ldots\subset E_1^\perp\subset 0^\perp=\OO_S^{2n+2}
\]
We have for any $i=0,\ldots,n$ an embedding $\mathrm{Gr}(2,E_i^{\perp})\subset \mathrm{Gr}(2,\OO_S^{2n+2})$ and we set $\overline X_{2i}:=\mathrm{Gr}(2,E_i^{\perp})\cap \HP^{n+1}$. In terms of the universal sequence, we can interpret these closed subschemes as follows. Consider the projection $p_2:\OO_S^{2n+2}\to \OO_S$ on the second factor. As $E_1^\perp=\langle e_1,e_3,e_4,\ldots,e_{2n+2}\rangle$ we see that $\overline X_2$ is the closed subscheme where the section
\[
\mathcal{U}\xrightarrow{i} \OO^{2n+2}\xrightarrow{p_2} \OO
\]
vanishes. In particular, the normal bundle to $\overline X_2$ in $\HP^n$ is $(\mathcal{U}^\vee)_{\vert \overline X_2}$. Besides, if $p_1:\OO_S^{2n+2}\to \OO_S$ is the projection to the first factor, it is easy to check that if the above composite vanishes (i.e. we are on $\overline X_2$) then the composite 
\[
\mathcal{U}\xrightarrow{i} \OO^{2n+2}\xrightarrow{p_1} \OO
\]
is arbitrary (i.e. the symplectic form on $\mathcal{U}$ is automatically nondegenerate). Observing that the restriction of $\mathcal{U}$ to $\HP^{n-1}$ is the universal bundle, this realizes $\overline X_2$ as the total space of $\mathcal{U}^\vee$ over $\HP^{n-1}$. We also deduce from this that the normal bundle to $\HP^{n-1}$ in $\HP^{n}$ is $\mathcal{U}^\vee\oplus \mathcal{U}^\vee$. To conclude, let us say that the filtration $\HP^{n-1}=X^\prime_0\sqcup\ldots \sqcup X^\prime_{2n-2}$ pulls back to the filtration $X_2\sqcup\ldots \sqcup X_{2n}$ under the projection $q_{n-1}:\overline X_2\to \HP^{n-1}$. 

Let us now concentrate on the geometry of $X_{2i}$ for $i=0,\ldots,n$, starting with $X_0$. By definition, $X_0$ is the open subscheme over which the composite 
\[
\pi:\mathcal{U}\xrightarrow{i} \OO^{2n+2}\xrightarrow{p_2} \OO
\]
is non trivial. This yields a nowhere vanishing section of $\mathcal{U}$ with kernel $\OO$ (as $\mathcal U$ has trivial determinant). We set $Y=\{f\in U\vert \pi(y)=1\}$. This is a $\ker(\pi)$-torsor over $X_0$, i.e. an $\AA^1$-torsor over $X_0$ of the form $\mu:Y\to X_0$. We now work over $Y$. It is easy to see that there exists a unique global section $e$ of $\mu^*\mathcal{U}$ which has the property that the symplectic form on $\OO^2$ induced by the isomorphism $j:\OO^2\simeq \mathcal{U}$ obtained via $f,e$ is the usual hyperbolic one. 

We may then write $i:\mathcal{U}\to \OO^{2n+2}$ as a matrix
\[
\begin{pmatrix} 1 & 0 \\ a & b \\ v & w\end{pmatrix}
\]
where $a,b\in \AA^1$ and $v,w\in \AA^{2n}$. The restriction of $\sigma_{2n+2}$ under this map is the matrix
\[
\begin{pmatrix} 0 & b+v^t\sigma_{2n}w \\ -b+w^t\sigma_{2n}v & 0\end{pmatrix}
\]
and it follows that $b+v^t\sigma_{2n}w=1$. Now, the action of $\mathbb{G}_a$ on $(a,v,w)$ is of the form $\mathbb{G}_a\times \AA^1\times \AA^{2n}\times \AA^{2n}\to \AA^1\times \AA^{2n}\times \AA^{2n}$  given by 
\[
\lambda\cdot \begin{pmatrix} 1 & 0 \\ a & b \\ v & w\end{pmatrix}=\begin{pmatrix} 1 & 0 \\ a & b \\ v & w\end{pmatrix}\begin{pmatrix} 1 & 0 \\ \lambda & 1\end{pmatrix},
\]
i.e. $\lambda\cdot (a,v,w)=(a+b\lambda,v+\lambda w,w)=(a+\lambda(1-v^t\sigma_{2n}w),v+\lambda w,w)$. We may extend this description to the other open cells $X_{2i}$ (or use inductively the fact that the pull-back of $X^\prime_0\sqcup\ldots \sqcup X^\prime_{2n-2}$ is $X_2\sqcup\ldots \sqcup X_{2n}$) to obtain the following result (\cite[Theorem 3.4]{PWQuaternionic}).

\begin{lm}
The cell $X_{2i}$ is the quotient of $\AA^1\times \AA^i\times \AA^{2n-2i}\times \AA^i\times \AA^{2n-2i}$ under the free action of $\mathbb{G}_a$ given by
\[
\lambda\cdot (a,v_1,v_2,w_1,w_2)\mapsto (a+\lambda(1-v_2^t\sigma_{2n-2i}w_2),v_1+\lambda w_1,v_2+\lambda w_2,w_1,w_2).
\]
In particular, the cells $X_{2i}$ are $\AA^1$-contractible for $i=0,\ldots,n$.
\end{lm}

As a consequence, we obtain the following lemma. 

\begin{lm}
For $n\geq 1$, there is a canonical weak-equivalence 
\[
c_n:\HP^n\simeq \mathrm{Th}(U_{\vert \HP^{n-1}}).
\]
\end{lm}

\begin{proof}
Consider the open immersion $X_0\subset \HP^n$, whose closed complement is $\overline X_2$. The normal bundle to $\overline X_2$ in $\HP^n$ being $\mathcal{U}^\vee$, we obtain by purity a cofiber sequence
\[
X_0\to\HP^n\to  \mathrm{Th}(U^\vee_{\vert\overline X_2})
\]
and therefore, as $X_0$ is contractible, a weak-equivalence $\HP^n\to  \mathrm{Th}(U^\vee_{\overline X_2})$. The inclusion $\HP^{n-1}\subset \overline X_2$ is a weak-equivalence (being the zero section of a vector bundle) and the pull back of $U^\vee_{\vert\overline X_2}$ under this map is just $U^\vee_{\HP^{n-1}}$, which is isomorphic to its dual. We then obtain a string of canonical weak-equivalences $\mathrm{Th}(U^\vee_{\vert\overline X_2})\simeq \mathrm{Th}(U^\vee_{\vert \HP^{n-1}})\simeq \mathrm{Th}(U_{\vert \HP^{n-1}})$.
\end{proof}

For $n\geq 0$, the universal subbundle $\mathcal U$ on $\HP^{n+1}$ restricts to the canonical bundle on $\HP^n$ and we obtain an induced map 
\[
j_n:\mathrm{Th}(U_{\vert \HP^{n}})\to \mathrm{Th}(U_{\vert \HP^{n+1}}).
\]
The following lemma will prove useful in the sequel. 

\begin{lm}\label{lem:fundamental}
Let $n\geq 1$ and let $\tau_n:\HP^n\to \mathrm{Th}(U_{\vert \HP^{n}})$ be the canonical morphism.
Then, the following diagram commutes in $\mathcal{H}(S)$
\[
\xymatrix{\HP^n\ar[r]^-{c_n}\ar[d]_-{i_n}\ar[rd]_-{\tau_n} & \mathrm{Th}(U_{\vert \HP^{n-1}})\ar[d]^-{j_{n-1}}  \\
\HP^{n+1}\ar[r]_-{c_{n+1}} &  \mathrm{Th}(U_{\vert \HP^{n}}).
}
\]
\end{lm}

\begin{proof}
We first prove that the bottom triangle commutes. Consider the composite of the zero section $\overline X_2\to U_{\vert \overline X_2}$ with the projection $U_{\vert \overline X_2}\to\mathrm{Th}(U_{\vert \overline X_2})$. The deformation to the normal cone shows that the following diagram commutes in $\mathcal{H}(S)$ 
\[
\xymatrix{ & \overline X_2\ar[r]\ar[d] & U_{\vert \overline X_2}\ar[d] \\
X_0\ar[r] & \HP^{n+1}\ar[r] & \mathrm{Th}(U_{\vert \overline X_2}).}
\]
On the other hand, we have a commutative diagram
\[
\xymatrix{\HP^{n}\ar[r]\ar[d] & U_{\vert \HP^n}\ar[d] \\
\overline X_2\ar[r] & U_{\vert \overline X_2}}
\] 
inducing the weak-equivalence $\mathrm{Th}(U_{\vert \HP^n})\to \mathrm{Th}(U_{\vert \overline X_2})$. The claim follows from the fact that the composite 
\[
\HP^{n}\to \overline X_2\to \HP^{n+1}
\]
is $i_n$. We prove next that the other triangle commutes. For this, recall that $X_0^\prime\subset \HP^{n}$ is defined as the complement of the vanishing locus of a section $s:\mathcal{U}\to \OO$. We then obtain a commutative diagram
\[
\xymatrix{X_0^\prime\ar[r]\ar[d] & \HP^{n}\ar[r]\ar[d]_-s & \mathrm{Th}(U_{\vert \overline X_2^\prime})\ar@{-->}[d] \\
U\setminus 0\ar[r] & U\ar[r] & \mathrm{Th}(U).
}
\]
In $\mathcal{H}(S)$, $s$ and the zero section are equal, and we conclude observing that the composite 
\[
\xymatrix{\mathrm{Th}(U_{\vert \HP^{n-1}})\ar[r] & \mathrm{Th}(U_{\vert \overline X_2^\prime})\ar@{-->}[r] & \mathrm{Th}(U)}
\]
is just $j_n$.
\end{proof}

We now pass to the definition of an action of $\mathbb{G}_m$ on $\HP^n$ for any $n\in\NN$. Let $\lambda\in \OO(S)^\times$ and let $m_\lambda:\OO_S^2\to \OO_S^2$ be given by the matrix
\[
\begin{pmatrix} \lambda & 0 \\ 0 & 1\end{pmatrix}.
\]
A direct computation shows that $m_\lambda^\vee \sigma_2 m_\lambda=\lambda\cdot \sigma_2$ and it follows that $m_{\lambda}$ induces a map
\[
\mathrm{Sp}_2(S)\to \mathrm{Sp}_2(S)
\]
given by $A=\begin{pmatrix} a & b \\ c & d\end{pmatrix}\mapsto m_{\lambda}^{-1}\cdot A\cdot m_{\lambda}=\begin{pmatrix} a & \lambda^{-1}b \\ \lambda c & d \end{pmatrix}$.
This induces an action of $\mathbb{G}_m$ on $\mathrm{Sp}_2$, and consequently an action of $\mathbb{G}_m$ on $\mathrm{BSp}_2$. This could be defined at the level of $\HP^n$ for $n\in\NN$ as follows: Set
\[
\mathrm{Gr}(2,\OO_S^{2n+2})\to \mathrm{Gr}(2,\OO_S^{2n+2})
\]
by considering the composite 
\[
\mathcal{U}\xrightarrow{i} \OO^{2n+2}\xrightarrow{m_{\lambda}^{\oplus n+1}} \OO^{2n+2}.
\]
If $i^\vee \sigma_{2n+2} i$ is non degenerate, then so is the composite 
\[
i^\vee (m_{\lambda}^{\oplus n+1})^\vee \sigma_{2n+2} (m_{\lambda}^{\oplus n+1}i)=\lambda\cdot  (i^\vee \sigma_{2n+2} i).
\]
Consequently, the above action of $\mathbb{G}_m$ on $\mathrm{Gr}(2,\OO_S^{2n+2})$ restricts to an action of $\mathbb{G}_m$ on $\mathrm{HGr}(2,\OO_S^{2n+2})=\HP^{n}$. Let us denote by 
\[
\mu:\mathbb{G}_m\times \HP^{n} \to \HP^{n}
\]
this action, and $\pi:\mathbb{G}_m\times \HP^{n}\to \HP^{n}$ the projection to the second factor.

\begin{lm}
There is a canonical isomorphism $f:\pi^*\mathcal{U}\to \mu^*\mathcal{U}$ on $\mathbb{G}_m\times \HP^{n}$.
\end{lm}

\begin{proof}
We start by exhibiting a $\mathbb{G}_m$-equivariant open covering $\cup_{i=1}^m V_i$ of $ \HP^{n}$ which has the following properties
\begin{enumerate}
\item $V_i$ is affine for any $i=1,\ldots,m$.
\item The restriction of $\mathcal U$ to $V_i$ is free for any $i=1,\ldots,m$.
\end{enumerate} 
To obtain this, we consider the composite morphism
\[
\HP^n\to \mathrm{Gr}(2,\OO_S^{2n+2})\to \mathbb{P}_S^{\frac{(2n+2)(2n+1)}2}
\]
where the second map is the Pl\"ucker embedding. We note that the $\mathbb{G}_m$-action on $\HP^n$ is the restriction of a $\mathbb{G}_m$-action on $\mathrm{Gr}(2,\OO_S^{2n+2})$. The latter can be extended (taking exterior powers) to an action on $\mathbb{P}_S^{\frac{(2n+2)(2n+1)}2}$ which is easily seen to be diagonal. It follows that the usual covering of the latter by affine spaces (complement of the zero locus of a given coordinate) is $\mathbb{G}_m$-equivariant. We define the open subschemes $V_i\subset \HP^n$ as the intersections of this covering with $\HP^n$. The latter being affine, it follows that (1) above is satisfied. For (2), we observe that the condition on the minors is equivalent to the composite
\[
\mathcal{U}\xrightarrow{i}\OO^{2n+2}\to \OO^2
\]
being an isomorphism (the projection corresponds to the identification of some minor). 

This covering being obtained, we now choose over each $V_i$ a symplectic basis of $\mathcal U_{\vert V_i}$ and we may recover $\mathcal U$ from the open covering $V_i$ and transition functions $V_i\cap V_j\to \mathrm{Sp}_2$. A direct computations shows that $\mu^*\mathcal U$ is then obtained via the transition functions
\[
V_i\cap V_j\to \mathrm{Sp}_2\to \mathrm{Sp}_2
\]
where the second map is given by 
\[
\begin{pmatrix} a & b \\ c & d\end{pmatrix}\mapsto \begin{pmatrix} a & tb \\ t^{-1}c & d\end{pmatrix}=\begin{pmatrix} t & 0 \\ 0 & 1\end{pmatrix}\begin{pmatrix} a & b \\ c & d\end{pmatrix}\begin{pmatrix} t^{-1} & 0 \\ 0 & 1\end{pmatrix}
\]
where $t$ is a parameter of $\mathbb{G}_m$. We now define $\pi^*\mathcal U_{\vert V_i}\to \mu^*\mathcal U_{\vert V_i}$ using the matrix $\begin{pmatrix} t & 0 \\ 0 & 1\end{pmatrix}$. A direct computation shows that this behaves well with the transition functions, yielding the required isomorphism $f:\pi^*\mathcal{U}\to \mu^*\mathcal{U}$.
\end{proof}

We now fix $\lambda\in \OO(S)^\times$, obtaining a multiplication by $\lambda$ morphism
\[
\mu_{\lambda}:\HP^n\to \HP^n.
\]
Using the above lemma, we obtain a commutative diagram
\[
\xymatrix{U\ar[r]^-f\ar[rd] & \mu_{\lambda}^*U\ar[r]\ar[d] & U\ar[d] \\
 & \HP^n\ar[r]_-{\mu_{\lambda}} & \HP^n}
\]
and then a morphism of Thom spaces $\mu_{\lambda}^f:\mathrm{Th}(U_{\vert \HP^n})\to \mathrm{Th}(U_{\vert \HP^n})$ making the diagram 
\[
\xymatrix{\HP^n\ar[r]^-{\mu_{\lambda}}\ar[d] & \HP^n\ar[d] \\
\mathrm{Th}(U_{\vert \HP^n})\ar[r]_-{\mu_{\lambda}^f} & \mathrm{Th}(U_{\vert \HP^n})}
\] 
commutative. 

We note here that the action of $\mathbb{G}_m$ on $\HP^n$ stabilizes the cellular filtration  
\[
\HP^n=X_0\sqcup \ldots\sqcup X_{2n}
\] 
in the sense that the subschemes $X_i$ (and their closure) are $\mathbb{G}_m$-equivariant. As a result, we see that the morphisms in Lemma \ref{lem:fundamental} all commute with the relevant morphisms $\mu_{\lambda}$ and $\mu_{\lambda}^f$.

\begin{lm}\label{lem:initialize}
The endomorphism $\mu_{\lambda}:\HP^1\to \HP^1$ is equal to $\langle\lambda\rangle$ in $\mathrm{End}_{\mathcal{SH}(S)}(\HP^1)$.
\end{lm}

\begin{proof}
By the above remark, we have a commutative diagram
\[
\xymatrix{\HP^1\ar[r]^-{c_1}\ar[d]_-{\mu_{\lambda}} & \mathrm{Th}(U_{\vert \HP^0})\ar[d]^-{\mu_{\lambda}^f} \\
\HP^1\ar[r]_-{c_1} & \mathrm{Th}(U_{\vert \HP^0})}
\]
Now, $U_{\vert \HP^0}$ is trivial and $f$ is just the multiplication by $\begin{pmatrix} \lambda & 0 \\ 0 & 1\end{pmatrix}$. The claim follows easily. 
\end{proof}

By construction, there is a tautological Thom class 
\[
\mathrm{th}(U_{\vert \HP^n}):\mathrm{Th}(U_{\vert \HP^n})\to \mathbf{MSp}(2)[4]
\]
which has the property that the following diagram commutes for any $n\in\NN$:
\[
\xymatrix@C=5em{\mathrm{Th}(U_{\vert \HP^n})\ar[r]^-{\mathrm{th}(U_{\vert \HP^n})}\ar[d]_-{j_n} & \mathbf{MSp}(2)[4]\ar@{=}[d] \\
\mathrm{Th}(U_{\vert \HP^{n+1}})\ar[r]_-{\mathrm{th}(U_{\vert \HP^{n+1}})} & \mathbf{MSp}(2)[4]}
\]

\begin{prop}
For any $n\geq 0$, the composite
\[
\mathrm{Th}(U_{\vert \HP^n})\xrightarrow{\mu_{\lambda}^f} \mathrm{Th}(U_{\vert \HP^n})\xrightarrow{\mathrm{th}(U_{\vert \HP^n})} \mathbf{MSp}(2)[4]
\]
is equal to $\langle \lambda\rangle \mathrm{th}(U_{\vert \HP^n})$ in $[\mathrm{Th}(U_{\vert \HP^n}),\mathbf{MSp}(2)[4]]_{\mathcal{SH}(R)}$.
\end{prop}

\begin{proof}
We work by induction on $n$. If $n=0$, this is a direct consequence of (the proof of) Lemma \ref{lem:initialize}. Let us then suppose that the result is true for some $n-1\geq 0$ and consider the composite
 \[
\mathrm{Th}(U_{\vert \HP^n})\xrightarrow{\mu_{\lambda}^f} \mathrm{Th}(U_{\vert \HP^n})\xrightarrow{\mathrm{th}(U_{\vert \HP^n})} \mathbf{MSp}(2)[4].
\]
We have a commutative diagram
\[
\xymatrix{\mathrm{Th}(U_{\vert \HP^{n-1}})\ar[r]^-{\mu_{\lambda}^f}\ar[d]_-{j_{n-1}} & \mathrm{Th}(U_{\vert \HP^{n-1}})\ar[r]^-{\mathrm{th}(U_{\vert \HP^{n-1}})}\ar[d]_-{j_{n-1}} & \mathbf{MSp}(2)[4]\ar@{=}[d]
\\ \mathrm{Th}(U_{\vert \HP^n})\ar[r]_-{\mu_{\lambda}^f} & \mathrm{Th}(U_{\vert \HP^n})\ar[r]_-{\mathrm{th}(U_{\vert \HP^n})} & \mathbf{MSp}(2)[4]}
\]
By induction, we know that  
\[
(\mathrm{th}(U_{\vert \HP^n})\circ \mu_{\lambda}^f)\circ j_{n-1}=(\langle \lambda\rangle \mathrm{th}(U_{\vert \HP^n}))\circ j_{n-1}
\]
and it suffices to prove that 
\[
j_{n-1}^*:[\mathrm{Th}(U_{\vert \HP^n}),\mathbf{MSp}(2)[4]]\to [\mathrm{Th}(U_{\vert \HP^{n-1}}),\mathbf{MSp}(2)[4]]
\]
is injective for $n\geq 1$. Now, we may use Lemma \ref{lem:fundamental} to obtain
\[
\tau_n^*=c_n^*\circ j_{n-1}^*
\]
and it suffices to show that $\tau_n^*$ is injective. However, it follows from \cite{PWQuaternionic} that the cofiber sequence
\[
U\setminus 0\to \HP^n \xrightarrow{\tau_n} \mathrm{Th}(U_{\vert  \HP^n})
\]
induces a short split exact sequence
\[
0\to [\mathrm{Th}(U_{\vert  \HP^n}),\mathbf{MSp}(2)[4]]\xrightarrow{\tau_n^*}[\HP^n,\mathbf{MSp}(2)[4]]\to  [U\setminus 0,\mathbf{MSp}(2)[4]]\to 0
\]
and in particular $\tau_n^*$ is injective.
\end{proof}

\begin{proof}[Proof of \Cref{thm:epsilon}]
Taking colimits, it is sufficient to prove the result for the Thom classes of $\mathcal U_{\vert \HP^n}$ for any $n\in \NN$. It follows from the above proposition that we have
\[
\langle \lambda\rangle \mathrm{th}(U_{\vert \HP^n})=\mathrm{th}(U_{\vert \HP^n})\circ \mu_{\lambda}^f.
\]
Now, the property of Thom classes reads as $\mathrm{th}(U_{\vert \HP^n})\circ \mu_{\lambda}^f=\mathrm{th}(f^*\mu_{\lambda}^*U_{\vert \HP^n})$ and a direct computation shows that $f^*\mu_{\lambda}^*(U_{\vert \HP^n},\varphi)=(U,\langle\lambda\rangle \varphi)$.
\end{proof}

\end{num}

\subsection{The FTL associated to an $\Sp$-oriented spectrum}

\begin{num}\label{num:SP-orient->FTL}
Let $(\E,b)$ be an $\Sp$-oriented ring spectrum over $S$.
 Its ring of coefficients $\E^{**}:=\E^{**}(S)$ is a bigraded ring.
 Moreover, the action of $\epsilon=-\langle -1 \rangle$ on $\SH(S)$
 allows to equip it with a bigraded $\ZZe$-algebra structure.
 When dealing with symplectic orientations, and the forthcoming associated
 FTL, it is more natural to consider the subring
 $\E^{\tw 2*}:=(\E^{4n,2n}(S), n \in \ZZ)$,
 which is now a graded $\ZZe$-algebra.\footnote{Indeed, Borel classes live
 in degree $\tw 2*$.}

In \cite[Def. 2.3.2]{DF3}, we proved that one can associate to $(\E,b)$
 a $(3,4)$-series that we will denote by $F_t^b(x,y,z)$
 with coefficients in $\E^{\tw 2*}$.
 The method is directly imported from the $\GL$-oriented case
 (compare \cite[2.1.19]{Deg12}):
 thanks to the symplectic projective bundle theorem (\cite[2.2.3]{DF3}),
 one gets a canonical isomorphism
$$
\E^{\tw 2*}(\HP^\infty_S \times_S \HP^\infty_S \times_S \HP^\infty_S)
 \simeq \E^{\tw 2*}(S)[[x,y,z]]
$$
where $x$ (resp. $y$, $z$) are formal variables which corresponds
 to the Borel class of the tautological symplectic bundle $\cU_1$ (resp. $\cU_2$, $\cU_3$) on $\HP^\infty_S$
  over the first (resp. second, third) coordinates.
 The triple tensor product $\cU_1 \otimes \cU_2 \otimes \cU_3$ is naturally a rank $8$ symplectic bundle. 
 It follows that one can write the total Borel class \eqref{eq:total_borel}:
\begin{equation}\label{eq:Borel&FTL}
F_t^b(x,y,z):=b_t(\cU_1 \otimes \cU_2 \otimes \cU_3)
\end{equation}
as a polynomial in $t$ of degree $4$, with coefficients in $\E^{**}[[x,y,z]]$.
 In other words, the expression $F_t^b(x,y,z)$
 defines a $(3,4)$-series in the formal variables $x,y,z$ as required.
 It satisfies the symmetry and associativity axioms of \Cref{df:FTL}
 according to the analogous properties of the tensor product
 of symplectic bundles.

In \emph{op. cit.}, the Neutral and Semi-neutral element axioms were only proved
 under the additional assumption that $\E$ is $\SL$-oriented.
 Thanks to \Cref{cor:epsilonlinearity}, we can now avoid the latter assumption on $\E$.
 In fact, the mentioned corollary readily implies that the FTL $F_t^b(x,y,z)$
 satisfies the $\epsilon$-linearity axiom.
 Once this property is known, one deduces the Neutral and Semi-neutral element
 properties as in the proof of \cite[Prop. 2.3.4]{DF3}.

The above procedure allows us to deduce the following theorem
 (which enhances \cite[\textsection 2.3]{DF3} thanks to the $\epsilon$-linearity axiom).
\end{num}
\begin{prop}
The association $(\E,b) \mapsto \big(\E^{\tw 2*},F_t^b(x,y,z)\big)$ described above
 defines a canonical pseudo-functor from the scheme-fibred category of $\Sp$-oriented ring spectra
 to the ring-fibred category of formal ternary laws $\FTL$.
\end{prop}
\begin{proof}
The association on objects was described in the paragraph preceding the statement.
 The functoriality assertion follows from Formulas \eqref{eq:Borel&FTL},
 \eqref{eq:Sp-morphisms&Borel_classes}
 and the symplectic splitting principle \cite[Par. 2.2.5]{DF3}.
\end{proof}

\begin{rem}\label{rem:repFTL&BC}
To be more explicit about the pseudo-functoriality assertion (see also the forthcoming proposition),
 we mean in particular that for any morphism
 of scheme $f:T \rightarrow S$ and any s$\Sp$-oriented ring spectrum $(\E,b)$ over $S$,
 if we let $\E_T=f^*(\E)$ with the $\Sp$-orientation $b_T$ obtained by pullback
 (from the $\MSp$-algebra structure), then one gets a canonical morphism of FTL:
$$
\big(\E^{\tw 2*}(S),F_t^b\big) \rightarrow \big(\E^{\tw 2*}(T),F_t^{b_T}\big)
$$
which can be decomposed as an extension of scalars along
 the ring extension $\E^{\tw 2*}(T)/\E^{\tw 2*}(S)$ and then an isomorphism.
\end{rem}

\begin{df}
Let $S$ be a scheme. We say that a formal ternary law $(R,F_t)$ is \emph{representable over $S$}
 if there exists an $\Sp$-oriented ring spectrum $(\E,b)$ over $S$ such that $(R,F_t) \simeq (\E^{**}(S),F_t^b)$.  
 We will say that $(R,F_t)$ is \emph{weakly representable over $S$} if there exists an $\Sp$-oriented ring spectrum $(\mathbb{E},b)$
 and a (graded) ring homomorphism $\varphi: \mathbb{E}^{*,*}(S)\to R$ such that $(R,F_t)=(\E^{**},\varphi_*F_t^b)$.

Representable (resp. weakly) will mean representable (resp. weakly) over some scheme.
\end{df}

\begin{ex}\label{ex:representableFTL}
Recall \Cref{df:addmulFTL} for additive and multiplicative FTL.
\begin{enumerate}
\item The additive FTL with coefficients in $\ZZe$ (resp. $\QQe$) is representable
 by Milnor-Witt motivic cohomology spectrum $\HMW \ZZ$ over $\QQ$ (resp. $\ZZ$)
 as proved in \cite[Th. 3.3.2]{DF3}) (resp. \cite[Cor. 3.3.5]{DF3}).
 More generally, over any $\ZZ[1/6]$-scheme (any scheme) $S$, the $\Sp$-oriented ring spectrum $\HMW\ZZ_S$
 (resp. $\HMW\QQ_S$) admits the additive formal group law (see \emph{loc. cit.} and use base change
 \Cref{rem:repFTL&BC}).
\item The multiplicative FTL with coefficients in $\ZZe^{mul}$
 is representable by the (higher) Grothendieck-Witt (aka hermitian K-theory) $\GWsp_\RR$ ring spectrum
 over $\RR$, according to \cite[Th. 6.6]{FH21} (see also \cite[\textsection 3.2]{CDFH}).
 More generally, using again base change, over any $\ZZ[1/2]$-scheme $S$,
 the $\Sp$-oriented ring spectrum $\GWsp_S$ admits the multiplicative formal group law.
\end{enumerate}
\end{ex}

\begin{rem}
There is an obvious obstruction for an FTL to be representable:
 its ring of coefficients must be isomorphic to the $\tw2*$-graded part of the
 ring of coefficients of some (motivic) ring spectrum. In particular, it must be graded!
 Note on the other hand that the grading on the ring of coefficients
 of a representable FTL is compatible with the notion of degree 
 introduced in \Cref{df:FTL}.

Apart from the obvious restriction on grading,
 the problem of determining which FTLs are representable is a mystery.
 Note that it is already an open question to determine which formal group laws
 are representable by a classical complex-oriented ring spectrum, or
 even a $\GL$-oriented (motivic) ring spectrum.
\end{rem}

\begin{prop}\label{prop:sp-orient&FTL}
Let $(\E,b)$ be an $\Sp$-oriented ring spectrum with associated formal ternary law $F_t^b$.

Then there exists a one-to-one correspondence between the strict isomorphisms of $FTL$ (\Cref{df:morph_FTL})
 of the form $(Id,\phi):(\E^{\tw 2*},F_t^b) \rightarrow (\E^{\tw 2*},F_t)$
 and the $\Sp$-orientations $b'$ of $\E$.
\end{prop}
\begin{proof}
Indeed, according to \Cref{cor:chg_Sp-orient},
 the $\Sp$-orientation $b'$ corresponds to the power series $\phi$
 with coefficients in $\E^{**}$ such that $b'=\phi(b)$.
 The fact that $(Id,\phi)$ is an isomorphism from $F_t^b$ to $F_t^{b'}$
 follows from the symplectic splitting principle and Formula
 \eqref{eq:firstBorel&chg_orient}.

Given now a strict isomorphism of the form
 $(Id,\phi):(\E^{**},F_t^b) \rightarrow (\E^{**},F_t)$,
 we get that $b'=\phi(b)$ is an $\Sp$-orientation and the isomorphism $\phi$
 gives an identification between $F_t$ and $F_t^b$.
\end{proof}

As in the classical case, one gets the following notable corollary.
\begin{cor}
Given an $\Sp$-orientable ring spectrum $\E$,
 there exists a unique formal ternary law $F_t^\E$ up to strict isomorphism associated with an $\Sp$-orientation on $\E$.

Moreover, given any ring spectrum $\F$, if $\F$ admits an $\E$-algebra structure,
 then $\F$ is $\Sp$-orientable and the FTL associated with any $\Sp$-orientation on $\F$
 is uniquely strictly isomorphism to $F_t^\E$
 (after extension of scalars along $\F^{\tw 2*}/\E^{\tw 2*}$).
\end{cor}

We close this section with the following theorem.
\begin{thm}\label{thm:log_FTL}
 Any representable formal ternary law $F_t$
 with coefficients in a $\QQ$-algebra is uniquely strictly isomorphic to the additive FTL.
\end{thm}
In other words, there exists a unique strict isomorphism of FTLs
$$
\mathrm{log}_{F_t}:(R,F_t) \rightarrow \big(R,F_t^{add}\big)
$$
called the \emph{logarithm} of $F_t$.
\begin{proof}
Let $(\E,b)$ be a rational $\Sp$-oriented ring spectrum over some scheme $S$
 whose associated FTL is $F_t$.
 The rationalization $L_\QQ\E$ still represents $F_t$,
 so we can assume that $E$ is rational.
 According to \cite[Cor. 6.2]{DFJK}, the rational sphere spectrum over $S$ is isomorphic
 to the rational Milnor-Witt motivic cohomology spectrum $\HMW\QQ_S$.
 In particular $\E$ admits a canonical $\HMW\QQ_S$-algebra structure.
 Then the theorem follows from \Cref{ex:representableFTL}(1) and the preceding corollary.
\end{proof}

\subsection{Todd classes and FTL}

\begin{num}
As in \Cref{num:todd},
 we consider $\Sp$-oriented ring spectra $(\E,b)$ and $(\F,b')$
 and a morphism  of ring spectra $\psi:\E \rightarrow \F$.
 For any scheme $S$, we consider the induced morphism of bigraded rings:
$$
\psi_{\HP^\infty_S}:\E^{**}(\HP^\infty_S) \rightarrow \F^{**}(\HP^\infty_S).
$$
According to the symplectic projective bundle theorem (\cite[Th. 2.2.3]{DF3}), $\psi_{\HP^\infty_S}$ corresponds to a morphism of bigraded rings:
$$
\E^{**}(S)[[u]] \rightarrow \F^{**}(S)[[t]]
$$
where $u$ and $t$ are indeterminates of bidegree $(4,2)$.
 Let us denote by $\phi(t)$ the power series image of $u$ under this latter map.
 By definition, $\phi(t)$ is uniquely determined by the relation in $\F^{**}(\HP^\infty_S)$:
\begin{equation}
\psi_{\HP^\infty_S}(b)=\phi(b').
\end{equation}
Moreover, $\phi(b')$ is an $\Sp$-orientation of $\F$ over $S$
 and it follows from \Cref{cor:chg_Sp-orient} that it has the form:
$$
\phi(t)=t+\sum_{i>1} a_i.t^i
$$
where $a_i \in \F^{4-4i,2-2i}(S)$. 
 In particular, the following power series is well defined:
$$
\tilde \Theta(t):=\frac t {\Theta(t)}=1-a_2.t+(a_2^2-a_3).t^2+(-a_4+2a_2a_3-a^3_2).t^4+ \cdots
$$
The next result is a direct analog of the $\GL$-oriented case
 (see \cite[Prop. 4.1.2]{Deg12}).  
\end{num}

\begin{prop}
The Todd class $\td_\psi:\KSp_0 \rightarrow \F^{00 \times}$
 associated with the morphism $\psi$ (\Cref{prop:computeToddclass})
 is uniquely characterized by the relation 
$$
\td_\psi([\cU])=\left. \frac t {\Theta(t)}\right\vert_{t=b_1(\cU)}
$$
for a rank $2$ symplectic bundle $\cU$ over $X$.
\end{prop}
The case of an arbitrary symplectic bundle $\cU$ of rank $2n$ is determined
 by the symplectic splitting principle \Cref{num:sp-splitting-principle}.
 One introduces Borel roots
 $b_1^{(1)}=b_1(\cU_1)$, ..., $b_1^{(n)}=b_1(\cU_n)$ such that the Borel classes
 of $\cU$ satisfies the relations given by $b_t(\cU)=\prod_i (1+b_1^{(i)}.t)$.
 Then we obtain the expression
$$
\td_\psi([\cU])=\prod_{i=1}^n \tilde \Theta\big(b_1^{(i)}\big)
$$
which gives a computation of the left hand-side in terms of the Borel classes
 $b_i(\cU)$.


\section{Todd classes and GRR formulas in the symplectic case}\label{sec:grr}

In this section, we compute the Todd classes associated to some morphisms of ring spectra. We focus on two morphisms, both involving ring spectra that are $\Sp$-oriented but not $\GL$-oriented.   

\subsection{Adams operations}

Let $S=\Spec (\ZZ[\frac 12])$, and let $\KO$ be the ring spectrum representing Hermitian $K$-theory (aka higher Grothendieck-Witt groups) over $S$. For any integer $n\in \NN$, let further $n^*\in \KO^{0,0}(S)$ be defined by 
\[
n^*=\begin{cases} n & \text{if $n$ is odd.} \\ \frac n2(1-\epsilon) & \text{if $n$ is even.}\end{cases}
\]
For any $n\geq 1$, we have a morphism of ring spectra (constructed either in \cite{FH21} or \cite{Bachmann20a})
\[
\Psi^n:\KO\to \KO\left[\frac 1{n^*}\right]
\]
called the \emph{$n$-th Adams operation}. These morphisms satisfy the same properties as the classical Adams operations on $K$-theory. For instance, we have $\Psi^m[\frac 1{n^*}]\circ\Psi^n=\Psi^{mn}$ for any $m,n\geq 1$. Moreover, denoting by $f:\KO\to \KGL$ the forgetful map, we obtain a commutative diagram for any $n\geq 1$
\[
\xymatrix{\KO\ar[r]^-{\Psi^n}\ar[d]_-f & \KO[\frac 1{n^*}]\ar[d]^-f \\
\KGL\ar[r]_-{\Psi^n} &  \KGL[\frac 1{n}]}
\]
where the bottom map is the classical Adams operation. Note that $n^*$ is hyperbolic when $n$ is even. The Adams operations on $\KO$ are obtained via unstable operations
\[
\psi^n:\GW^{2i}\to \GW^{2in}
\]
for any $i\in\ZZ$ and any $n\in\NN$ (\cite[\S 5]{FH21}). Here $\GW^{2i}$ represents symplectic $K$-theory if $i$ is odd and orthogonal $K$-theory if $i$ is even. There is a periodicity element $\gamma\in \GW^{4}(\ZZ)$ having the property that  
\[
\GW^{2i}(X)\xrightarrow{\cdot\gamma} \GW^{2i+4}(X)
\]
is an isomorphism for any $i\in \ZZ$ and any $X$. In that setting, the unstable operation $\psi^n$ can be inductively computed using the formula  (\cite[(5.1.c)]{FH21})
\begin{equation}\label{eqn:inductiveAdams}
\psi^n(y)=y\psi^{n-1}(y)-\gamma\psi^{n-2}(y)
\end{equation}
where $y$ is the class in $\GW^{2}(X)$ of a rank $2$ symplectic bundle.

Even though the spectrum we consider is a priori not defined over $\ZZ$, it is still convenient to consider the graded $\ZZ_{\epsilon}$-algebra
\[
\ZZ_{\epsilon}[\tau,\gamma^{\pm 1}]/\langle \tau^2-2h\gamma,(1+\epsilon)\tau\rangle
\]
in which $\tau$ is of degree $2$ and $\gamma$ of degree $4$. There is an isomorphism of graded $\ZZ_{\epsilon}=\GW^0(\ZZ)$-algebras
\[
\ZZ_{\epsilon}[\tau,\gamma^{\pm 1}]/\langle \tau^2-2h\gamma,(1+\epsilon)\tau\rangle\simeq \bigoplus_{i\in\ZZ}\GW^{2i}(\ZZ)
\]
under which $\tau$ is mapped to the class of the hyperbolic plane in $\GW^2(\ZZ)$. It is clear that for any $X$ as above, $\KO^{*,*}(X)$ is an algebra over this ring.

If $E$ is a symplectic bundle of rank $2$ on a scheme $X$, its (first) Borel class $b_1(E)$ is the class of $[E]-\tau$ in $\KO^{4,2}(X)=\GW^2(X)$. In particular, we have seen that $X=\HP^{\infty}_{\ZZ[\frac 12]}$ (which can be seen as an ind-object in the category of smooth schemes) has a universal symplectic bundle $U$ with associated class $u\in \KO^{4,2}(X)$, and we set $x:=b_1(U)=u-\tau$. 

\begin{lm}
For any $n\geq 1$, there exists a unique \emph{monic} polynomial 
\[
p_n(x)\in \ZZ_{\epsilon}[\tau,\gamma^{\pm 1},x]/\langle \tau^2-2h\gamma,(1+\epsilon)\tau\rangle
\] 
of degree $n-1$ such that 
\[
\psi^n(x)=xp_n(x).
\]
It satisfies the inductive formula $p_1(x)=1$, $p_2(x)=x+2\tau$ and
\[
p_n(x)=(x+\tau)p_{n-1}(x)-\gamma p_{n-2}(x)+\psi^{n-1}(\tau).
\]
\end{lm}

\begin{proof}
For $u$, we have $\psi^0(u)=2$, $\psi^1(u)=u$ and for $n\geq 2$
\[
\psi^n(u)=u\psi^{n-1}(u)-\gamma\psi^{n-2}(u)
\]
in view of formula \eqref{eqn:inductiveAdams}. The same holds for $\tau$ and we we obtain
\begin{eqnarray*}
\psi^n(u)-\psi^n(\tau) & = & u\psi^{n-1}(u)-\gamma\psi^{n-2}(u)-\tau\psi^{n-1}(\tau)+\gamma\psi^{n-2}(\tau) \\
 & = & (u-\tau)\psi^{n-1}(u)+\tau(\psi^{n-1}(u)-\psi^{n-1}(\tau))-\gamma(\psi^{n-2}(u)-\psi^{n-2}(\tau)) \\
 & = & (x+\tau)\psi^{n-1}(x)+x\psi^{n-1}(\tau)-\gamma\psi^{n-2}(x).
\end{eqnarray*}
This yields 
\begin{eqnarray*}
xp_n(x) & = & \psi^n(x) \\
 & = & (x+\tau)\psi^{n-1}(x)+x\psi^{n-1}(\tau)-\gamma\psi^{n-2}(x) \\
 & = & (x+\tau)xp_{n-1}(x)+x\psi^{n-1}(\tau)-\gamma x p_{n-2}(x).
\end{eqnarray*}
As $x$ is a not a zero divisor, we conclude.
\end{proof}

As noted in the proof, the polynomials $p_n$ actually compute the quotients $\psi^n(x)/x$. In view of \eqref{eq:Todd&Euler}, the Todd class of $u$ associated to the operation $\Psi^n$ can actually be computed as 
\[
\td_{\Psi^n}^{-1}(u)=\Psi^n(x)/x
\]
since $x=u-\tau=b_1(u)=e(U,\KO)$ (in principle, we would have to choose a stable orientation of $U$, but here we can take the ``canonical one'' since $U$ is symplectic). Now, the stable operations $\Psi^n$ are obtained out of the unstable ones by inverting a convenient element $\omega(n)$ defined as follows (\cite[Lemma 5.3.3]{FH21}): 
\[
\omega(n)=\begin{cases} 0 & \text{ if $n=0$.}  \\ 1 &  \text{ if $n=1$.} \\ \tau\omega(n-1)-\gamma\omega(n-2)+\psi^{n-1}(\tau) &  \text{ if $n\geq 2$.}\end{cases}
\]
One can then explicitly compute $\omega(n)$ as follows (\cite[Proposition 5.3.4]{FH21}):
\[
\omega(n) = 
\begin{cases}
\displaystyle{n \Big(\frac{n-1}{2} (1-\epsilon) + \langle -1\rangle^{\frac{n-1}{2}}\Big)\gamma^{\frac{n-1}{2}} }& \text{if $n$ is odd},\\
\displaystyle{\frac{n^2}{2}\tau\gamma^{\frac{n-2}{2}}} & \text{if $n$ is even}
\end{cases}
\]
and set 
\begin{equation}\label{eqn:stablevsunstable}
\Psi^n=\frac 1{\omega(n)}\psi^n
\end{equation}
for any $n\in\NN$ (one observes that $\omega(n)$ is indeed invertible whenever $n^*$ is).

\begin{lm}
For any $n\geq 2$, we have and $p_n(x)=x^{n-1}+n\tau x^{n-2}+r_{n-2}(x)$ for some polynomial $r_{n-2}(x)$ of degree $\leq n-3$. Further, we have $p_n(0)=r_{n-2}(0)=\omega(n)$.
\end{lm}

\begin{proof}
For the first assertion, we already know that $p_n$ is monic. To compute the coefficient of $x^{n-2}$, it suffices work by induction using  
\[
p_n(x)=(x+\tau)p_{n-1}(x)-\gamma p_{n-2}(x)+\psi^{n-1}(\tau)
\]
and $p_2=x+2\tau$. For the second statement, we observe that $p_n(0)$ and $\omega(n)$ satisfy the same inductive definition. It suffices then to prove that $\omega(2)=p_2(0)$ to conclude. The left-hand side is $\tau+\psi^{1}(\tau)=2\tau$, which is precisely $p_2(0)$.
\end{proof}

\begin{rem}
We were not able to give a closed formula for $p_n$. For small values of $n$, they can be inductively computed, but it would be interesting to be actually able to write them for all integers $n$.
\end{rem}

As $p_n=\psi^n(x)/x$ and $\Psi^n=\frac 1{\omega(n)}\psi^n$, this immediately yields the following proposition.

\begin{prop}
For any $n\geq 1$, there exists unique polynomials 
\[
q_n(x)\in  \ZZ_{\epsilon}[\tau,\gamma^{\pm 1},\omega(n)^{-1},x]/\langle \tau^2-2h\gamma,(1+\epsilon)\tau\rangle
\]
of degree $n-1$ such that $q_n(0)=1$, $q_n(x)=\frac 1{\omega(n)}(x^{n-1}+n\tau x^{n-2})+r_{n-2}(x)$ with $r_{n-2}$ of degree $\leq n-3$.  They satisfy 
\[
\td_{\Psi^n}^{-1}(u)=q_n(u-\tau).
\]
\end{prop}

This proposition allows to compute the Todd class $\td_{\Psi^n}$ associated to any symplectic bundle of rank $2$, using the splitting principle and the fact that $U$ is the universal rank $2$ bundle.

\subsection{The Borel character}

We begin by recalling that the Borel character is a morphism of ring spectra 
\[
\mathrm{bo}_t:\KO \xrightarrow{\ (\mathrm{bo}_{2n})_{n \in \ZZ}\ } \bigoplus_{n \text{ even}} \mathbf{H}_\mathrm{MW}\QQ(2n)[4n] \oplus \bigoplus_{n \text{ odd}} \mathbf{H}_{\mathrm M} \QQ(2n)[4n].
\]
Let $p:\mathbf{H}_\mathrm{MW}\QQ\to \mathbf{H}_\mathrm{M}\QQ$ be the projection from MW-motivic cohomology to ordinary motivic cohomology and let $f:\KO\to \KGL$ be the forgetful map considered in the previous section (both are  morphism of ring spectra over $S$). We will consider $\mathbf{H}_\mathrm{M}\QQ$ (resp. $\KGL$) as a $\mathbf{H}_\mathrm{MW}\QQ$-algebra using the morphism $p$ (resp. a $\KO$-algebra using the morphism $f$). More explicitly, if $X$ is a (ind-)smooth $S$-scheme, we will set $\alpha\cdot\beta:=p(\alpha)\cdot \beta$ for any $\alpha\in \mathbf{H}_\mathrm{MW}\QQ(i)[j](X)$ and any $\beta \in \mathbf{H}_\mathrm{M}\QQ(a)[b](X)$.

We have a commutative diagram of spectra
\begin{equation}\label{eq:commutativecharacters}
\xymatrix{\KO\ar[r]^-{\mathrm{bo}_t}\ar[d]_-f & \bigoplus_{n \text{ even}} \mathbf{H}_\mathrm{MW}\QQ(2n)[4n] \oplus \bigoplus_{n \text{ odd}} \mathbf{H}_{\mathrm M} \QQ(2n)[4n]\ar[d]^-{p\oplus \mathrm{Id}} \\
\KGL\ar[r]_-{\mathrm{ch}_t} & \bigoplus_{n \text{ even}} \mathbf{H}_\mathrm{M}\QQ(2n)[4n] \oplus \bigoplus_{n \text{ odd}} \mathbf{H}_{\mathrm M} \QQ(2n)[4n]
}
\end{equation}
where $\mathrm{ch}_t$ is the usual Chern character.

Let further $X$ be a (ind-)smooth $S$-scheme and $E$ be a symplectic bundle of rank $2d$ on $X$. We may see $E$ as the class of a morphism $\Sigma_T^{\infty}X\to \KO_{S}(2)[4]$ whose image under the Borel character is by definition of the form
\begin{equation}\label{eq:Borelchar}
\mathrm{bo}_t(E)=2d+\tilde\chi_2(E)+\frac 1{4!}\chi_4(f(E))+\frac 1{\psi_6!}\tilde\chi_6(E)+\ldots
\end{equation}
in $\bigoplus_{n \text{ odd}} \mathbf{H}_\mathrm{MW}\QQ(2n)[4n] \oplus \bigoplus_{n \text{ even}} \mathbf{H}_{\mathrm M} \QQ(2n)[4n]$ where 
\[
\tilde\chi_{2n}:\GW^2\to \mathbf{H}_\mathrm{MW}(2n)[4n]
\]
is defined inductively by setting $\tilde\chi_{2}=b_1$ (the first Borel class in $\mathbf{H}_\mathrm{MW}(2)[4]$) and using
\begin{equation}
\tilde \chi_{2n}-b_1\tilde \chi_{2n-2}+b_2\tilde \chi_{2n-4}
+ \hdots +(-1)^{n-1}b_{n-1}\tilde \chi_2+(-1)^nnb_n=0.
\end{equation}
The operation
\[
\chi_{2n}\circ f:\GW^2\to \mathbf{H}_\mathrm{M}(2n)[4n]
\]
is defined in a similar fashion, using the forgetful functor $f:\GW^2\to \mathrm{K}_0$ and the operation $\chi_{2n}:\mathrm{K}_0\to  \mathbf{H}_\mathrm{M}(2n)[4n]$ worked out by Riou (\cite[Remark 6.2.2.3]{Riou}). The symmetric bilinear forms $\psi_{2n}$ are defined by 
\[
\psi_{2n}=\begin{cases}
1 & \text{ if $n=0,1$} \\
n(2n-1)(2n-2)(2n-3)h & \text{if $n\geq 2$ is even,} \\
2n(2n-2)\big((2n^2-4n+1)h-\epsilon\big) & \text{if $n\geq 3$ is odd,} \\
\end{cases}
\]
and we set
\[
\psi_{2n}!=\psi_2 \cdot \psi_6 \cdot \hdots \cdot \psi_{2n}.
\]
One checks that their inverses are indeed well defined in $\mathbf{H}_\mathrm{MW}\QQ(0)[0](S)$, at least if $\frac 12\in \OO(S)$ (\cite[Theorem 4.3.2]{DF3}).

If $U$ is the universal bundle on $\mathrm{H}\mathbb{P}^{\infty}_S$ and $u$ is its class in $\GW^2(\mathrm{H}\mathbb{P}^{\infty}_S)$, we have $\tilde \chi_{2+4n}(E)=b_1(U)^{2n+1}$ according to \cite[\S 4.1.7]{DF3}, while $\chi_{4n}(f(U))=2c_2(U)^{2n}$ \cite[Lemma 4.3.4]{DF3}. If $x:=b_1(u)=b_1(u-\tau)$, then its image $p(x)$ under the  morphism $p:\mathbf{H}_\mathrm{MW}\QQ(2)[4]\to \mathbf{H}_\mathrm{M}\QQ(2)[4]$ is $-c_2(U)$ and the $\mathbf{H}_\mathrm{MW}$-algebra structure on $\mathbf{H}_\mathrm{M}$ reads in this case as 
\[
p(x)^{2n}=x\cdot p(x)^{2n-1}
\] 
for any $n\geq 1$. As $b_1(\tau)$ and $c_2(\tau)$ are both trivial, we obtain an expression (slightly reordering terms) of $\mathrm{bo}_t(u-\tau)$ of the form
\[
\left(\sum_{n\geq 1}\frac 2{(4n)!}p(x)^{2n}\right )+\left( x+\sum_{n\geq 1}\frac 1{\psi_{2+4n}!} x^{2n+1}\right)=x\cdot \left(\sum_{n\geq 1}\frac 2{(4n)!}p(x)^{2n-1}+ 1 +\sum_{n\geq 1}\frac 1{\psi_{2+4n}!} t^{2n}\right).
\]
As $x=u-\tau$ is the first Borel class of $U$ in $\mathrm{KGL}(2)[4]$, we then see that 
\[
\td_{\mathrm{bo}_t}^{-1}(u)=\left(1+\sum_{n\geq 1}\frac 2{4n!}p(x)^{2n-1}+\sum_{n\geq 1}\frac 1{\psi_{2+4n}!} x^{2n}\right).
\]
We now make use of the fact that rationally the ring spectrum representing MW-motivic cohomology decomposes as a product (this could even by taken as a definition \cite[\S 5, \S 6]{DFJK})
\[
\mathbf{H}_\mathrm{MW}\QQ\simeq \mathbf{H}_\mathrm{M}\QQ\times \mathbf{H}_\mathrm{W}\QQ
\]
where $ \mathbf{H}_\mathrm{W}\QQ$ is the spectrum representing the cohomology of the rational Witt sheaf. Here the composite $\mathbf{H}_\mathrm{MW}\QQ\simeq \mathbf{H}_\mathrm{M}\QQ\times \mathbf{H}_\mathrm{W}\QQ\to \mathbf{H}_\mathrm{M}\QQ$ is just the morphism we denoted by $p$, and we denote by $p':\mathbf{H}_\mathrm{MW}\QQ\simeq \mathbf{H}_\mathrm{M}\QQ\times \mathbf{H}_\mathrm{W}\QQ\to \mathbf{H}_\mathrm{W}\QQ$ the composite of the decomposition and the second projection. This equivalence yields in particular in degree $(0,0)$ an isomorphism 
\[
\mathbf{H}^{0,0}_\mathrm{MW}\QQ(S)\simeq \mathbf{K}^{\mathrm{MW}}_0(S)\simeq \underline\QQ(S)\times \mathbf{W}(S)
\]
where $\underline\QQ$ is the sheaf associated to the constant presheaf $\QQ$. In view of this decomposition and the commutative diagram \eqref{eq:commutativecharacters}, we obtain
\[
p(\td_{\mathrm{bo}_t}^{-1}(u))=\td_{\mathrm{ch}_t}^{-1}(f(u))
\]
and we see that it suffices to compute the latter and $p'(\td_{\mathrm{bo}_t}(u))$ to obtain $\td_{\mathrm{bo}_t}(u)$. 
\begin{prop}\label{prop:Toddforgetful}
If $U$ is a rank $2$ symplectic bundle on a smooth $S$-scheme $X$ and $u\in \GW^2(X)$ is its associated class, we have
\[
\td_{\mathrm{ch}_t}^{-1}(f(u))=2\sum_{i\geq 1} \frac{(-c_2(U))^{i-1}}{(2i)!}.
\]
\end{prop}

\begin{proof}
By universality, it suffices to consider the situation where $X=\HP^{\infty}_S$ and $U$ is the universal rank two symplectic bundle. We may consider the projective bundle $\pi:Y:=\Proj(U)\to X$ and replace $X$ by $Y$ in view of the splitting principle, as both $K$-theory and motivic cohomology are $\GL$-oriented. Over $Y$, $U$ splits as an extension of a line bundle $L$ with its dual, and we have 
\[
\mathrm{ch}_t(f(u))=e^{c_1(L)}+e^{c_1(L^\vee)}=e^{c_1(L)}+e^{-c_1(L)}=2\cosh(c_1(L)).
\]
This yields $\mathrm{ch}_t(f(u-\tau))=2\cosh(c_1(L))-2$. Using $c_1(L)^2=-c_1(L)c_1(L^\vee)=-c_2(U)$, the expression reads as
\[
\mathrm{ch}_t(f(u-\tau))=2\sum_{i\geq 1} \frac{(-c_2(U))^i}{(2i)!}
\]
As the first Borel class of $U$ in $\mathrm{CH}^2(X)$ is $-c_2(U)$, we finally obtain
\[
\td_{\mathrm{ch}_t}^{-1}(f(u))=2\sum_{i\geq 1} \frac{(-c_2(U))^{i-1}}{(2i)!}.
\]
\end{proof}

We now start the computation of $p'(\td_{\mathrm{bo}_t}(u))$ which is obtained by omitting the terms of the form $\chi_{2n}$ (including $2d$) in the expression \eqref{eq:Borelchar}, replacing the operations $\tilde \chi_{2+4n}$ by the operations $\tilde \chi_{2+4n}^{\mathrm{W}}:=p'\circ \tilde \chi_{2+4n}$ and taking the images of the symmetric bilinear forms $\psi_{2n}$ in the group $\mathbf{W}(S)\otimes \QQ$. Since $\epsilon=1$ and $h=0$ in $\mathbf{W}(S)\otimes \QQ$, we obtain for $n\geq 3$ odd
\[
\psi_{2n}=-2n(2n-2),
\]
so that $\psi_{4n+2}=-4n(4n+2)=-8n(2n+1)=-4(2n+1)2n$ for any $n\geq 1$ and $\psi_2=1$.
The following lemma is easily proved by induction.
\begin{lm}
For $n\geq 0$, we have $\psi_{4n+2}!=(-1)^n2^{2n}(2n+1)!$.
\end{lm}

We then obtain an expression of the form
\[
p'(\td_{\mathrm{bo}_t}^{-1}(u))= x+\sum_{n\geq 1}\frac 1{\psi_{2+4n}!} x^{2n+1}=x+\sum_{n\geq 1}\frac {(-1)^n}{2^{2n}(2n+1)!} x^{2n+1}
\]
where $x=e(U,\mathbf{H}_\mathrm{W}\QQ)$ is the Euler class of $U$ (which coincides with the first Borel class). Dividing by $x$, we finally obtain the following result.

\begin{prop}\label{prop:Toddquadratic}
The Todd class associated to the morphism of ring spectra 
\[
\KO \xrightarrow{p'\circ \mathrm{bo}_t} \bigoplus_{n \text{ even}} \mathbf{H}_\mathrm{W}\QQ(2n)[4n]
\]
is of the form
\[
\td_{p'\circ\mathrm{bo}_t}^{-1}(u)=1 +\sum_{n\geq 1}\frac {(-1)^n}{(2n+1)!} \left(\frac x2\right)^{2n}
\]
where $x=e(U,\mathbf{H}_\mathrm{W}\QQ)$ is the Euler class of the universal rank $2$ bundle $U$ on $\HP^{\infty}$.
\end{prop}

\begin{rem}
The power series in the expression of the inverse of the Todd class of $u$ is in fact the power series $\frac{\sin (\frac x2)}{\frac x2}$.
\end{rem}

If now $(E,\varphi)$ is a symplectic bundle of rank $2n$ over a smooth scheme $X$, we may now use the splitting principle to suppose (up to passing to a smooth scheme $Y/X$) that $E$ splits as a sum of rank $2$ symplectic bundles, the so-called Borel roots $\alpha_1,\ldots,\alpha_n$ of $E$. The Todd class of $E$ is then of the form
\[
\td_{p'\circ\mathrm{bo}_t}(E)=\prod_{i=1}^n\frac{\frac {\alpha_i}2}{\sin (\frac {\alpha_i}2)}
\]
where we have somewhat abusively denoted by $\alpha_i$ the first Borel class of the respective Borel roots. For the Todd class $\td_{\mathrm{ch}_t}$, we refer to \cite[Example 3.2.4]{Fulton98}. 

\subsection{The quadratic Hirzebruch-Riemann-Roch theorem}

\begin{num}
In this section, we illustrate Theorem \ref{thm:GRR} in the case of the Borel character discussed in the previous section.

We consider a smooth proper scheme $X$ over a field $k$ of even dimension $d=2n$,
 with projection $p\colon X\to \Spec(k)$ and canonical bundle $\omega_{X/k}=\det(\Omega_{X/k})$.
 We let $f\colon L^{\otimes 2}\to \omega_{X/k}$ be an orientation of $X/k$
 (see \Cref{ex:orientations}).
 Using this isomorphism, we obtain an isomorphism (corresponding to the fact $\GW$ is $\SL^c$-oriented)
\[
\GW^{2}(X)\simeq \GW^2(X,L^{\otimes 2})\simeq \GW^{2}(X,\omega_{X/k})
\]
mapping the (isometry class of the) symplectic bundle $(V,\varphi)$ to $(V\otimes L,f\circ(\varphi\otimes \mathrm{Id}))$. The same formula applies for $\GW^{0}(X)$.

The isomorphism $f\circ(\varphi\otimes \mathrm{Id})$ induces for any $i\in \NN$ an isomorphism
\[
[f\circ(\varphi\otimes \mathrm{Id})]_i\colon\H^i(X,V\otimes L)\to \H^i(X,(V\otimes L)^\vee\otimes\omega_{X/k}).
\]
Serre duality yields a canonical isomorphism $\mathrm{can}:\H^i(X,(V\otimes L)^\vee\otimes\omega_{X/k})\simeq \H^{2n-i}(X,V\otimes L)^*$ and we finally obtain isomorphisms
\[
\varphi_i\colon \H^i(X,V\otimes L)\to \H^{2n-i}(X,V\otimes L)^*.
\]
In particular, for $i=n$, the isomorphism 
\[
\varphi_n\colon\H^n(X,V\otimes L)\to \H^{n}(X,V\otimes L)^*
\]
yields either a symmetric or symplectic form on the vector space $\H^n(X,V\otimes L)$ in the following way: If $[V,\varphi]\in \GW^{2m}(X)$, then $[\H^n(X,V\otimes L),\varphi_n]\in \GW^{2m-2n}(k)$. For instance, if $\varphi$ is symplectic and $n$ is odd then $\varphi_n$ is symmetric. The following proposition can be seen as a computation of the tranfer map in Grothendieck-Witt theory.
\end{num}
\begin{prop}
Let $[V,\varphi]\in \GW^{2m}(X)$. Then 
\[
p_*([V,\varphi])=(-1)^{m+n}[\H^n(X,V\otimes L),\varphi_n]+(\sum_{i=0}^{n-1}(-1)^{(m+i)}\mathrm{dim}_k\H^i(X,V\otimes L))\cdot (1-\epsilon)
\]
in $\GW^{2m-2n}(k)$.
\end{prop}

\begin{rem}
The bundle $V$ should be seen as a complex concentrated in degree $m$ in the statement. Moreover, one observes that if $m-n$ is odd then $\varphi_n$ is symplectic, and then hyperbolic. Thus, the interesting information arises only when $m$ and $n$ have the same parity, and is then encoded by $\varphi_n$. 
\end{rem}

In view of this computation, we may set the following definition.

\begin{df}\label{df:quad-Euler-char}
Let $X$ be a smooth proper scheme of even dimension $d=2n$ and let $(V,\varphi)$ be either a symplectic bundle if $d=2\pmod 4$ or a symmetric bundle if $d=0\pmod 4$. We define the quadratic Euler characteristic of $(V,\varphi)$ under the orientation $f:L^{\otimes2}\simeq \omega_{X/k}$ as the following expression
$$
[\H^n(X,V\otimes L),\varphi_n]+\sum_{i=0}^{n-1}(-1)^{(m+i)}\mathrm{dim}_k\H^i(X,V\otimes L)\cdot (1-\epsilon)
$$
where $m=1$ if $\varphi$ is symplectic and $m=0$ else.  We denote by $\tilde \chi(X,f;V,\varphi)$ the class of this symmetric bilinear form.
\end{df}

The Grothendieck-Riemann-Roch formula for the Borel character $\bo_t$ then reads as follows.
\begin{thm}\label{thm:HRR}
Let $X$ be a smooth proper scheme of even dimension $d$ such that its tangent bundle $T_{X/k}$ admits a stable $\Sp$-orientation $\tilde\tau_X$. We let $f$ be the orientation of $X/k$ associated with $\tilde\tau_X$.

Let $(V,\varphi)$ be either a symplectic bundle if $d=2\pmod 4$ or a symmetric bundle if $d=0\pmod 4$. Then the following equality holds in $\GW(k)$
$$
\tilde \chi(X,f;V,\varphi)=\tdeg_{\tilde \tau_X}\left(\td_{\bo_t}(\tilde \tau_X).\bo_t(V,\varphi)\right)
$$
where $\tdeg_{\tilde \tau_X}\colon\CHt^{d}(X)\to \GW(k)$ is the quadratic degree map
 (pushforwar) associated with the symplectic orientation $\tilde \tau_X$ of $X/K$.
\end{thm}

\begin{ex}\label{ex:HRR-K3}
Suppose that $d=2$ and that $X$ is a K3 surface. Then, a symplectic orientation $\tilde\tau_X$ of $\Omega_{X/k}$ (or $T_{X/k}$) is just obtained by choosing an isomorphism $f:\OO_X\to\omega_{X/k}$. Indeed, one just considers the composite $\Omega_{X/k}\otimes \Omega_{X/k}\to \omega_{X/k}\xrightarrow{f^{-1}} \OO_X$. The Todd class can be easily computed using \Cref{prop:Toddforgetful} and \Cref{prop:Toddquadratic}. We have
\[
\td_{\mathrm{ch}_t}^{-1}(\Omega_{X/k})=1-\frac{c_2(\Omega_{X/k})}{12}, \td_{p'\circ\mathrm{ch}_t}^{-1}(\Omega_{X/k})=1
\]
and it follows easily from the fact that $X$ is a surface that $\td_{\bo_t}(\Omega_{X/k})=1+\frac{c_2(\Omega_{X/k})}{24}(1-\epsilon)$. Note further that there is no need to tensor with $\QQ$ as the Borel character is defined integrally in case of surfaces. If $(V,\varphi)$ is a symplectic bundle of rank $2r$ over $X$, we finally obtain
\begin{eqnarray*}
\tilde \chi(X,f;V,\varphi) &= &\tdeg_{\tilde \tau_X}\left((2r+e(V,\varphi))\cdot (1+\frac{c_2(\Omega_{X/k})}{24}(1-\epsilon))\right) \\
 & = & \tdeg_{\tilde \tau_X}\left(r\cdot \frac{c_2(\Omega_{X/k})}{12}(1-\epsilon)+e(V,\varphi)\right) \\
 & = & \mathrm{deg}\left(r\cdot \frac{c_2(\Omega_{X/k})}{12}\right)\cdot (1-\epsilon)+ \tdeg_{\tilde \tau_X}(e(V,\varphi)).
\end{eqnarray*}
Using the fact that $\mathrm{deg}(c_2(\Omega_{X/k}))=24$ (e.g. \cite{Huybrechts16}), we finally obtain the formula:
\[
\tilde \chi(X,f;V,\varphi)=2r(1-\epsilon)+\tdeg_{\tilde \tau_X}(e(V,\varphi)).
\]
\end{ex}

\bibliographystyle{amsalpha}
\bibliography{QRR}

\end{document}